\documentclass[11.5pt,regno]{amsart}

\usepackage{amssymb}
\usepackage{amsfonts}
\usepackage{amsthm}
\usepackage{amsmath}
\usepackage{bbm}
\usepackage{bm}
\usepackage{mathtools}
\usepackage{eufrak}
\usepackage{enumerate}

\numberwithin{equation}{section}
\allowdisplaybreaks

\newtheorem{theorem}{Theorem}[section]
\newtheorem{proposition}[theorem]{Proposition}
\newtheorem{prop}[theorem]{Proposition}
\newtheorem{lemma}[theorem]{Lemma}
\newtheorem{corollary}[theorem]{Corollary}

\theoremstyle{definition}
\newtheorem{definition}[theorem]{Definition}
\newtheorem{notation}[theorem]{Notation}
\theoremstyle{remark}
\newtheorem{remark}[theorem]{Remark}
\theoremstyle{condition}
\newtheorem{condition}[theorem]{Condition}

\newenvironment{claim}[1]{\par\noindent\underline{Claim:}\space#1}{}
\newenvironment{claimproof}[1]{\par\noindent\underline{Proof:}\space#1}{\hfill $\blacksquare$}

\newcommand{\R}{\mathbb{R}}

\newcommand{\N}{\mathbb{N}}
\newcommand{\Z}{\mathbb{Z}}
\newcommand{\C}{\mathbb{C}}

\newcommand{\x}{\mathbbm{x}}
\newcommand{\y}{\mathbbm{y}}

\newcommand{\sgn}{\text{sgn}}

\newcommand{\scriptA}{\mathcal{A}}

\newcommand{\scriptD}{\mathcal{D}}

\newcommand{\scriptH}{\mathcal{H}}

\newcommand{\scriptN}{\mathcal{N}}
\newcommand{\scriptO}{\mathcal{O}}

\newcommand{\scriptR}{\mathcal{R}}
\newcommand{\scriptS}{\mathcal{S}}
\newcommand{\scriptT}{\mathcal{T}}

\newcommand{\scriptW}{\mathcal{W}}

\newcommand{\qtq}[1]{\quad\text{#1}\quad}

\DeclareMathOperator*{\dist}{dist}

\DeclarePairedDelimiter{\norm}{\lVert}{\rVert}

\def\Xint#1{\mathchoice
{\XXint\displaystyle\textstyle{#1}}%
{\XXint\textstyle\scriptstyle{#1}}%
{\XXint\scriptstyle\scriptscriptstyle{#1}}%
{\XXint\scriptscriptstyle\scriptscriptstyle{#1}}%
\!\int}
\def\XXint#1#2#3{{\setbox0=\hbox{$#1{#2#3}{\int}$ }
\vcenter{\hbox{$#2#3$ }}\kern-.6\wd0}}

\def\dashint{\Xint-}

\begin{document}
\title[Euclidean Averages Along Prototypical Hypersurfaces in $\R^3$]{Near Optimal $L^p\rightarrow L^q$ Estimates for Euclidean Averages Over Prototypical Hypersurfaces in $\mathbb{R}^3$}
\author{Jeremy Schwend}
\address{Department of Mathematics, University of Georgia, Athens, GA 30602}
\email{Jeremy.Schwend@uga.edu}
\begin{abstract}
We find the precise range of $(p,q)$ for which local averages along graphs of a class of two-variable polynomials in $\R^3$ are of restricted weak type $(p,q)$, given the hypersurfaces have Euclidean surface measure. We derive these results using non-oscillatory, geometric methods, for a model class of polynomials bearing a strong connection to the general real-analytic case.
\end{abstract}

\maketitle


\section{Introduction}\label{S:Introduction}

In this article we establish the sharp range of restricted weak type $(p,q)$ inequalities for averaging operators of the form
\begin{equation}\label{intro}
\scriptT f(x):=\int_{[-1,1]^2}f(x'-t,x_3-\varphi(t))dt, \qquad x=:(x',x_3)\in \R^3,
\end{equation}
when $\varphi$ is a mixed homogeneous polynomial, satisfying $\varphi(\sigma^{\kappa_1}t_1,\sigma^{\kappa_2}t_2)=\sigma \varphi(t)$ for all $\sigma>0$ and some fixed $\kappa_1, \kappa_2 \in (0,\infty)$. When $dt$ is replaced by the affine surface measure, $|det D^2\varphi(t)|^\frac 14 dt$, the range of $p$ and $q$ is known \cite{Obe,Gre}, with a range that is independent of $\varphi$, and those result extends to hypersurfaces in $\R^n$ for every $n$. (See Theorem \ref{Gressman}.) When the Euclidean surface measure is used instead, as in \eqref{intro}, numerous subtleties are introduced.


 In the case of a hypersurface with the Euclidean surface measure, only partial results exist for hypersurfaces in $\R^n$, when $n$ is greater than $2$. For the Euclidean surface measure, Ferreyra, Godoy, and Urciuolo completed the homogeneous polynomial case for $n=3$ \cite{FGU2}\cite{Urc}, and  the additive case $\varphi(t)=\sum |t_i|^{a_i}$ for every $n$ \cite{FGU1}. In addition, for $n=3$, Iosevich, Sawyer, and Seeger's work \cite{ISS} includes results for convex surfaces, based on multi-type information. More recently, Dendrinos and Zimmerman \cite{DZ} investigated this question when $n=3$ in the case that $\varphi$ is a mixed homogeneous polynomial of two variables. Such polynomials form a natural model class, because their members contain all anisotropic scaling limits of polynomials of two variables, and because mixed-homogeneous polynomials are related to the faces of the Newton diagram for all real-analytic functions. Thus, techniques for these polynomials should give insight on how to proceed in a more general case. In many cases, the results in \cite{DZ} were nearly optimal, missing only the boundary, while in other cases, like $\varphi(t)=t_1^4+t_1^2t_2+\frac 16t_2^2$, there remained an additional region between the conjecture and the theorem.
\par
In this article we complete the mixed-homogeneous picture, obtaining the full range of $L^p\rightarrow L^q$ estimates (albeit restricted weak type at the endpoints). We use an alternate approach to \cite{DZ}, which allows us to obtain the full range of restricted weak type estimates for all such polynomials and which provides some clarification of the relationship between the Euclidean versions of these results and the affine versions. More precisely, we completely avoid using oscillatory methods, instead using methods such as Christ's method of refinements in \cite{Chr1}, often complimented by additional techniques such as orthogonality arguments as in \cite{CDSS}. In several places, we modify the method of refinements to find the influence of lower-dimensional curvature information (in the simpler cases, this is instead accomplished with Minkowski's and Young's inequalities). Finally, for the cases where regions remain between the Dendrinos-Zimmermann \cite{DZ} results and conjecture, we alter the method of refinement to exploit the geometric information of the surface.

In Sections \ref{S:Introduction} and \ref{S:Refomulation}, we present two versions of our main theorem, the former written fully in terms of the Newton polytope, and the latter written more in terms of the factorization of $\varphi$, with notation similar to the results in \cite{DZ}. Then Section \ref{S:Overview} and \ref{S:Outline} will give further detail into our techniques, along with an outline of our proof. One additional key purpose of the techniques here will be to clarify the complex relationship between the affine measure results of Oberlin \cite{Obe} and Gressman \cite{Gre}, restated in Theorem \ref{Gressman}, and the Euclidean or push-forward measure results proven here.
\par
 One hope is that the relationship between mixed-homogeneous polynomials and general real analytic functions will allow for these methods to be greatly generalized. Additionally, some of the concepts demonstrated here can be applied to analyzing the effects of curvature in a much broader class of problems. For example, in \cite{Sch}, which is a joint work of the author and Stovall, these concepts are applied to analyzing the effects of curvature in Fourier restriction.


\vspace{1pc}

In this article, we consider the averaging operator
$$
\scriptT f(x)=\int_{[-1,1]^2}f(x'-t,x_3-\varphi(t))dt \quad \quad \text{where } x=(x',x_3)\in \R^3.
$$
where $\varphi$ are polynomials satisfying the following definition:
\begin{definition} A polynomial $\varphi:\R^2\rightarrow \R$ is mixed homogeneous if there exists $\kappa_1,\kappa_2\in (0,\infty)$ such that $\varphi(\sigma^{\kappa_1}t_1,\sigma^{\kappa_2}t_2)=\sigma \varphi(t)$ for all $\sigma>0$.
\end{definition}
To state our main theorem, we need some additional definitions. 
\begin{definition}\label{Newton definitions}: Let $\varphi(z_1,z_2)=\sum_{\alpha_1,\alpha_2=0}^{M}c_{\alpha_1,\alpha_2}z_1^{\alpha_1}z_2^{\alpha_2}$ be a mixed homogeneous polynomial mapping of $\R^2$ into $\R$. We will denote the Taylor support of $\varphi$ at $(0,0)$, the Newton polyhedron of $\varphi$, and the Newton distance, respectively, to be 
\begin{align*}
\scriptS(\varphi)&:=\{(\alpha_1,\alpha_2)\in \N_0^2: c_{\alpha_1,\alpha_2}\neq 0\}; \\
\scriptN(\varphi)&:=\overline{\text{Conv}}(\{(\alpha_1,\alpha_2)+\R_+^2:(\alpha_1,\alpha_2)\in \scriptS(\varphi)\}); \\
d(\varphi)&:=\inf\{c: (c,c)\in \scriptN(\varphi)\};
\end{align*}
where $\overline{\text{Conv}}$ denotes the closed convex hull, and where $(d(\varphi),d(\varphi))$ is the intersection of the bisetrix $x_1=x_2$ with the Newton polyhedron $\scriptN(\varphi)$.

\vspace{.75pc}

Next, define $\varphi_{R1}=\varphi-\{\text{all terms for which } \alpha_1=0\}$, and denote the $z_1$-reduced Taylor support of $\varphi$ as 
$$
\scriptS(\varphi_{R1})=\{(\alpha_1,\alpha_2)\in \scriptS(\varphi): \alpha_1\neq 0\},
$$
which removes all $z_2$-axis terms from the Taylor support of $\varphi$. Under this definition, the Newton polytope $\scriptN(\varphi_{R1})$ of $\varphi_{R1}$ is a simple translate of the Newton polytope of $\partial_{z_1}\varphi$. Let $\scriptS(\varphi_{R2})$ be defined similarly, and denote the reduced Newton distance as $d(\varphi_R)=\max(d(\varphi_{R1}),d(\varphi_{R2}))$, with $\varphi_R$ defined accordingly. If $d(\varphi_{R1})=d(\varphi_{R2})$, then set $\varphi_R$ to be $\varphi_{R1}$.

\vspace{.75pc}

A mixed-homogeneous function $\varphi$ will be said to be \textit{linearly adapted} if every factor of the form $(z_2-\lambda z_1)$, some $\lambda\in \R/\{0\}$, has multiplicity less that or equal to $d(\varphi)$. By this definition, if $\varphi$ is mixed homogeneous but not homogeneous, then $\varphi$ is already linearly adapted. This ends up being equivalent to the definition given in \cite{IM2}. For brevity the proof of this equivalence is omitted.

\vspace{.75pc}

Finally, let $o(\varphi)$ denote the maximal multiplicity of any real irreducible (over $\C[z_1,z_2]$) factor of $\varphi$, and denote the height of $\varphi$ by $h(\varphi)=\max(d(\varphi),o(\varphi))$. 

\end{definition}

\begin{theorem}\label{Main Theorem Original}

Let $\varphi(z_1,z_2):\R^2\rightarrow \R$ be a mixed homogeneous polynomial, with $\varphi(0)=0$, $\nabla\varphi(0)=0$. If $\varphi$ is homogeneous, we additionally assume that $\varphi$ is linearly adapted and that $\varphi(z_1,z_2)\neq C(\lambda_1z_1+\lambda_2z_2)^J$, for any $C\in \R$, $\lambda_i\in \C$, $J\in \N$. Then $\scriptT$ is bounded from $L^p(\R^3)$ to $ L^q(\R^3)$ in the restricted-weak sense iff $p,q$ satisfy each of the following conditions:
\begin{align}
\label{E:Main 1}\tfrac 1q &\leq \tfrac 1p \\
\label{E:Main 2}\tfrac 1q &\geq \tfrac 1{3p}  \quad \quad \quad \quad \quad \quad \quad \quad \quad \quad \quad \tfrac 1q \geq \tfrac 3p-2 \\
\label{E:Main 3}\tfrac 1q &\geq \tfrac 1p-\tfrac 1{d(\varphi)+1} \\
\label{E:Main 4}\tfrac 1q &\geq \tfrac{d(\varphi_R)+1}{2d(\varphi_R)+1}\tfrac 1p-\tfrac 1{2d(\varphi_R)+1} \quad \quad \quad
\tfrac 1q\geq \tfrac{2d(\varphi_R)+1}{d(\varphi_R)+1}\tfrac 1p-1  \\
\label{E:Main 5}\tfrac 1q &\geq \tfrac{h(\varphi)+1}{h(\varphi)+2}\tfrac 1p-\tfrac 1{h(\varphi)+2} \quad \quad \quad \quad \quad
\tfrac 1q \geq \tfrac{h(\varphi)+2}{h(\varphi)+1}\tfrac 1p-\tfrac{2}{h(\varphi)+1} \\
\label{E:Main 6}\tfrac 1q &\geq \tfrac 1p-\tfrac 1{h(\varphi)}.
\end{align}
\begin{remark}  Since non-degenerate linear transformations  preserve $L^p\rightarrow L^q$ norms, up to a constant, the assumptions that $\varphi$ is linearly adapted and that $\varphi(0)=0$, $\nabla\varphi(0)=0$ do not weaken the generality of the theorem.
\end{remark}

\begin{remark} Equations \eqref{E:Main 4}, \eqref{E:Main 5}, and \eqref{E:Main 6} only become relevant if $d(\varphi_R)>2d(\varphi)$, $h(\varphi)>d(\varphi)+\frac 12$, or $h(\varphi)>d(\varphi)+1$, respectively.  
\end{remark}


\begin{remark}
The two cases not covered by our theorem are when, after a linear transformation, $\varphi\equiv 0$ or $\varphi=z_1^J$, $J\geq 2$. In the former case, it is well known that $\scriptT$ is bounded if and only if $p=q$. In the latter case, it was proved in \cite{FGU2} that $\scriptT$ maps $L^p$ boundedly into $L^q$ precisely when $p,q$ satisfy all of the following:
$$
\tfrac 1q\leq \tfrac 1p \quad \quad \quad \tfrac 1q\geq \tfrac 1{2p} \quad \quad \quad \tfrac 1q\geq \tfrac 2p-1 \quad \quad \quad \tfrac 1q\geq \tfrac 1p-\tfrac 1{J+1}.
$$

This can also be shown with a proof similar to that in Subsection \ref{SS:Second Relevant Vertex 1}, namely, by interpolating $L^{\frac 32}\rightarrow L^3$ and 
$L^\infty\rightarrow L^\infty$ bounds after a dyadic decomposition. For brevity this proof is omitted.
\end{remark}

\end{theorem}

\textbf{Acknowledgements}: The author would like to immensely thank his advisor Betsy Stovall for proposing this problem and for her advice, help, and insight throughout this project. The author would also like to thank S. Dendrinos for sending me an early preprint of \cite{DZ}, and D. M\"{u}ller for his suggestion to rewrite the results in terms of the Newton polytope. This work was supported in part by NSF DMS-1600458, NSF DMS-1653264, and NSF DMS-1147523.

\section{Reformulation}\label{S:Refomulation}

In this section, we will state (Theorem \ref{Main Theorem Alt}) an alternate formulation of Theorem \ref{Main Theorem Original}, and prove in Proposition \ref{Main Theorem Equivalence} that these theorems are equivalent. By way of background, we restate a result of Gressman, Theorem 3 of \cite{Gre}. For $\varphi: \R^{d-1}\rightarrow \R$, and $x\in \R^d$, define the affine averaging operator $\scriptA$ as 
$$
\scriptA f(x):=\int_{\R^{d-1}} f(x'-t,x_d-\varphi(t))|\det(D^2\varphi)|^{\frac 1{d+1}}dt,
$$
where $D^2\varphi$ is the Hessian. 
\begin{theorem}[\cite{Gre}]\label{Gressman}
Let $\varphi: \R^{d-1}\rightarrow \R$ be a polynomial. Then $\scriptA$ extends as a bounded linear operator from $L^\frac{d+1}{d}(\R^d)$ to $L^{d+1}(\R^d)$, with operator norm bounded by a constant depending only on the dimension $d$ and the degree of the polynomial $\varphi$.
\end{theorem}
In fact, the restricted strong type version of this theorem, due to D. Oberlin (\cite{Obe}), suffices for most of our applications. We will use this result extensively in proving Theorem \ref{Main Theorem Original}.

 \begin{definition} Let $\bm{\kappa}=(\kappa_1,\kappa_2)\in (0,\infty)^2$, and let $\psi:\R^2\rightarrow \R$. We say $\psi$ is \textit{$\bm{\kappa}$-mixed homogeneous} if, for every $\sigma>0$, $\psi(\sigma^{\kappa_1}\cdot,\sigma^{\kappa_2}\cdot)=\sigma\psi(\cdot,\cdot)$. 
 \end{definition}

If $\varphi$ is mixed-homogeneous, there exists a $ \bm{\kappa}=(\kappa_1,\kappa_2)\in (0,\infty)^2$ such that $\varphi$ is $\bm{\kappa}$-mixed homogeneous. If $\varphi$ is a not a monomial, this $\bm{\kappa}$ is unique, while for monomials, uniqueness holds under the additional assumption $\kappa_1=\kappa_2$. We denote this unique $\bm{\kappa}$ as $\bm{\kappa_\varphi}$. We define homogeneous distance as
$$
d_h:=d_h(\varphi):=\tfrac{1}{\kappa_1+\kappa_2}
$$
Further, there exist unique $r, s, m$ satisfying $\text{gcd}\{r,s\}=1$ such that $\bm{\kappa_\varphi}=(\frac sm,\frac rm)$.

\vspace{.75pc} 

The mixed-homogeneous polynomial $\varphi$ may be expanded as $\varphi=\sum_{l=0}^{l_f}c_lz_1^{J-rl}z_2^{K+sl}$, for some $J, K, c_l$. We observe that
$$
d_h(\varphi)=\tfrac{Js+Kr}{r+s}.
$$
By the Fundamental Theorem of Algebra, we can factor $\varphi$ as
\begin{equation}\label{mixedhomogeneouspolynomial}
\varphi(z_1,z_2)=Cz_1^{\tilde{\nu}_1}z_2^{\tilde{\nu}_2} \prod_{j=3}^{\tilde m_2} (z_2^s-\lambda_jz_1^r)^{\tilde{n}_j}, \text{ with }\lambda_j \text{ real iff } j\leq \tilde m_1\leq \tilde m_2.
\end{equation}
For brevity, we refer to the irreducible factors $z_1, z_2, z_2^s-\lambda_j z_1^r$, simply as factors. The homogeneous distance $d_h$ can be rewritten as
$$
d_h=\tfrac{s\tilde{\nu}_1+r\tilde{\nu}_2 +rs\sum \tilde{n}_j}{r+s}.
$$

Consider $\omega:=\det(D^2\varphi)$. As we will show in Lemma \ref{domega dh} and \eqref{domega:2nd}, if $d_h\neq 1$, $\omega$ is $\frac{\bm{\kappa_\varphi}}{2-2/d_h}$-mixed homogeneous, so $\omega$ is mixed homogeneous and shares the same $r$ and $s$ with $\varphi$, and its mixed homogeneity satisfies
$$
d_\omega:=d_h(\omega)=2d_h-2.
$$
If $d_h=1$, $\omega$ is constant (see \eqref{domega:1st}), and we can say $d_\omega=0$. Thus, in either case, we can factor $\omega$ as 
\begin{equation}\label{omega:1 rep}
\omega(z_1, z_2)=Cz_1^{\nu_1}z_2^{\nu_2} \prod_{j=3}^{m_2} (z_2^s-\lambda_jz_1^r)^{n_j},\text{ with } \lambda_j \text{ real iff } j\leq m_1\leq m_2.
\end{equation}

We will use the following proposition from \cite{IM}, which follows from simple algebra:

\begin{prop}[\cite{IM}] \label{dpsi}
 Let $\psi$ be a mixed homogeneous polynomial as in \eqref{mixedhomogeneouspolynomial} or \eqref{omega:1 rep}, and denote $d_\psi$ as the homogeneous distance of $\psi$. 

\begin{enumerate}

\item If $\min\{r,s\}> 1$ then any irreducible factor of $\psi$ the form $(z_2^s-\lambda z_1^r)$, $\lambda\neq 0$, has multiplicity strictly less than $d_\psi$.

\item If $\psi$ has a factor of multiplicity greater than $d_\psi$, then every other factor of $\psi$ has multiplicity less than $d_\psi$.

\end{enumerate}
\end{prop}

Let $T$ denote the maximum multiplicity of any real factor of $\omega$, i.e. $T=\max\{\{\nu_1,\nu_2\}\cup\{n_j:j\leq m_1\}\}$. As we will see in Proposition \ref{omega=0}, our assumptions imply that $\omega\not\equiv 0$, so $T$ is well-defined. 

Suppose $T>d_\omega$. By Proposition \ref{dpsi}, we can define
\begin{equation}\label{f_T}
f_T=\text{The unique factor of $\omega$, with multiplicity $T$, when $T>d_\omega$}.
\end{equation}
If $f_T$ is linear, let $\nu$ denote its multiplicity in $\varphi$, and if $f_T$ is nonlinear, let $N$ denote its multiplicity in $\varphi$. If $f_T$ is linear and $\nu=0$, interchanging indices allows for the assumption $f_T(z)\neq z_1$, so we may expand 
$$
\varphi(z)=z_1^J+z_1^{J-lr}f_T^{ls}+\scriptO(f_T^{ls+s}), \text{ some } l\geq 1,
$$ 
and in any such case where $T>d_\omega$, $f_T$ is linear, and $\nu=0$, we set $A:=ls$.

If we are in any case such that $\nu$ is not defined in the preceding (likewise $A$, $N$) we set it to be $0$. We use the convention that $\frac 10=\infty$.

\begin{theorem}\label{Main Theorem Alt} Let $\varphi(z_1,z_2)$ be a mixed homogeneous polynomial, with vanishing gradient at the origin,  and not of the form $(\lambda_1z_1+\lambda_2z_2)^J$, $\lambda_i\in \C$, $J\in \N$. Then for $1\leq p,q \leq \infty$, $\mathcal{T}$ is of restricted-weak type $(p,q)$ iff the following hold: 
\begin{align}
\label{E:Alt 3}\tfrac{1}{q}&\leq \tfrac{1}{p} \quad \quad \tfrac{1}{q}\geq \tfrac{1}{3p} \quad \quad \tfrac{1}{q}\geq \tfrac{3}{p}-2
\\
\label{E:Alt 4}\tfrac{1}{q}&\geq \tfrac{1}{p}-\tfrac{1}{d_h+1} 
\\
\label{E:Alt 5}\tfrac{1}{q}&\geq \tfrac{1}{p}-\tfrac{1}{\nu+1}
\\
\label{E:Alt 6}\tfrac{1}{q}&\geq  \tfrac{A+1}{2A+1}\tfrac{1}{p}-\tfrac{1}{2A+1} \quad \quad \quad \tfrac{1}{q}\geq \tfrac{2A+1}{A+1}\tfrac{1}{p}-1 
\\
\label{E:Alt 7}\tfrac{1}{q}&\geq \tfrac{N+1}{N+2}\tfrac{1}{p}-\tfrac{1}{N+2} \quad \quad \quad \tfrac{1}{q}\geq \tfrac{N+2}{N+1}\tfrac{1}{p}-\tfrac{2}{N+1} 
\\
\label{E:Alt 8}\tfrac{1}{q}&\geq \tfrac{1}{p}-\tfrac{1}{N} 
\end{align}























 


\end{theorem}

\begin{proposition}\label{Main Theorem Equivalence}
Theorems \ref{Main Theorem Original} and \ref{Main Theorem Alt} are equivalent.
\end{proposition}
Proposition \ref{Main Theorem Equivalence} will be proved in Proposition \ref{Main Theorem Equivalence 2}.

\section{Overview}\label{S:Overview}

Our strategy will be as follows: we will use the mixed homogeneity of $\varphi$ (and by implication that of $\omega:=\text{det}(D^2\varphi)$, which is comparable to the Gaussian curvature) to decompose $[-1,1]^2$ into dyadic level-set ``strips" on which $|\omega|\approx 2^{-j}$, and further decompose those strips into dyadic rectangles (after a possible coordinate change). On any such strip or rectangle, we will have three useful restricted weak type $L^p \rightarrow L^q$ estimates: For $p=q$, Young's inequality implies $L^p\rightarrow L^q$ bounds with constant equal to the measure of the strip or rectangle. Since $\omega$ is nearly constant on each strip or rectangle, Theorem \ref{Gressman} gives us $L^\frac 43\rightarrow L^4$ bounds on each dyadic strip. And finally, at $(p,q)=(\frac 32,3)$, we will use bounds for averages on curves in $\R^2$ (\cite{Lit} and later \cite{Gre}), and either a variant of Christ's method of refinements or, in simpler cases, Minkowski's Inequality with Young's Inequality, to achieve either strong-type or restricted-weak-type versions of $L^\frac 32\rightarrow L^3$ bounds. In some cases, by summing over the dyadic regions, summing over the minimum of the $L^\frac 43\rightarrow L^4$, the $L^\infty\rightarrow L^\infty$, and the $L^\frac 32 \rightarrow L^3$ bounds, we can achieve restricted-weak-type bounds for $\scriptT$ in the optimal range. In some cases, we also utilize the scaling symmetry of mixed-homogeneous polynomials in this argument, as demonstrated in Proposition \ref{scaling prop}. However, these methods are not always sufficient. Two types of obstructions can arise.
\par 
One issue is that degeneracies sometimes lead to a logarithmic factor in the $L^p\rightarrow L^p$ bounds on the relevant strips. To remove the logarithmic factor, we use an orthogonality argument incorporated into Christ's method of refinements, found in \cite{CDSS}. 
\par 
A more delicate issue arises when $\omega=\text{det}(D^2\varphi)$ possesses a factor $f_T$ of high multiplicity $T>d_\omega$, but all components of $\nabla \varphi$ lack $f_T$ as a factor. This causes the portion of the surface over our dyadic ``rectangles" to have a structure more closely resembling a helix than a parabola or hyperbola. As a consequence, the image of any such dyadic ``rectangle" under $\nabla \varphi$ is highly non-convex. We can beat the above mentioned interpolation argument by quantifying this non-convexity. We use a variant of Christ's method of refinements \cite{Chr1}, to show that non-convexity makes it impossible for $L^\infty\rightarrow L^\infty$ quasi-extremals to also be $L^\frac 43\rightarrow L^4$ quasi-extremals, and we interpolate after quantifying this trade-off. Further detail will be provided in the next section.

\section{Outline}\label{S:Outline}

We first decompose into the cases $T>d_\omega$ and 

\hspace{4.5pc} $\left.\begin{tabular}{||c||} 
 \hline
Case: $T\leq d_\omega$  \\ 
 \hline

\end{tabular}\right\}$

and we further break the case $T>d_\omega$ into the following cases:

\hspace{4.5pc} $\left.\begin{tabular}{||c c||} 
 \hline
Case($\nu$): & $f_T$ linear and $\nu\geq 1$  \\ 
 \hline
Case(A): & $f_T$ linear and $\nu=0$, $A\geq 2$ \\
 \hline
Case(N): & $f_T$ nonlinear and $N\geq 2$ \\
 \hline
\end{tabular}\right\}$ 
Rectangular Cases \\
\indent \hspace{4.5pc} $\left.\begin{tabular}{||c c||}
 \hline
Case(i): & $f_T$ linear and $\nu=0$, $A=1$ \\
 \hline
Case(iia): & $f_T$ nonlinear and $N=0$ \\
 \hline
Case(iib): & $f_T$ nonlinear and $N=1$ \\
 
\hline
\end{tabular}\right\}$ 
Twisted Cases \\

where $f_T$ was defined in \eqref{f_T} to be the unique factor of $\omega=det D\varphi$ having multiplicity $T$ when $T>d_\omega$. When $T>d_\omega$ and $f_T$ is linear with $\nu=0$, $A$ cannot be $0$, so this decomposition covers all cases where $T>d_\omega$. As we will see in Corollary \ref{multiplicities connection}, in cases $(\nu)$, $(A)$, and $(N)$, $T$ satisfies $T=2\nu-2$, $T=A-2$, and $T=2N-3$, respectively. In these cases, on a local scale, the graph of $\varphi$ will resemble a rectangular piece of a paraboloid or a saddle. We classify these three cases as Rectangular Cases. In cases $(i)$, $(iia)$, and $(iib)$, on the other hand, $T$ has no such restrictions, and on a local scale the graph of $\varphi$ has a helix-like structure. These cases will be denoted as Twisted Cases.


In Section \ref{S:Algebraic Lemmas}, we prove algebraic lemmas relating $\varphi$ and its derivatives, show that the hypotheses of Theorem \ref{Main Theorem Alt} imply that $\omega\not\equiv0$, show that if $\varphi$ is homogeneous and $T>d_\omega$, it can only belong to case($\nu$), and prove that the two versions of our main theorem, Theorems \ref{Main Theorem Original} and \ref{Main Theorem Alt}, are equivalent. Section \ref{S:Necessary Conditions} is dedicated to proving that every condition stated in Theorem \ref{Main Theorem Alt} is necessary. In Section 8 we exploit the mixed-homogeneous scaling symmetry of $\varphi$ to shorten the interpolation arguments in numerous later sections. 

In Section \ref{S:Decomposition}, we make a finite decomposition of the region $[-1,1]^2$ into neighborhoods of each irreducible subvariety of $\{\omega:=det D^2\varphi=0\}$. This induces a finite decomposition of the operator. The most troublesome of these neighborhoods, which we denote by $R_T$, is the neighborhood of the subvariety $\{f_T=0\}$. When $T\leq d_\omega$, we set $R_T=\emptyset$.

We begin Section \ref{S:First Relevant Vertex} by showing in Lemmas \ref{relevant vertices} and \ref{Vertex1location} that to prove Theorem \ref{Main Theorem Alt}, it suffices to prove that $\scriptT$ is of rwt $(p_{v_1},q_{v_1})$, for a specific $(p_{v_1},q_{v_1})$ satisfying $q_{v_1}=3p_{v_1}$, and in cases (N) and (A), to additionally prove that $\scriptT$ is of rwt $(p_{v_2},q_{v_2})$, for a specific $(p_{v_2},q_{v_2})$ satisfying $p_{v_2}'\leq q_{v_2}<3p_{v_2}$.

And to decompose $\scriptT$, we define $\scriptT_R f(x):=\int_R f(x'-t,x_3-\varphi(t))dt.$

Under these decompositions, our proof of Theorem \ref{Main Theorem Alt} reduces to proving each of the following:

\begin{enumerate}[\hspace{1pc} (I)]
\item $\scriptT_{[-1,1]^2\backslash R_T}$ is of restricted-weak-type $(\frac{2d_h+2}{3},2d_h+2)$.

\item In the Rectangular Cases, $\scriptT_{R_T}$ is of restricted-weak-type $(\frac{T+4}{3},T+4)$.

\item In the Twisted Cases, $\scriptT_{R_T}$ is of restricted-weak-type $(\frac{2d_h+2}{3},2d_h+2)$.

\item In Cases (N) and (A), $\scriptT_{R_T}$ is of restricted-weak-type $(p_{v_2},q_{v_2})$.
\end{enumerate}

In Section \ref{S:First Relevant Vertex}, we prove (II) and, when $T\neq d_\omega$, we prove (I). In Sections \ref{S:Method of Refinements} and \ref{S: Conclusion of Degenerate Case} we finish the proof of (I) when $T=d_\omega$ by incorporating an orthogonality argument into Christ's method of refinements, as in \cite{CDSS}. We then prove (III) in Section \ref{S:Twisted Cases}, by analyzing the ``twisting'' of the graph of $\varphi$ by quantifying the nonconvexity of the image of local regions under $\nabla\varphi$, and using those calculations with a variant of the method of refinements.
\par
Finally, in Section \ref{S:Second Relevant Vertex Intro}, we turn our attention to proving (IV). We first begin by decomposing cases (A) and (N) into subcases, based on the location of $(\tfrac{1}{p_{v_2}},\tfrac{1}{q_{v_2}})$. 

In Section \ref{S:Second Relevant Vertex Conclusion}, we complete the proof of (IV), by going through each subcase and proving that $\scriptT_{R_T}$ is of restricted-weak-type $(p_{v_2},q_{v_2})$. In each subcase, our argument begins by computing restricted-weak-type $(\frac 32,3)$ bounds over small regions. In several of the subcases of Case (N), this computation also involves a variant of the method of refinements.

We then proceed to interpolate the restricted-weak-type $(\frac 32,3)$ bound with one or both of the restricted-weak-type $(\infty,\infty)$ and restricted-weak-type $(\frac 43,4)$ bounds to achieve our result. However, one of the subcases, denoted $(N_{q=2p}^{scal})$, relies on a more complex argument that involves incorporating an orthogonality argument into the method of refinements, similar to the argument in Section \ref{S:Method of Refinements}.


\section{Notation}\label{S:Notation}

For the operator $\scriptT$, most of our results will be bounds of restricted-weak-type. $\scriptT$ is of rwt $(p,q)$, with bound $C_{p,q,w}$, iff for all measurable sets $E$ of finite measure, 
$$
\norm{\scriptT\chi_E}_{L^{q,w}}\leq C_{p,q,w}\norm{\chi_E}_{L^p}=C_{p,q,w}|E|^{\frac 1p},
$$
or, equivalently, if for all $E$, $F$ finite, 
$$
\langle \scriptT\chi_E,\chi_F\rangle\leq C_{p,q,w}\norm{\chi_E}_{L^p}\norm{\chi_F}_{L^{q'}}=C_{p,q,w}|E|^\frac 1p|F|^{1-\frac 1q}.
$$

We will sometimes use the notation $\scriptT(E,F):=\langle \scriptT\chi_E,\chi_F\rangle$.

We define $\pi_1$ and $\pi_2$ to be the projection onto the first and second coordinate in $\R^2$, respectively, so that $\pi_1(z_1,z_2)=z_1$. And for $G\in \R^2$, we denote $\langle G \rangle$ to be the span of $G$, which would be a line, whenever $G\neq 0$.

The notation $\overline{Conv}(S)$ will denote the closed convex hull of $S$, while $\mu(\overline{Conv}(S))$ will refer to the Lebesgue measure of the convex hull.

To represent mixed Lebesgue norm spaces, we will use the notation
$$\|F(u,t)\|_{L^a_{u_2}L^b_{u_1,u_3}L^c_t(\{t \in S\})}:=\norm{\norm{\norm{F(u,t)}_{L^c(t\in S)}}_{L^b((u_1,u_3)\in \R^2)}}_{L^a(u_2\in \R)}.$$

We will sometimes need to decompose the operator $\scriptT$. If $R\subset \R^2$ is a measurable set, we define
\begin{equation}
\scriptT_R f(x):=\int_R f(x'-t,x_3-\varphi(t))dt.
\end{equation}
Also, unless otherwise stated, we will use $A\lesssim B$ to imply that there exists a constant $C$, uniform except for a possible $\varphi$ dependence, such that $A\leq CB$, and similarly for $\gtrsim$. If we are incorporating extra dependence, for instance $\epsilon$-dependence, we will use $\lesssim_\epsilon$. Finally, $A\approx B$ is equivalent to $A\lesssim B$ and $A\gtrsim B$. By $(A,B)\approx (D,E)$, we mean $A\approx D$ and $B\approx E$. Also, $C$ will be a constant that can change from line to line, which will depend only on $\varphi$ unless otherwise stated.
 
 
 \section{Algebraic Lemmas}\label{S:Algebraic Lemmas}

Our first few lemmas will connect the form that $\varphi$ and $\omega$ take.

\begin{lemma} \label{2N-3} If $\varphi=z_1^J(z_2-z_1^r)^{N'}+\scriptO((z_2-z_1^r)^{N'+1})$, some $J\geq 0$, $N'\geq 2$, and $r\geq 2$, then $\omega=Cz_1^{2J+r-2}(z_2-z_1^r)^{2N'-3}+\scriptO((z_2-z_1^r)^{2N'-2})$, some $C\neq 0$.
\end{lemma}

\begin{lemma} \label{A-2} If $\varphi=z_1^J+cz_1^{J-lr}z_2^{ls}+\scriptO(z_2^{ls+1})$, some $c\neq 0$ and $l\geq 1$, with $ls\geq 2$ and $J\geq 1$, then $\omega=Cz_1^{2J-lr-2}z_2^{ls-2}+\scriptO(z_2^{ls-1})$, some $C\neq 0$.
\end{lemma}
\begin{lemma} \label{2nu-2} If $\varphi=z_1^Jz_2^{\nu'}+\scriptO(z_2^{\nu'+1})$, some $J,\nu'\geq 1$, then $\omega=Cz_1^{2J-2}z_2^{2\nu'-2}+\scriptO(z_2^{2\nu'-1})$, some $C\neq 0$.
\end{lemma}

\begin{corollary} \label{multiplicities connection}: In Case(N), $T=2N-3$; in Case(A), $T=A-2$; and in Case($\nu$), $T=2\nu-2$. 
\end{corollary}
Next, we will prove each of these lemmas.

\begin{proof}[Proof of Lemma \ref{2N-3}] Making a change of variables, we set $x=z_1$ and $y=z_2-z_1^r$.

Then $\partial_{z_1}^2=\partial_{x}^2+r^2x^{2r-2}\partial_y^2-2 rx^{r-1}\partial_{xy}- r(r-1)x^{r-2}\partial_y$; $\partial_{z_2}^2=\partial_{y}^2$; and $\partial_{z_1z_2}=\partial_{xy}- rx^{r-1}\partial_{y}^2$, so 
$$
\omega=[\partial_x^2\varphi\partial_y^2\varphi-(\partial_{xy}\varphi)^2]-r(r-1)x^{r-2}\partial_y\varphi\partial_{y}^2\varphi.
$$

If $\varphi=x^Jy^{N'}+\scriptO(y^{N'+1})$, with $N'\geq 2$, then the terms $\partial_x^2\varphi\partial_y^2\varphi$ and $(\partial_{xy}\varphi)^2$ are both of $\scriptO(y^{2N'-2})$, while the final term can be written as
$$x^{r-2}\partial_y\varphi\partial_{y}^2\varphi=(N')^2(N'-1)x^{2J+r-2}y^{2N'-3}+\scriptO(y^{2N'-2}).$$
Thus, 
\begin{equation*}
\omega=Cx^{2J+r-2}y^{2N'-3}+\scriptO(y^{2N'-2}), \text{ some } C\neq 0. \qedhere
\end{equation*}
\end{proof} 
\begin{proof}[Proof of Lemma \ref{A-2}] For $\varphi=z_1^J+cz_1^{J-lr}z_2^{ls}+\scriptO(z_2^{ls+1})$,
$$
\omega=\partial_{z_1}^2\varphi\partial_{z_2}^2\varphi-(\partial_{z_1z_2}\varphi)^2.
$$
The term $(\partial_{z_1z_2}\varphi)^2$ is $\scriptO(z_2^{2ls-2})$, while 
$$
\partial_{z_1}^2\varphi\partial_{z_2}^2\varphi=cJ(J-1)ls(ls-1)z_1^{2J-lr-2}z_2^{ls-2}+\scriptO(z_2^{ls-1}).
$$
Since $ls\geq 2$, and since $\nabla\varphi(0)=0$ implies that $J\geq 2$, 
\begin{equation*}
\omega=Cz_1^{2J-lr-2}z_2^{ls-2}+\scriptO(z_2^{ls-1}), \text{ some } C\neq 0. \qedhere
\end{equation*}
\end{proof}
\begin{proof}[Proof of Lemma \ref{2nu-2}] For $\varphi=z_1^Jz_2^{\nu'}+\scriptO(z_2^{\nu'+1})$, with $J,\nu'\geq 1$,
$$
\omega=\partial_{z_1}^2\varphi\partial_{z_2}^2\varphi-(\partial_{z_1z_2}\varphi)^2.
$$
Looking at each term separately,
\begin{align*}
(\partial_{z_1z_2}\varphi)^2&=[J\nu' z_1^{J-1}z_2^{\nu'-1}]^2+\scriptO(z_2^{2\nu'-1}),
\\
\partial_{z_1}^2\varphi\partial_{z_2}^2\varphi&=[J(J-1)z_1^{J-2}z_2^{\nu'}][\nu'(\nu'-1)z_1^{J}z_2^{\nu'-2}]+\scriptO(z_2^{2\nu'-1}).
\end{align*}
Thus, 
\begin{align*}
\omega&=[1-\nu'-J]J\nu' z_1^{2J-2}z_2^{2\nu'-2}+\scriptO(z_2^{2\nu'-1})
\\
&=Cz_1^{2J-2}z_2^{2\nu'-2}+\scriptO(z_2^{2\nu'-1}), \text{ some } C\neq 0.\qedhere
\end{align*}
\end{proof}
\begin{proof}[Proof of Corollary \ref{multiplicities connection}]
In Case(N), $\min\{r,s\}=1$ by Proposition \ref{dpsi}, so after rescaling, $\varphi$ can be written as in Lemma \ref{2N-3} with $N=N'\geq 2$, implying that $T=2N-3$. In Case(A), after rescaling $\varphi$ can be written as in Lemma \ref{A-2}, with $A=ls\geq 2$. Additionally, if $J$ were less that $2$, then the hypothesis $\nabla\varphi(0)=0$ of Theorem \ref{Main Theorem Alt} wouldn't be satisfied. Thus, by Lemma \ref{A-2}, $T=A-2$. In Case($\nu$), after rescaling, $\varphi$ can be written as in Lemma \ref{2nu-2} with $\nu=\nu'\geq 1$. Additionally, if $J=0$, then $\varphi=z_2^\nu$, violating the hypothesis $\nabla\varphi(0)=0$ of Theorem \ref{Main Theorem Alt}. Thus, by Lemma \ref{2nu-2}, $T=2\nu-2$.
\end{proof}

Next, we look at the relationship between the homogeneous distances for $\varphi$ and its derivatives.

\begin{lemma}\label{domega dh} The homogeneous distances of $\varphi$ and its derivatives obey the following relations: $d_\omega=2d_h-2$, $d_h(\partial_{z_1}\varphi)=d_h-\frac{s}{r+s}$, and $d_h(\partial_{z_2}\varphi)=d_h-\frac{r}{r+s}$.
\end{lemma}

\begin{proof} Since $\varphi$ is mixed-homogeneous, there exists some $m$ such that $\sigma \varphi(z_1,z_2)=\varphi(\sigma^{\frac sm}z_1,\sigma^{\frac rm}z_2)$. Taking derivatives,
\begin{align*}
\sigma^2\partial_{z_1}^2\varphi(z_1,z_2)\partial_{z_2}^2\varphi(z_1,z_2)&=\sigma^{2\frac{s+r}{m}}\partial_{z_1}^2\varphi(\sigma^\frac sm z_1,\sigma^\frac rm z_2)\partial_{z_2}^2\varphi(\sigma^\frac sm z_1,\sigma^\frac rm z_2)
\\
\sigma^2(\partial_{z_1z_2}\varphi(z_1,z_2))^2&=\sigma^{2\frac{s+r}{m}}(\partial_{z_1z_2}\varphi(\sigma^\frac sm z_1,\sigma^\frac rm z_2))^2.
\end{align*}
Putting these together, and using $d_h=\frac{m}{r+s}$:
\begin{equation}\label{domega:1st}
\sigma^2\omega(z_1,z_2)=\sigma^{\frac 2{d_h}}\omega(\sigma^\frac sm z_1,\sigma^\frac rm z_2).
\end{equation}
When $d_h\neq 1$, we can change variables using $\gamma=\sigma^{2-\frac 2{d_h}}$  to get
\begin{equation}\label{domega:2nd}
\gamma\omega(z_1,z_2)=\omega(\gamma^\frac{s}{m(2-2/d_h)}z_1,\gamma^\frac{r}{m(2-2/d_h)}z_2).
\end{equation}
Thus, $d_\omega=\frac{m(2-2/d_h)}{r+s}=d_h(2-\frac 2{d_h})=2d_h-2$. And when $d_h=1$, $\omega$ is invariant under scaling, implying that $\omega$ is constant and $d_\omega=0$.
\qed
\\
\\
Next, by symmetry, it suffices to show the claim for $d_h(\partial_{z_1}\varphi)=d_h-\frac{s}{r+s}$. Taking a derivative, we get
$$
\sigma \partial_{z_1}\varphi(z_1,z_2)=\sigma^\frac sm \partial_{z_1}\varphi(\sigma^\frac sm z_1,\sigma^\frac rm z_2),
$$
which becomes, under the change of variables $\gamma=\sigma^{1-\frac sm}=\sigma^\frac{m-s}{m}$, 
$$
\gamma\partial_{z_1}\varphi(z_1,z_2)=\partial_{z_1}\varphi(\gamma^\frac s{m-s} z_1,\gamma^\frac r{m-s} z_2).
$$
Thus, $d_h(\partial_{z_1}\varphi)=\frac{m-s}{s+r}=d_h-\frac{s}{s+r}$, and therefore by symmetry, $d_h(\partial_{z_2}\varphi)=d_h-\frac{r}{s+r}$.
\end{proof}

Next, we want to show that if $\omega\equiv 0$, then $\varphi$ does not satisfy the assumptions of Theorem \ref{Main Theorem Alt}.

\begin{prop}\label{omega=0} If $\varphi$ is a mixed homogeneous polynomial with $\nabla\varphi(0)=0$ and $\omega\equiv 0$, then $\varphi$ is a constant multiple of $z_1^J$, $z_2^J$, or $(z_1+\lambda z_2)^J$, with $\lambda\in \R$ and $J\in \N$. Consequently, in Theorem \ref{Main Theorem Alt}, $\omega$ is never identically $0$.
\end{prop}
\begin{proof} Exclude the aforementioned cases. Then after possibly swapping $z_1$ and $z_2$ and rescaling $\varphi$, either
\\
(1) $\varphi=z_1^Jz_2^{\nu'}+\scriptO(z_2^{\nu'+1})$,  some $J, \nu'\geq 1$, or
\\
(2) $\varphi=z_1^J+cz_1^{J-lr}z_2^{ls}+\scriptO(z_2^{ls+s})$
\\
.\hspace{1.5pc} $=Cz_2^K+\tilde{c}z_2^{K-ks}z_1^{kr}+\scriptO(z_1^{kr+r})$, for some $k,l$, with $C,c,\tilde{c}\neq 0$, and $J,K\geq 2$.

\vspace{.5pc}

In case 1, by Lemma \ref{2nu-2}, 
$$\omega=Cz_1^{2J-2}z_2^{2\nu'-2}+\scriptO(z_2^{2\nu'-1})\not\equiv 0$$
In case 2, if $\max\{ls,kr\}>1$, then we can use Lemma \ref{A-2}. By symmetry, we need only consider $ls>1$. Then for some $C\neq 0$,
$$
\omega=Cz_1^{2J-lr-2}z_2^{ls-2}+\scriptO(z_2^{ls-1})\not\equiv 0.
$$
Lastly, if $\max\{ls,kr\}=1$, then $\varphi$ is homogeneous and case 2 simplifies to
$$
\varphi=z_1^J+c_1z_1^{J-1}z_2+...+c_{J-1}z_1z_2^{J-1}+c_Jz_2^J, \text{ with } c_1\neq 0.
$$
We want to show that if $\omega\equiv 0$, then $\varphi=(z_1-\lambda z_2)^J$, for some $\lambda$.
\par 
For later convenience, we will prove an even stronger statement here:

\begin{lemma}\label{vanishing omega} Suppose that $\varphi=z_1^J+c_1z_1^{J-1}z_2+...+c_{J-1}z_1z_2^{J-1}+c_Jz_2^J, \text{ with } c_1\neq 0$, and suppose that $\omega=\scriptO(z_2^{J-1})$. Then $\varphi=(z_1-\lambda z_2)^J$, with $\lambda=\frac{c_1}{J}$, and consequently $\omega\equiv 0$.
\end{lemma}

\begin{proof} Replacing $c_1$ with $J\lambda $, we can write $\varphi$ as
$$
\varphi=z_1^J+J\lambda z_1^{J-1}z_2+c_2z_1^{J-2}z_2^2+...+c_{J-1}z_1z_2^{J-1}+c_Jz_2^J, \text{ with } \lambda\neq 0.
$$ 
Since $\tilde{\varphi}=(z_1+\lambda z_2)^J$ produces such a $\varphi$, with $c_2$,...,$c_J$ fixed in terms of $\lambda$ and $J$, and since $det D^2\tilde\varphi\equiv 0$, it suffices to show that:
\\
\\
If $\omega=\scriptO(z_2^{J-1})$, then $c_2$,...,$c_J$ are uniquely determined by $\lambda$ and $J$.

Expanding, there exist functions $f_{1,K}$ and $f_{2,K}$, for $2\leq K\leq J$, such that
$$(\partial_{z_1}^2\varphi\partial_{z_2}^2\varphi)=\sum_{K=2}^J [JK(J-1)(K-1)c_K+f_{2,K}(J,\lambda,c_2,...,c_{K-1})]z_1^{2J-K-2}z_2^{K-2}+\scriptO(z_2^{J-1})$$
and
$$(\partial_{z_1z_2}\varphi)^2=\sum_{K=2}^J f_{1,K}(J,\lambda,c_2,...,c_{K-1})z_1^{2J-K-2}z_2^{K-2}+\scriptO(z_2^{J-1}),$$
so there exist functions $f_{3,K}$, $2\leq K\leq J$, such that
$$
\omega=\sum_{K=2}^J [JK(J-1)(K-1)c_K-f_{3,K}(J,\lambda, c_2,...,c_{K-1})]z_1^{2J-K-2}z_2^{K-2}+\scriptO(z_2^{J-1}).
$$ 
Therefore, $\omega=\scriptO(z_2^{J-1})$ implies that, for $K=2,3,...,J$, 
$$
c_K=\tfrac{f_{3,K}(J,\lambda,c_2,...,c_{K-1})}{JK(J-1)(K-1)},
$$
and therefore inductively each $c_K$ is uniquely determined by $J$ and $\lambda$, and the Lemma proof is complete.
\end{proof}

Then due to Lemma \ref{vanishing omega}, the proof of Proposition \ref{omega=0} is complete.
\end{proof}

\begin{corollary}\label{dh>1}
Theorem \ref{Main Theorem Alt} holds in the case $d_h=1$, and in the case $d_h<1$, Theorem \ref{Main Theorem Alt} holds vacuously.
\end{corollary}

\begin{proof}
If $d_h=1$, then $\omega$ is constant by Lemma \ref{domega dh}. If that constant is nonzero, then Theorem \ref{Gressman} applies, while if $\omega\equiv 0$, then Theorem \ref{Main Theorem Alt} holds vacuously by Proposition \ref{omega=0}. If $d_h<1$, then $d_\omega<0$ by Lemma \ref{domega dh}, implying $\omega\equiv 0$, so again Theorem \ref{Main Theorem Alt} holds vacuously.
\end{proof}

Next, we show that homogeneous $\varphi$ can only belong to to a couple cases of Theorem \ref{Main Theorem Alt}.

\begin{prop} \label{Homogeneous Cases}
If $\varphi$ is homogeneous, and $\varphi$ satisfies the hypotheses of Theorem \ref{Main Theorem Alt}, then either $T\leq d_\omega$ or $\varphi$ belongs to case($\nu$). Consequently, in case(A) and twisted case(i), we have $\max\{r,s\}\geq 2$.
\end{prop}
\begin{proof} If $\varphi$ is homogeneous, then every factor of $\omega$ is linear. Thus, when $T>d_\omega$, $\varphi$ can only belong to case($\nu$),  case(A), or twisted case(i). 

If $\varphi$ belongs to case(A), then after a linear change of coordinates and rescaling, $\varphi=z_1^J+cz_1^{J-A}z_2^A+\scriptO(z_2^{A+1})$, $A\geq 2, c\neq 0$, and by Lemma \ref{A-2} or Corollary \ref{multiplicities connection}, $T=A-2$. Also, $d_\omega=2d_h-2=2\frac J2-2=J-2$. But since $c\neq 0$ implies that $A\leq J$, then $T>d_\omega$ is impossible, resulting in a contradiction. 

If $\varphi$ belongs to case(i), then after a linear change of coordinates and rescaling, $\varphi=z_1^J+cz_1^{J-1}z_2+\scriptO(z_2^2)$, $c\neq 0$, and $\omega=Cz_1^Mz_2^T+\scriptO(z_2^{T+1})$, some $M$, some $C\neq 0$. Then $T>d_\omega=2d_h-2=2\frac J2-2=J-2$, so $T\geq J-1$. Thus, $\omega=\scriptO(z_2^{J-1})$, and by Lemma \ref{vanishing omega}, $\omega\equiv 0$ and $\varphi=(z_1+\frac cJ z_2)^J$, so $\varphi$ violates the hypotheses of Theorem \ref{Main Theorem Alt}.
\end{proof}
\begin{remark}\label{Homogeneous Cases Remark}
The argument for case(A) in Proposition \ref{Homogeneous Cases} also implies that if $\varphi=z_1^J+cz_1^{J-ls}z_2^{ls}+\scriptO(z_2^{ls+1})$, with $ls\geq 2$, and $z_2$ has multiplicity $T=d_\omega$ in $\omega$, then $J=ls$ and $\varphi=z_1^J+cz_2^J$, some $c\neq 0$.
\end{remark}
We will now complete the proof of Proposition \ref{Main Theorem Equivalence}.
\begin{proposition}\label{Main Theorem Equivalence 2}
Theorems \ref{Main Theorem Original} and \ref{Main Theorem Alt} are equivalent.
\end{proposition}
We begin with three technical lemmas.

\begin{lemma}\label{equivalencelemma1} The Newton distance, $d(\varphi)$, equals $\max\{\nu,d_h\}$.
\end{lemma}
\begin{proof}[Proof of Lemma \ref{equivalencelemma1}] First, the Newton diagram of a mixed homogeneous polynomial will have at most 3 edges: a horizontal edge $y=\tilde\nu_2$ corresponding to the multiplicity of $z_2$ in $\varphi$, a vertical edge $x=\tilde\nu_1$ corresponding to the multiplicity of $z_1$ in $\varphi$, and an edge corresponding to the mixed homogeneity of $\varphi$, which if extended into an infinite line, would intersect the bisectrix at $(d_h,d_h)$. (One can check that $\varphi(z_1,z_2)+z_1^{d_h}z_2^{d_h}$ is still mixed homogeneous to verify this.) Thus, since the bisectrix intersects the Newton diagram at $(d(\varphi),d(\varphi))$, then $d(\varphi)=\max\{\tilde\nu_1,\tilde\nu_2,d_h\}$. 

If $\tilde\nu_1>d_h$, then by Lemmas \ref{2nu-2} and \ref{domega dh}, $z_1$ has multiplicity greater than $d_\omega$ in $\omega$, so by our definitions, $\nu=\tilde\nu_1$. Similarly for $\tilde\nu_2$. Likewise, if $\nu>d_h$, then after linearly adapting $\varphi$ as is assumed by Theorem \ref{Main Theorem Original}, $\nu$ will be the multiplicity of either $z_1$ or $z_2$ in $\varphi$, so $\nu=\tilde\nu_1$ or $\nu=\tilde\nu_2$.  Thus, $d(\varphi)=\max\{\tilde\nu_1,\tilde\nu_2,d_h\}=\max\{\nu,d_h\}$. \end{proof}

\begin{lemma}\label{equivalencelemma2} If $\max\{A,d(\varphi_R)\}>2d(\varphi)$, then $A=d(\varphi_R)$.
\end{lemma}
\begin{proof}[Proof of Lemma \ref{equivalencelemma2}] Suppose $d(\varphi_R)>2d(\varphi)$. Then, since $d(\varphi_R)>d(\varphi)$, and therefore by Lemma \ref{equivalencelemma1} $d(\varphi_R)>d_h$, the bisectrix must intersect the Newton diagram of $\varphi_R$ on a vertical or horizontal edge, and so we have that $\varphi=c_1z_1^J+c_2z_1^Kz_2^{d(\varphi_R)}+o(z_2^{d(\varphi_R)})$, some $J\geq 1, K\geq 0$, $c_1,c_2\neq 0$, up to a swapping of $z_1$ and $z_2$. Then, by Lemma \ref{A-2} and Lemma \ref{domega dh}, since $d(\varphi_R)>2d_h$, $z_2$ will have multiplicity greater than $d_\omega$ as a factor of $\omega$, so by the definition of $A$, $d(\varphi_R)=A$. 

Similarly, suppose $A>2d(\varphi)$. Then by Lemma \ref{equivalencelemma1}, $A>2d_h$, so by Lemma \ref{domega dh}, $A-2>d_\omega$. Since $A\neq 0$, then after a possible swap of $z_1$ and $z_2$ and rescaling, $
\varphi(z)=z_1^J+cz_1^Kf_T^{A}+o(f_T^{A})$, some $J\geq 1, K\geq 0, c\neq 0$ by the definition of $A$, where $f_T\neq z_1$ is linear. Then, after a linear transformation, we can use Lemma \ref{A-2} to find that $f_T$ has multiplicity greater than $d_\omega$ in $\omega$. Then $\varphi$ belongs to Case(A), so by Lemma \ref{Homogeneous Cases}, $\varphi$ cannot be homogeneous, so $f_T=z_2$ up to a constant. Then, $\varphi_{R2}=z_1^Kz_2^A+o(z_2^A)$, so since $A>d_h$, the bisectrix will intersect the Newton diagram of $\varphi_{R2}$ on a vertical or horizontal edge, at $(A,A)$, implying that $d(\varphi_{R2})=A$. Thus, $d(\varphi_R)\geq A>2d(\varphi)$, so by the first part of the proof, $d(\varphi_R)=A$.
\end{proof}

\begin{lemma}\label{equivalencelemma3} If $\max\{N,h(\varphi)\}> d(\varphi)+\frac 12$, then $N=d(\varphi)$.
\end{lemma}
\begin{proof}[Proof of Lemma \ref{equivalencelemma3}] Suppose $h(\varphi)>d(\varphi)+\frac 12$. Since $h(\varphi)=\max\{d(\varphi),o(\varphi)\}$, this implies that $h(\varphi)=o(\varphi)$, the maximal multiplicity of the real irreducible factors of $\varphi$. Since all the linear factors of $\varphi$ have multiplicity not exceeding $d(\varphi)$ (by our \textit{linearly adapted} assumption), any factor associated with $o(\varphi)$ must be nonlinear. Since $d(\varphi)\geq d_h$ by Lemma \ref{equivalencelemma1}, then by Lemma \ref{domega dh}, $h(\varphi)>\frac{d_\omega+3}{2}$.  Since $\nabla\varphi(0)=0$ and $\varphi(0)=0$ imply that $d(\varphi)>\frac 12$, then $o(\varphi)=h(\varphi)>1$. Additionally, since $d(\varphi)\geq d_h$, then $o(\varphi)>d_h$, so by Proposition \ref{dpsi}, $\min(r,s)=1$. Therefore, Lemma \ref{2N-3} can be applied to $\varphi$ with $h(\varphi)=N'$, implying that the factor associated with $h(\varphi)$ will have multiplicity greater than $d_\omega$ in $\omega$, so by the definition of $N$, $h(\varphi)=o(\varphi)=N$.

Similarly, suppose $N>d(\varphi)+\frac 12$. Since $N>d_h$ by Lemma \ref{equivalencelemma1}, then by Proposition \ref{dpsi}, $N$ is the highest multiplicity of any real irreducible factor of $\varphi$, so $N=o(\varphi)=h(\varphi)$.
\end{proof}

Now we turn to the proof of Proposition \ref{Main Theorem Equivalence 2} by proving Proposition \ref{Main Theorem Equivalence}:

\begin{proof}[Proof of Proposition \ref{Main Theorem Equivalence}]   
Inequality \eqref{E:Alt 3} is equivalent to the validity of \eqref{E:Main 1} and \eqref{E:Main 2}. By Lemma \ref{equivalencelemma1}, the validity of \eqref{E:Alt 4} and \eqref{E:Alt 5} is equivalent to \eqref{E:Main 3}.

If $\max\{A,d(\varphi_R)\}\leq 2d(\varphi)$, then by Lemma \ref{equivalencelemma1}, \eqref{E:Alt 6} follows from \eqref{E:Alt 3}, \eqref{E:Alt 4}, and \eqref{E:Alt 5}, while \eqref{E:Main 4} follows from \eqref{E:Main 1}, \eqref{E:Main 2}, and \eqref{E:Main 3}. If $\max\{A,d(\varphi_R)\}>2d(\varphi)$, then by Lemma \ref{equivalencelemma2}, $A=d(\varphi_R)$, in which case \eqref{E:Alt 6} is equivalent to \eqref{E:Main 4}.

If $\max\{N,h(\varphi)\}\leq d(\varphi)+\frac 12$, then inequalities \eqref{E:Alt 3}, \eqref{E:Alt 4}, and \eqref{E:Alt 5} imply \eqref{E:Alt 7} and \eqref{E:Alt 8}, while \eqref{E:Main 1}, \eqref{E:Main 2}, and \eqref{E:Main 3} imply \eqref{E:Main 5} and \eqref{E:Main 6}. By Lemma \ref{equivalencelemma3}, if $\max\{N,h(\varphi)\}> d(\varphi)+\frac 12$, \eqref{E:Alt 7} is equivalent to \eqref{E:Main 5} and \eqref{E:Alt 8} is equivalent to \eqref{E:Main 6}. 

This completes the proof of Proposition \ref{Main Theorem Equivalence 2}, and likewise Proposition \ref{Main Theorem Equivalence}.
\end{proof}

 \section{Necessary Conditions}\label{S:Necessary Conditions}

The necessity of each bound in Theorems \ref{Main Theorem Alt} and \ref{Main Theorem Original} was proven in \cite{DZ} for the strictly mixed homogeneous cases, with proofs that can also be applied in the homogeneous cases, and by \cite{FGU1} in the homogeneous cases. For completeness, we will include similar proofs here.

\vspace{.5pc}
 
 \textit{Necessity of $q\geq p$}: Let $E=[-3K,3K]^3$ and $F=[-K,K]^3$, where $K\gg 1$. Then $|E|\approx |F|\approx K^3$, and on $F$, $\scriptT\chi_E\approx 1$. Then $\scriptT$ being of rwt $(p,q)$ requires that
$$
1\approx \text{avg}_F \scriptT \chi_E\lesssim |E|^\frac 1p|F|^{-\frac 1q}\approx K^{3(\frac 1p-\frac 1q)},
$$
implying that $1\lesssim K^{3(\frac 1p-\frac 1q)}$. Since $K$ can be taken to be arbitrarily large, this implies that $0\leq \frac 1p-\frac 1q$, which simplifies to
$$
q\geq p.
$$

\textit{Necessity of  $q\leq 3p$}: (By duality the necessity of $\frac 1q\geq \frac 3p-2$ follows.)
Let $E=[-\frac 14,\frac 14]^2\times [\varphi-C\epsilon,\varphi+C\epsilon]$ and $F=[-\epsilon,\epsilon]^3$, where $\epsilon\ll 1$. Then $|E|\approx \epsilon$, $|F|\approx \epsilon^3$, and since $|\nabla \varphi|$ is bounded on $[-1,1]^2$, there exists a $C$ large enough, independent of $\epsilon$, so that on $F$, $\scriptT\chi_E\approx 1$. Then $\scriptT$ being of rwt $(p,q)$ requires that
$$
1\approx \text{avg}_F \scriptT \chi_E \lesssim |E|^\frac 1p|F|^{-\frac 1q}\approx \epsilon^{\frac 1p-\frac 3q},
$$
implying that $1\lesssim \epsilon^{\frac 1p-\frac 3q}$. Since $\epsilon$ can be taken to be arbitrarily small, this implies that $0\geq \frac 1p-\frac 3q$, which after simplifying becomes
$$
q\leq 3p.
$$
 
 \textit{Necessity of the Scaling line $\frac 1q\geq \frac 1p-\frac 1{d_h+1}$}: 
 Let $S$ be a hypersurface in $\R^3$, and $S_R$ be the portion of the hypersurface over region $R\in \R^2$. Denote $\scriptT_R$ as the averaging with hypersurface $S_R$. By the structure of $\scriptT$, if $R'\supset R$, then $\scriptT_{R'}\chi_E\geq \scriptT_{R}\chi_E$, so the (strong or restricted-weak-type) bound of $\scriptT$ cannot increase as $R$ increases. 

Denote $\sigma I$ as $[-\sigma^{\kappa_1},\sigma^{\kappa_1}]\times [-\sigma^{\kappa_2},\sigma^{\kappa_2}]$. Then by the upcoming Lemma \ref{scaling} (letting $\norm{\scriptT}_{p,q}$ denote either the strong or restricted-weak-type bound at $(\frac 1p,\frac 1q)$), 
$$
\norm{\scriptT_{\sigma I}}_{p,q}=\sigma^{\frac{d_h+1}{d_h}(\frac 1q-\frac 1p+\frac 1{d_h+1})}\norm{\scriptT_{[-1,1]^2}}_{p,q}.
$$ 
Then, since $\sigma I \supset [-1,1]^2$ for $\sigma>1$, we need $\sigma^{\frac{d_h+1}{d_h}(\frac 1q-\frac 1p+\frac 1{d_h+1})}\geq 1$ for $\sigma>1$, implying that boundedness at $(\frac 1p,\frac 1q)$ requires
$$
\tfrac 1q\geq \tfrac 1p-\tfrac 1{d_h+1}.
$$
 
 \textit{Necessity of $\frac 1q\geq \frac 1p-\frac 1{\nu+1}$ in $Case(\nu)$}:  In this case, it suffices to consider $\varphi=z_1^Jz_2^\nu+\scriptO(z_2^{\nu+1})$. Then, choose $E=[-3,3]\times [-3\epsilon,3\epsilon]\times [-3\epsilon^\nu,3\epsilon^\nu]$ and $F=[-1,1]\times [-\epsilon,\epsilon]\times [-\epsilon^\nu,\epsilon^\nu]$, where $\epsilon\ll 1$. Then $|E|\approx |F|\approx \epsilon^{\nu+1}$, and on $F$, $\scriptT\chi_E\approx \epsilon$. Then $\scriptT$ being of rwt $(p,q)$ requires that
$$
\epsilon\approx \text{avg}_F \scriptT \chi_E\lesssim |E|^\frac 1p|F|^{-\frac 1q}\approx \epsilon^{(\nu+1)(\frac 1p-\frac 1q)},
$$
implying that $\epsilon\lesssim \epsilon^{(\nu+1)(\frac 1p-\frac 1q)}$. Since $\epsilon$ can be taken to be arbitrarily small, this implies that $1\geq (\nu+1)(\frac 1p-\frac 1q)$, or after simplifying, 
$$
\tfrac 1q\geq \tfrac 1p-\tfrac 1{\nu+1}.
$$
 
 \textit{Necessity of $\frac 1q\geq\frac 1p-\frac 1N$ in Case(N)}: In this case, by Proposition \ref{dpsi}, we can assume $s=1$ and write $\varphi=z_1^J(z_2-\lambda z_1^r)^N+\scriptO((z_2-\lambda z_1^r)^{N+1})$. Then, choose $E=[-3,3]\times [-3\lambda ,3\lambda]\times [-3\epsilon^N,3\epsilon^N]$ and $F=[-1,1]\times [-\lambda,\lambda]\times [-\epsilon^N,\epsilon^N]$, where $\epsilon\ll 1$. Then $|E|\approx |F|\approx \epsilon^N$, and on $F$, $\scriptT\chi_E\approx \epsilon$. Then $\scriptT$ being of rwt $(p,q)$ requires that
$$
\epsilon \approx \text{avg}_F \scriptT \chi_E\lesssim |E|^\frac 1p|F|^{-\frac 1q}\approx \epsilon^{N(\frac 1p-\frac 1q)},
$$
implying that $\epsilon\lesssim \epsilon^{N(\frac 1p-\frac 1q)}$. Since $\epsilon$ can be taken to be arbitrarily small, this implies that $1\geq N(\frac 1p-\frac 1q)$, or after simplifying, 
$$
\tfrac 1q\geq \tfrac 1p-\tfrac 1N.
$$
 
 \textit{Necessity of $\frac 1q\geq\frac{N+1}{N+2}\frac{1}{p}-\frac{1}{N+2}$ in Case(N)}: (By duality the necessity of $\frac 1q\geq\frac{N+2}{N+1}\frac{1}{p}-\frac{2}{N+1}$ follows.) In this case, by Proposition \ref{dpsi}, we can assume $s=1$ and write $\varphi=z_1^J(z_2-\lambda z_1^r)^N+\scriptO((z_2-\lambda z_1^r)^{N+1})$. Then, choose $E=[\frac 12,1]\times [\lambda z_1^r-3\epsilon,\lambda z_1^r+3\epsilon]\times [-3\epsilon^N,3\epsilon^N]$ and $F=[-\epsilon,\epsilon]^2 \times [-\epsilon^N,\epsilon^N]$, where $\epsilon\ll 1$. Then $|E|\approx \epsilon^{N+1}$, $|F|\approx \epsilon^{N+2}$, and on $F$, $\scriptT\chi_E\approx \epsilon$. Then $\scriptT$ being of rwt $(p,q)$ requires that
$$
\epsilon\approx \text{avg}_F \scriptT \chi_E\lesssim |E|^\frac 1p|F|^{-\frac 1q}\approx \epsilon^{(N+1)\frac 1p-(N+2)\frac 1q},
$$
implying that $\epsilon\lesssim \epsilon^{(N+1)\frac 1p-(N+2)\frac 1q}$. Since $\epsilon$ can be taken to be arbitrarily small, this implies that $1\geq (N+1)\frac 1p-(N+2)\frac 1q$, or 
$$
\tfrac 1q\geq \tfrac{N+1}{N+2}\tfrac 1p-\tfrac 1{N+2}.
$$
 
 \textit{Necessity of $\frac{1}{q}\geq \frac{A+1}{2A+1}\frac{1}{p}-\frac{1}{2A+1}$ in Case(A)}: (By duality the necessity of $\frac 1q\geq\frac{2A+1}{A+1}\frac{1}{p}-1$ follows.) In this case, it suffices to consider $\varphi=z_1^J+z_1^{J-lr}z_2^A+\scriptO(z_2^{A+1})$, $A\geq 2$. Then, choose $E=[\frac 12,1]\times [-3\epsilon,3\epsilon]\times [z_1^J+z_1^{J-lr}z_2^A-3\epsilon^A,z_1^J+z_1^{J-lr}z_2^A+3\epsilon^A]$ and $F=[-\epsilon^A,\epsilon^A]\times [-\epsilon,\epsilon]\times [-\epsilon^A,\epsilon^A]$, where $\epsilon\ll 1$. Then $|E|\approx \epsilon^{A+1}$, $|F|\approx \epsilon^{2A+1}$, and on $F$, $\scriptT\chi_E\approx \epsilon$. Then $\scriptT$ being of rwt $(p,q)$ requires that
$$
\epsilon \approx \text{avg}_F \scriptT \chi_E\lesssim |E|^\frac 1p|F|^{-\frac 1q}\approx \epsilon^{(A+1)\frac 1p-(2A+1)\frac 1q},
$$
implying that $\epsilon\lesssim \epsilon^{(A+1)\frac 1p-(2A+1)\frac 1q}$. Since $\epsilon$ can be taken to be arbitrarily small, this implies that $1\geq (A+1)\frac 1p-(2A+1)\frac 1q$, or
$$
\tfrac 1q\geq \tfrac{A+1}{2A+1}\tfrac 1p-\tfrac 1{2A+1}. \qquad  \qed
$$

 \section{Scaling Symmetries}\label{S:Scaling Symmetries}
 
 When a vertex of the polygon arising in Theorem \ref{Main Theorem Alt} occurs on the scaling line $\frac 1q=\frac 1p-\frac 1{d_h+1}$, the mixed-homogeneity of $\varphi$ and scale-invariance of $\scriptT$ can be exploited as follows. 

 Let $R\subset \R^2$, $f:\R^3\rightarrow \R$, $\bm{\kappa}=\bm{\kappa_\varphi}$, and $\sigma>0$. We will use notation
\begin{equation}
 R_{\bm{\kappa},\sigma}:=\{(\sigma^{\kappa_1}z_1,\sigma^{\kappa_2}z_2)|(z_1,z_2)\in R\}, \quad \quad f_\sigma(\cdot,\cdot,\cdot):=f(\sigma^{\kappa_1}\cdot,\sigma^{\kappa_2}\cdot,\sigma\cdot).
\end{equation}

 \begin{definition} Let $\bm{\kappa}=(\kappa_1,\kappa_2)\in (0,\infty)^2$, and let $\psi:\R^2\rightarrow \R$. We say $\psi$ is \textit{$\bm{\kappa}$-mixed homogeneous} if, for every $\sigma>0$, $\psi(\sigma^{\kappa_1}\cdot,\sigma^{\kappa_2}\cdot)=\sigma\psi(\cdot,\cdot)$.
 \end{definition}

\begin{definition} 
 $\scriptR\subset \R^2$ is \textit{$\bm{\kappa}$-scale invariant} if $\scriptR=\scriptR_{\kappa,\sigma}$ for all $\sigma>0$.
\end{definition}

\begin{lemma}\label{scaling} Let $R\subset \R^2$, let $\bm{\kappa}=\bm{\kappa_\varphi}$, and let $\sigma>0$. Let $\norm{\cdot}_{p,q}$ refer to either the strong-type bound $\norm{\cdot}_{L^p\rightarrow L^q}$ or the restricted-weak-type bound $\norm{\cdot}_{L^{p,1}\rightarrow L^{q,\infty}}$. If $\norm{\scriptT_{R}}_{p,q}<\infty$, then $\norm{\scriptT_{R_{\bm{\kappa},\sigma}}}_{p,q}=\sigma^{(\kappa_1+\kappa_2)[1+\frac{1}{q}-\frac{1}{p}]+[\frac{1}{q}-\frac{1}{p}]}\norm{\scriptT_{R}}_{p,q}$. 
\end{lemma} 
 
\begin{proof}

\begin{align*}
\scriptT_R f_\sigma &:=T_R[f(\sigma^{\kappa_1}\cdot,\sigma^{\kappa_2}\cdot,\sigma\cdot)]\\
&=\int_{R}f(\sigma^{\kappa_1}x_1-(\sigma^{\kappa_1}t_1),\sigma^{\kappa_2}x_2-(\sigma^{\kappa_2}t_2),\sigma x_3-\varphi(\sigma^{\kappa_1}t_1,\sigma^{\kappa_2}t_2))dt_1dt_2 \\
&=\sigma^{-(\kappa_1+\kappa_2)}\int_{R_{\bm{\kappa},\sigma}}f(\sigma^{\kappa_1}x_1-u_1,\sigma^{\kappa_2}x_2-u_2,\sigma x_3-\varphi(u_1,u_2))du_1du_2 \\
&=\sigma^{-(\kappa_1+\kappa_2)}(\scriptT_{R_{\bm{\kappa},\sigma}}f)(\sigma^{\kappa_1}x_1,\sigma^{\kappa_2}x_2,\sigma x_3) =\sigma^{-(\kappa_1+\kappa_2)}(\scriptT_{R_{\bm{\kappa},\sigma}}f)_{\sigma}.
\end{align*}

Then, by scaling,
\begin{align*}
\sigma^{-(\kappa_1+\kappa_2)}\sigma^{-\frac{(\kappa_1+\kappa_2+1)}{q}}\norm{\scriptT_{R_{\bm{\kappa},\sigma}}f}_{L^q(\R^3)}&=\sigma^{-(\kappa_1+\kappa_2)}\norm{(T_{R_{\bm{\kappa},\sigma}}f)_\sigma}_{L^q(\R^3)} \\
&=\norm{T_R(f_\sigma)}_{L^q(\R^3)}\\
&\leq \norm{\scriptT_R}_{L^p\rightarrow L^q}\norm{f_\sigma}_{L^p(\R^3)} \\
&=\sigma^{-\frac{(\kappa_1+\kappa_2+1)}{p}}\norm{\scriptT_R}_{L^p\rightarrow L^q}\norm{f}_{L^p(\R^3)},
\end{align*}

and by looking at $f$ that are near-extremal, we get, for $\norm{\scriptT_R}_{L^p\rightarrow L^q}<\infty$,
\begin{align*}
\norm{\scriptT_{R_{\bm{\kappa},\sigma}}}_{L^p\rightarrow L^q}&=\sigma^{(\kappa_1+\kappa_2)[1+\frac{1}{q}-\frac{1}{p}]+[\frac{1}{q}-\frac{1}{p}]}\norm{\scriptT_R}_{L^p\rightarrow L^q}\\
&=\sigma^{\frac{1}{d_h}[1+\frac{1}{q}-\frac{1}{p}]+[\frac{1}{q}-\frac{1}{p}]}\norm{\scriptT_R}_{L^p\rightarrow L^q}
\\
&=\sigma^{\frac{d_h+1}{d_h}(\frac 1q-\frac 1p+\frac 1{d_h+1})}\norm{\scriptT_R}_{L^p\rightarrow L^q}.
\end{align*}

This proves Lemma \ref{scaling} for strong-type bounds. Replacing $\norm{\scriptT f}_{L^q(\R^3)}$ with $\norm{\scriptT f}_{L^q_w(\R^3)}$, and $f$ with $\chi_E$, we get an identical result for restricted-weak-type bounds. 
\end{proof}


\begin{prop}\label{scaling prop} Let $\bm{\kappa}=\bm{\kappa_\varphi}$. Let $\scriptR$ be $\bm{\kappa}$-scale invariant, and let $\psi$ be $\frac{\bm{\kappa}}{D}$-mixed homogeneous, some $D>0$. Suppose that $\frac 1{q_S}=\frac 1{p_S}-\frac 1{d_h+1}$ and that $(\frac 1{p_S},\frac 1{q_S})=(1-\theta)(\frac 1{p_0},\frac 1{q_0})+\theta(\frac 1{p_I},\frac 1{q_I})$, for some $\theta\in(0,1)$, for some $(p_0,q_0)$ and $(p_I,q_I)$ where $\frac 1{q_{0}}\neq \frac 1{p_{0}}-\frac 1{d_h+1}$ and $\frac 1{q_{I}}\neq \frac 1{p_{I}}-\frac 1{d_h+1}$.  If $\scriptT_{\scriptR\bigcap \{|\psi|\approx 1\}}$ is of rwt $(p_0,q_0)$ and $(p_I,q_I)$, then $\scriptT_\scriptR$ is of rwt $(p_S,q_S)$. 
\end{prop}

\begin{proof} It suffices to consider $\frac{1}{q_0}<\frac{1}{p_0}-\frac 1{d_h+1}$ and $\frac{1}{q_I}>\frac 1{p_I}-\frac 1{d_h+1}$. Let $\scriptT_j:=\scriptT_{\scriptR\bigcap\{|\psi|^{-1}([2^{jD},2^{(j+1)D}))\}}$. Then $\scriptT_j=\scriptT_{\scriptR\bigcap\{|\psi|^{-1}([1,2^D))\}_{\bm{\kappa},2^j}}$, by the scale-invariance of $\scriptR$ and the mixed-homogeneity of $\psi$. Thus, by Lemma \ref{scaling},
$$
\scriptT_j(E,F)\lesssim \min\{2^{j\frac{d_h+1}{d_h}(\frac 1{q_0}-\frac 1{p_0}+\frac 1{d_h+1})}|E|^\frac 1{p_0}|F|^{1-\frac 1{q_0}},2^{j\frac{d_h+1}{d_h}(\frac 1{q_I}-\frac 1{p_I}+\frac 1{d_h+1})}|E|^\frac 1{p_I}|F|^{1-\frac 1{q_I}}  \}.
$$
Then, interpolating,
\begin{align*}
\scriptT_\scriptR(E,F)&\lesssim \sum_j\min\{2^{j\frac{d_h+1}{d_h}(\frac 1{q_0}-\frac 1{p_0}+\frac 1{d_h+1})}|E|^\frac 1{p_0}|F|^{1-\frac 1{q_0}},2^{j\frac{d_h+1}{d_h}(\frac 1{q_I}-\frac 1{p_I}+\frac 1{d_h+1})}|E|^\frac 1{p_I}|F|^{1-\frac 1{q_I}}  \}
\\
&\approx |E|^\frac{(\frac 1{p_0}-\frac{1}{p_I})\frac 1{d_h+1}+\frac 1{p_0q_I}-\frac 1{q_0p_I}}{(\frac 1{q_I}-\frac 1{p_I})-(\frac 1{q_0}-\frac 1{p_0})}|F|^{1-\frac{(\frac 1{q_0}-\frac 1{q_I})\frac 1{d_h+1}+\frac 1{q_Ip_0}-\frac 1{q_0p_I}}{(\frac 1{q_I}-\frac 1{p_I})-(\frac 1{q_0}-\frac 1{p_0})}}.
\end{align*}
Thus, $\scriptT_\scriptR$ is of rwt $(p_S,q_S)$ for
\begin{equation}\label{scalingpoint}
(\tfrac 1{p_S},\tfrac 1{q_S})=( \tfrac{(\frac 1{p_0}-\frac{1}{p_I})\frac 1{d_h+1}+\frac 1{p_0q_I}-\frac 1{q_0p_I}}{(\frac 1{q_I}-\frac 1{p_I})-(\frac 1{q_0}-\frac 1{p_0})},\tfrac{(\frac 1{q_0}-\frac 1{q_I})\frac 1{d_h+1}+\frac 1{q_Ip_0}-\frac 1{q_0p_I}}{(\frac 1{q_I}-\frac 1{p_I})-(\frac 1{q_0}-\frac 1{p_0})} ),
\end{equation}
 the point on the scaling line $\frac 1q=\frac 1p-\frac 1{d_h+1}$ directly between $(\frac 1{p_0},\frac 1{q_0})$ and $(\frac 1{p_I},\frac 1{q_I})$.
\end{proof}

There are three cases where we will use this result:
\\
\\
\underline{Case 1}: $(\frac 1{p_0},\frac 1{q_0})=(\frac 34,\frac 14)$, and $(\frac 1{p_I},\frac 1{q_I})$ satisfies $\frac 1{p_I}=\frac 3{q_I}$.

\vspace{.5pc}

In this case, the point $(\frac 1{p_S},\frac 1{q_S})$ is the intersection of the lines $\frac 1p=\frac 3q$ and scaling line $\frac 1q=\frac 1p-\frac 1{d_h+1}$, namely $(\frac 1{p_S},\frac 1{q_S})=(\frac 3{2d_h+2},\frac 1{2d_h+2})=(\frac 3{d_\omega+4},\frac 1{d_\omega+4})$, using the fact that $d_\omega=2d_h-2$. 
\\
\\
\underline{Case 2}: $(\frac 1{p_0},\frac 1{q_0})=(\frac 23,\frac 13)$, and $(\frac 1{p_I},\frac 1{q_I})$ satisfies $\frac 1{p_I}=\frac 2{q_I}$.

\vspace{.5pc}

In this case, the point $(\frac 1{p_S},\frac 1{q_S})$ is the intersection of the lines $\frac 1p=\frac 2q$ and scaling line $\frac 1q=\frac 1p-\frac 1{d_h+1}$, namely $(\frac 1{p_S},\frac 1{q_S})=(\frac 2{d_h+1},\frac 1{d_h+1})$. 
\\
\\
\underline{Case 3}: $(\frac 1{p_0},\frac 1{q_0})=(\frac 34,\frac 14)$, and $(\frac 1{p_I},\frac 1{q_I})$ satisfies $\frac 1{p_I}=\frac 2{q_I}$.

\vspace{.75pc}

Here, by simplifying \eqref{scalingpoint}, $(\frac 1{p_S},\frac 1{q_S})$ simplifies to 
$$
\big(\tfrac{2}{p_S},\tfrac 2{q_S}\big)=\big(\tfrac{3-\frac 8{q_I}+\frac 1{q_I}(d_h+1)}{(1-\frac 2{q_I})(d_h+1)},\tfrac{1-\frac 4{q_I}+\frac 1{q_I}(d_h+1)}{(1-\frac 2{q_I})(d_h+1)}\big).
$$
Additionally, using identity $d_\omega=2d_h-2$, we can rewrite $\frac 1{p_S}$ as
\begin{equation}\label{Scaling p}
\tfrac 1{p_S}=\tfrac{3-\frac 8{q_I}+\frac 12\frac 1{q_I}(d_\omega+4)}{(1-\frac 2{q_I})(d_\omega+4)}.
\end{equation}

\section{Decomposition}\label{S:Decomposition}

Our goal in this section will be to decompose $[-1,1]^2$, and hence our operator $\scriptT_{[-1,1]^2}$, around the factors of the determinant Hessian $\omega$. Since $\omega$ is mixed homogeneous, with homogeneous distance $d_\omega=2d_h-2$, 
$$
\omega(z_1, z_2)=Cz_1^{\nu_1}z_2^{\nu_2} \prod_{j=3}^{M_2} (z_2^s-\lambda_jz_1^r)^{n_j}, \text{ with } \lambda_j \text{ real iff } j\leq M_1, \text{ some } M_1\leq M_2.
$$
 Let $\tilde{\epsilon}> 0$ be sufficiently small. We use the following covering of $[-1,1]^2$:
 
 \vspace{.5pc}
 
 $R_j:=[-1,1]^2\cap \{|z_2^s-\lambda_j z_1^r|<\tilde{\epsilon} |z_1|^r\}, j=3,...,M_1$ (each real $\lambda_j$);
 
 \vspace{.5pc} 
 
 $R_2 := [-1,1]^2\cap \{|z_2|^s <\tilde{\epsilon} |z_1|^r\}$ if $\nu_2\neq 0$, and $R_2:=\emptyset$ otherwise;
 
 \vspace{.5pc}
 
 $R_1 := [-1,1]^2\cap \{|z_1|^r <\tilde{\epsilon} |z_2|^s\}$ if $\nu_1\neq 0$, and $R_1:=\emptyset$ otherwise;
 
 \vspace{.5pc}
 
 $R_0:=[-1,1]^2\backslash \{\bigcup_{j=1}^{j=M_1} R_j\}$.
 
 \vspace{.5pc}
 
Sometimes, especially when we have a vertex on the scaling line $\frac 1q=\frac 1p-\frac 1{d_h+1}$, it will be more useful to have this decomposition extended to all of $\R^2$:

\vspace{.5pc}
 
 $R_j^e:=\{|z_2^s-\lambda_j z_1^r|<\tilde{\epsilon} |z_1|^r\}, j=3,...,M_1$ (each real $\lambda_j$);
 
 \vspace{.5pc} 
 
 $R_2^e :=\{|z_2|^s <\tilde{\epsilon} |z_1|^r\}$ if $\nu_2\neq 0$, and $R_2^e:=\emptyset$ otherwise;
 
 \vspace{.5pc}
 
 $R_1^e :=\{|z_1|^r <\tilde{\epsilon} |z_2|^s\}$ if $\nu_1\neq 0$, and $R_1^e:=\emptyset$ otherwise.
 
\vspace{.5pc}

These $R_j^e$ are $\bm{\kappa_\varphi}$-scale invariant, which will allow us to apply Proposition \ref{scaling prop}. By Proposition \ref{dpsi}, when $T>d_\omega$ there exists a unique index $j_0$ such that either $j_0\in\{1,2\}$ and $\nu_{j_0}>d_\omega$ or $j_0\geq 3$ and $n_{j_0}>d_\omega$. We set $$R_T:=R_{j_0} \quad \text{ and } \quad R_T^e:=R_{j_0}^e.$$
when $T>d_\omega$. When $T\leq d_\omega$, we set $R_T:=\emptyset$ and $R_T^e:=\emptyset$.


\section{First relevant vertex, Initial observations}\label{S:First Relevant Vertex}

\begin{definition} A \textit{relevant vertex} is a vertex $(\frac 1p,\frac 1q)$ of the polygon described in Theorem \ref{Main Theorem Alt} that satisfies $q'\leq p<q$.
\end{definition}
To prove Theorem \ref{Main Theorem Alt}, by Young's Inequality, duality, and real interpolation, it suffices to prove that $\scriptT(E,F)\lesssim |E|^\frac 1p|F|^\frac 1{q'}$ for all relevant vertices. 
 
We recall from Section \ref{S:Outline} that we can divide our problem into three broad cases: the case $T\leq d_\omega$, the rectangular cases, and the twisted cases, and further decompose the rectangular cases into Cases ($\nu$), (A), and (N). 
 
\begin{lemma}\label{relevant vertices}
The polygon described in Theorem $\ref{Main Theorem Alt}$ has exactly one relevant vertex, denoted $(\frac 1{p_{v_1}},\frac 1{q_{v_1}})$, that lies on the line $q=3p$. If we are in the twisted cases, case $(\nu)$, or if $T\leq d_\omega$, this is the only relevant vertex. In Cases $(N)$ and $(A)$, there exists exactly one additional relevant vertex $(\frac 1{p_{v_2}},\frac 1{q_{v_2}})$, and this additional vertex satisfies $(\frac 1{p_{v_2}},\frac 1{q_{v_2}})\in \overline{Conv}\{(0,0),(\frac 34,\frac 14),(\frac 23,\frac 13)\}$. 
\end{lemma}
This lemma follows by simple algebra with the boundaries in Theorem \ref{Main Theorem Alt}. For the rest of Sections \ref{S:First Relevant Vertex}-\ref{S:Twisted Cases}, we will focus on the first relevant vertex.

\begin{lemma}\label{Vertex1location} For the polygon described in Theorem \ref{Main Theorem Alt}, $(\frac{1}{p_{v_1}},\frac{1}{q_{v_1}})=(\frac 3{d_\omega+4},\frac 1{d_\omega+4})$ in the twisted cases and case $T\leq d_\omega$. In the rectangular cases, $(\frac{1}{p_{v_1}},\frac{1}{q_{v_1}})=(\frac 3{T+4},\frac 1{T+4})$.
\end{lemma}
\begin{proof} In the case $T\leq d_\omega$ and the twisted cases, $(p_{v_1},q_{v_1})$ lies on the intersection of the curves $q=3p$ and $\frac 1q=\frac 1p-\frac 1{d_h+1}$, and since $d_\omega=2d_h-2$,
$$
(\tfrac 1{p_{v_1}},\tfrac{1}{q_{v_1}})=(\tfrac{3}{2d_h+2},\tfrac{1}{2d_h+2})=(\tfrac{3}{d_\omega+4},\tfrac{1}{d_\omega+4}),
$$
In Case($\nu$), $(p_{v_1},q_{v_1})$ lies on the intersection of the curves $q=3p$ and $\frac 1q=\frac 1p-\frac 1{\nu+1}$, and since $T=2\nu-2$,
$$
(\tfrac 1{p_{v_1}},\tfrac{1}{q_{v_1}})=(\tfrac 3{2\nu+2},\tfrac{1}{2\nu+2})=(\tfrac 3{T+4},\tfrac{1}{T+4}),
$$
In Case(A), $(p_{v_1},q_{v_1})$ lies on the intersection of the curves $q=3p$ and $\frac 1q=\frac{A+1}{2A+1}\frac 1p-\frac 1{2A+1}$, and since $T=A-2$,
$$
(\tfrac 1{p_{v_1}},\tfrac 1{q_{v_1}})=(\tfrac 3{A+2},\tfrac 1{A+2})=(\tfrac 3{T+4},\tfrac 1{T+4}),
$$
In Case(N), $(p_{v_1},q_{v_1})$ lies on the intersection of the curves $q=3p$ and $\frac 1q=\frac{N+1}{N+2}\frac 1p-\frac 1{N+2}$, and since $T=2N-3$,
\begin{equation*}
(\tfrac 1{p_{v_1}},\tfrac{1}{q_{v_1}})=(\tfrac 3{2N+1},\tfrac 1{2N+1})=(\tfrac 3{T+4},\tfrac 1{T+4}).\qedhere
\end{equation*}
\end{proof}

This lemma leads to the following proposition, whose proof will occupy the next 3 sections:

\begin{proposition}\label{Vertex1Rectprop}
The operator $\scriptT_{R_j}$ is of rwt $(\frac{d_\omega+4}{3},d_\omega+4)$ when $j=0$ or when $n_j$ (likewise $\nu_j$) is less than or equal to $d_\omega$. When $n_j$ (likewise $\nu_j$) is greater than $d_\omega$, $\scriptT_{R_j}$ is of rwt $(\frac{T+4}{3},T+4)$.
\end{proposition}

Together, Proposition \ref{Vertex1Rectprop} and Lemma \ref{Vertex1location} imply the following corollary:

\begin{corollary}\label{Vertex1Rectcoro}
The operator $\scriptT_{[-1,1]^2\backslash R_T}$ is of rwt $(p,q)$ for $(\frac 1p,\frac 1q)$ lying in the polygon of Theorem \ref{Main Theorem Alt}. Additionally, in the rectangular cases, $\scriptT_{R_T}$ is of rwt $(p_{v_1},q_{v_1})$.
\end{corollary}

We will spend the remainder of this section proving Proposition \ref{Vertex1Rectprop} in all cases except when $n_j$ or $\nu_j$ equals $d_\omega$. The remaining cases will be handled in Sections \ref{S:Method of Refinements} and \ref{S: Conclusion of Degenerate Case}. 
\par

The following lemma will allow us to use H\"older's Inequality to compute $L^\infty\rightarrow L^\infty$ bounds, which will be useful in interpolation.
 
\begin{lemma}\label{Vertex1MeasureLemma}
Let $\mu$ be the standard Lebesgue measure on $\R^2$. The regions in the dyadic decomposition of $\omega$ satisfy the following inequalities: 
\begin{itemize}
\item $\mu(R_0\cap\{|\omega|\approx 2^{-m}\})\lesssim_{\tilde{\epsilon}} 2^{-\frac m{d_\omega}}. $
\item $\mu(R_i\cap\{|\omega|\approx 2^{-m}\})\lesssim_{\tilde{\epsilon}} 2^{-\frac m{\max(\nu_i,d_\omega)}} $ for $i=1,2$, whenever $\nu_i\neq d_\omega.$
\item $\mu(R_i\cap\{|\omega|\approx 2^{-m}\})\lesssim_{\tilde{\epsilon}} 2^{-\frac m{\max(n_i,d_\omega)}} $ for $3\leq i \leq M_1$, whenever $n_i\neq d_\omega.$
\end{itemize}
\end{lemma}
\begin{proof}

 On $R_0$, $|z_2-\lambda_j z_1^r|\sim |z_2|^s\sim |z_1|^r\sim |z_1|^r+|z_2|^s$ for all $j$, so $|\omega(z_1,z_2)|\gtrsim (|z_1|^r+|z_2|^s)^{\frac{(r+s)d_\omega}{rs}}$. Therefore, 
\begin{align*}
 \mu( R_0 \cap \{|\omega| &\approx 2^{-m}\})   \\ 
 &\lesssim_{\tilde{\epsilon}} \mu([-1,1]^2\cap\{(|z_1|^r+|z_2|^s)\leq 2^{-\frac{m}{d_\omega}\frac{rs}{r+s}} \})\lesssim_{\tilde{\epsilon}} 2^{-\frac{m}{d_\omega}}. 
\end{align*}

Next, on $R_2$, we have $|z_2|^s<\tilde{\epsilon} |z_1|^r$, so for $\tilde{\epsilon}$ sufficiently small, $|z_2^s-\lambda_iz_1^r|\sim  |z_1|^r$ for all $i$. Thus, $|\omega|\gtrsim |z_1|^Q|z_2|^{\nu_2}$, for some $Q$ satisfying $\frac{Qs+\nu_2r}{r+s}=d_\omega$. Therefore, after a few elementary calculus calculations,
\begin{align*}
\mu(R_2\cap \{|\omega|\approx 2^{-m}\})
&\lesssim_{\tilde\epsilon}  \mu( R_2 \cap \{ |z_1|^Q|z_2|^{\nu_2}\leq 2^{-m}\})
 \\
 &\lesssim_{\tilde\epsilon} \mu(\{|z_2|\leq \min(2^{-\frac m{\nu_2}}|z_1|^{-\frac Q{\nu_2}},\tilde\epsilon|z_1|^\frac{r}{s}),|z_1|\leq 1\})
\\
 &\lesssim_{\tilde\epsilon}\max{( 2^{-\frac{m}{\nu_2}}, 2^{-\frac{m(s+r)}{Qs+\nu_2r}}})= 2^{-\frac{m}{\max{(\nu_2, d_\omega)}}} 
\end{align*}
as long as $\nu_2\neq d_\omega$. The case for $R_1$ follows similarly.
 
Finally, on $R_j$, for $j\geq 3$, we have $|z_2^s-\lambda_j z_1^r|< \tilde{\epsilon} |z_1|^r$, so for $\tilde{\epsilon}$ sufficiently small, $|z_2^s-\lambda_i z_1^r|\sim |z_2^s|\sim |z_1^r|$ for all $i\neq j$. Thus, $|\omega|\gtrsim|z_1|^Q|z_2^s-\lambda_j z_1^r|^{n_j}$, for some $Q$ satisfying $\frac{Qs+n_jrs}{r+s}=d_\omega$. Therefore, after a few elementary calculus calculations,
\begin{align*}
 \mu( R_j \cap \{ |\omega|\approx 2^{-m}\})
&\lesssim_{\tilde\epsilon} \mu(R_j\cap\{ |z_1|^Q|z_2^s-\lambda_j z_1^r|^{n_j}\leq 2^{-m}\})
\\
&\lesssim_{\tilde\epsilon} \mu(\{|z_2^s-\lambda_j z_1^r|\leq \min(2^{-\frac m{n_j}}|z_1|^{-\frac Q{n_j}},\tilde\epsilon|z_1|^r),|z_1|\leq 1\})
\\
&\lesssim_{\tilde\epsilon}\max{( 2^{-\frac{m}{n_j}}, 2^{-\frac{m(s+r)}{Qs+n_jrs}}})=2^{-\frac{m}{\max{(n_j, d_\omega)}}}
\end{align*}
as long as $n_j\neq d_\omega$. 
\end{proof}

\begin{proposition}\label{Vertex1NotEqualConclus}
For $\tilde{\epsilon}$ sufficiently small, the operator $\scriptT_{R_j}$ is of rwt $(p,q)$ when $(\frac 1p,\frac 1q)$ equals:
\begin{itemize}
\item $(\frac 3{d_\omega+4},\frac 1{d_\omega+4})$ for $j=0$.
\item $(\frac 3{\max(\nu_j,d_\omega)+4},\frac 1{\max(\nu_j,d_\omega)+4})$ for $j=1,2$, and $\nu_j\neq d_\omega$.
\item $(\frac 3{\max(n_j, d_\omega)+4},\frac 1{\max(n_j, d_\omega)+4})$ for $3\leq j \leq M_1$, and $n_j\neq d_\omega$.
\end{itemize}
\end{proposition}
\begin{proof}
Define $\scriptT_{j,m}:=\scriptT_{R_j\bigcap\{|\omega|\approx 2^{-m}\}}$, and denote $h_j:=\max(\nu_j,d_\omega)$ for $j\in \{1,2\}$, $h_j:=\max(n_j,d_\omega)$ for $j\geq 3$, and $h_0:=d_\omega$. We assume that $n_j$ (likewise $\nu_j$) is not equal to $d_\omega$. By H\"older's Inequality and Lemma \ref{Vertex1MeasureLemma}, $\norm{\scriptT_{j,m}}_{\infty\rightarrow\infty}\lesssim 2^{-\frac{m}{h_j}}$. In addition, by Theorem \ref{Gressman}, $\norm{\scriptT_{j,m}}_{\frac 43\rightarrow 4}\lesssim 2^{\frac{m}{4}}$. Combining these,
$$
\scriptT_{R_j}(E,F)\lesssim \sum_m \min(2^{-\frac{m}{h_j}}|F|, 2^{\frac{m}{4}}|E|^\frac{3}{4}|F|^\frac{3}{4})\lesssim |E|^\frac{3}{h_j+4}|F|^{1-\frac{1}{h_j+4}}.
$$
Thus, $\scriptT_{R_j}$ is of rwt $(\frac{h_j+4}{3},\frac{h_j+4}{1})$.
\end{proof}
Proposition \ref{Vertex1NotEqualConclus} implies Proposition \ref{Vertex1Rectprop} in all cases except when $n_j$ or $\nu_j$ equals $d_\omega$. However, if we extend Lemma \ref{Vertex1MeasureLemma} to the remaining cases, the measures and thereby the $L^\infty \rightarrow L^\infty$ bounds will contain logarithmic terms, which are undesirable for the interpolation performed in Proposition \ref{Vertex1NotEqualConclus}. To avoid this, we will show in Sections \ref{S:Method of Refinements} and \ref{S: Conclusion of Degenerate Case} that for $q$ arbitrarily large, the rwt $(\frac q3,q)$ bounds lack the extra logarithmic term, by proving the following lemma:
\begin{lemma}\label{Vertex1Equalityintro} If $n_j$ or $\nu_j$ equals $d_\omega$, then $\scriptT_{R_j^e\cap\{|\omega|\approx 1\}}$ is of rwt $(\frac q3,q)$ for all $q$ satisfying $4\leq q<\infty$. 
\end{lemma}
By Proposition \ref{scaling prop} (we are in Case 1, as defined after the proof of that proposition), this lemma implies the following corollary:
\begin{corollary}\label{Vertex1Equalitycoro}
If $n_j$ or $\nu_j$ equals $d_\omega$, then $\scriptT_{R_j^e}$ is of rwt $(\frac {d_\omega+4}3, d_\omega+4)$.
\end{corollary}
Together, Proposition \ref{Vertex1NotEqualConclus} and Corollary \ref{Vertex1Equalitycoro} imply Proposition \ref{Vertex1Rectprop}. In Sections \ref{S:Method of Refinements} and \ref{S: Conclusion of Degenerate Case}, we will complete our argument by proving Lemma \ref{Vertex1Equalityintro}, with an argument that will have further applications in later sections.

\section{Application of Method of Refinements}\label{S:Method of Refinements} 

We begin our argument by looking at a general subset $S$ of $\R^2$, which is then subdivided into smaller subsets $\tau_n$. In light of Section \ref{S:First Relevant Vertex}, for now consider $S$ to be the set $R_j^e\cap\{|\omega|\approx 1\}$, where $n_j$ (likewise $\nu_j$) equals $d_\omega$, and consider the sets $\tau_n$ to be some sort of dyadic decomposition of $S$.
\begin{notation}\label{refinement notation}
Let $\tau_n$ be subsets of $\R^2$, let $S:=\bigcup_{n\in \N}\tau_n$, and denote $\scriptT_n:=\scriptT_{\tau_n}$. For any $E,F\subset \R^3$, to capture the primary contributors to $\scriptT_n(E,F)$, we will define the following refinements:
\begin{align*}
F_n&:=\{u\in F: \mathcal{T}_n\chi_E(u)\geq \tfrac{1}{4}\tfrac{\mathcal{T}_n(E,F)}{|F|}=\tfrac 14\text{avg}_{F}\scriptT_n\chi_E\};
\\
E_n&:=\{w\in E: \mathcal{T}_n^*\chi_{F_n}(w)\geq \tfrac{1}{4}\tfrac{\mathcal{T}_n(E,F_n)}{|E|}=\tfrac 14\text{avg}_E\scriptT^*_n\chi_{F_n}=:\alpha_{E_n} \};
\\
\tilde{F}_n&:=\{u\in F_n: \mathcal{T}_n\chi_{E_n}(u)\geq \tfrac{1}{4}\tfrac{\mathcal{T}_n(E_n,F_n)}{|F_n|}=\tfrac 14\text{avg}_{F_n}\scriptT_n\chi_{E_n}=:\alpha_{F_n} \};
\\
F_{kl}&:=\{u\in F_k: \mathcal{T}_k\chi_{E_k\cap E_l}(u)\geq \tfrac{1}{4}\tfrac{\scriptT_k(E_k\cap E_l,F_k)}{|F_k|}= \tfrac{1}{4}\text{avg}_{F_k}\scriptT_k\chi_{E_k\cap E_l}=:\beta_{F_{kl}} \};
\\
E_{kl}&:=\{w\in E_k: \mathcal{T}^*_k\chi_{\tilde F_k\cap \tilde F_l}(w)\geq \tfrac{1}{4}\tfrac{\scriptT_k(E_k,\tilde F_k\cap \tilde F_l)}{|E_k|}= \tfrac{1}{4}\text{avg}_{E_k}\scriptT_k^*\chi_{\tilde{F}_k\cap\tilde{F}_l}=:\beta_{E_{kl}} \}.
\end{align*}
With these refinements, 
\begin{align*}
\scriptT_n(E,F)&=\int_{F_n}\mathcal{T}_n\chi_E+\int_{F\backslash F_n}\mathcal{T}_n\chi_E 
\ \leq \ \scriptT_n(E,F_n) + \int_{F\backslash F_n}\tfrac{1}{4}\tfrac{\scriptT_n(E,F)}{|F|}
\\
&\leq \scriptT_n(E,F_n)+\tfrac{1}{4}\scriptT_n(E,F),
\end{align*}
and consequently $\mathcal{T}_n(E,F_n)\geq\frac{3}{4}\mathcal{T}_n(E,F)$. Likewise, $\mathcal{T}_n(E_n,F_n)\geq \frac{3}{4}\mathcal{T}_n(E,F_n)$ and $\mathcal{T}_n(E_n,\tilde F_n)\geq \frac{3}{4}\mathcal{T}_n(E_n,F_n)$, so 
\begin{equation}\label{refinement approx}
\scriptT_n(E_n,\tilde F_n)\approx  \mathcal{T}_n(E,F).
\end{equation}
\end{notation}
Therefore, $E_n$ and $\tilde F_n$ capture the primary contributions of $E$ and $F$ pertaining to $\scriptT_n$. Additionally, let $b\in \{2,3\}$ be fixed, and define $\eta(\epsilon) := \{n\in \mathbb{N}: \mathcal{T}_n(E,F) \approx \epsilon |E|^\frac{b}{b+1}|F|^\frac{b}{b+1}\}$, for $\epsilon\in 2^\Z$. (Note that $\eta$ depends implicitly on $b, E, F$.) Pertaining to Section \ref{S:First Relevant Vertex}, we will use $b=3$, but in a later section, we will reuse these arguments with $b=2$.






When $b=3$ and when $S=R_j\cap\{|\omega|\approx 1\}$, the following lemma implies Lemma $\ref{Vertex1Equalityintro}$ if all the hypotheses can be shown to be satisfied.

\begin{lemma}\label{refinement epsilon decomp}
Let $\norm{\scriptT_{n}}_{L^{\frac{b}{b+1},1}\rightarrow L^{\frac{1}{b+1},\infty}}\leq \Lambda_\varphi\lesssim 1$, some $\Lambda_\varphi$ depending only on $\varphi$, and let $|\tau_n|\lesssim 1$ uniformly in $n$.  If
$$
(1) \sum_{n\in \eta(\epsilon)}|E_n|\lesssim_{\delta}\epsilon^{-\delta}|E| \quad \text{and} \quad (2) \sum_{n\in \eta(\epsilon)}|\tilde F_n|\lesssim_{\delta}\epsilon^{-\delta}|F|
$$
for all $\delta>0$, uniformly over $E, F, \epsilon$, then $\norm{\scriptT_S}_{L^{\frac{\theta b}{b+1},1}\rightarrow L^{\frac{\theta}{b+1},\infty}}<\infty$ for all $\theta\in (0,1]$.
\end{lemma}
\begin{proof} In what follows, we will introduce an $a\in [0,1]$ to be chosen near the end. In the first line, we use $\norm{\scriptT_n}_{\frac b{b+1},\frac 1{b+1}}\lesssim 1$ to imply $\epsilon \lesssim 1$; the second line is due to \eqref{refinement approx}; the definition of $\eta(\epsilon)$, Young's Inequality using $|\tau_n|\lesssim 1$, and the hypothesis that $\norm{\scriptT_n}_{\frac b{b+1},\frac 1{b+1}}\lesssim 1$ give the third line; the fourth line is simple rearrangement; in the fifth line $a$ is chosen so that $1-a=\frac{b+1}{b+1+(b-1)\theta}$, and H\"{o}lder is applied; for the final line, conditions $(1)$ and $(2)$ give the first inequality, and a geometric sum with $\delta$ sufficiently small completes the argument:
\begin{align*}
\langle \mathcal{T}_S\chi_E, \chi_F\rangle &= \sum_{\epsilon\in 2^{-\N}}\sum_{n\in \eta(\epsilon)} \langle \mathcal{T}_n\chi_E, \chi_F\rangle
\\
&\approx \sum_{\epsilon 
\in 2^{-\N}}\sum_{n\in \eta(\epsilon)} \langle \mathcal{T}_n\chi_E, \chi_F\rangle^{a\theta}\langle \mathcal{T}_n\chi_E, \chi_F\rangle^{a(1-\theta)}\langle \mathcal{T}_n\chi_{E_n}, \chi_{\tilde F_n}\rangle^{1-a}
\\
&\lesssim \sum_{\epsilon 
\in 2^{-\N}}\sum_{n\in \eta(\epsilon)} (\epsilon |E|^\frac{b}{b+1}|F|^\frac{b}{b+1})^{a\theta}|F|^{a(1-\theta)} (|E_n|^\frac{b\theta}{b+1}|\tilde F_n|^{1-\frac{\theta}{b+1}})^{1-a} 
\\
&= \sum_{\epsilon\in 2^{-\N}}\epsilon^{a\theta}(|E|^\frac{b\theta}{b+1}|F|^{1-\frac{\theta}{b+1}})^a\sum_{n\in \eta(\epsilon)}(|E_n|^\frac{b\theta}{b+1}|\tilde F_n|^{(1-\frac{\theta}{b+1})})^{(1-a)} 
\\
&\leq \sum_{\epsilon\in 2^{-\N}}\epsilon^{a\theta}(|E|^\frac{b\theta}{b+1}|F|^{1-\frac{\theta}{b+1}})^a[(\sum_{n\in \eta(\epsilon)}|E_n|)^\frac{b\theta}{b+1}(\sum_{n\in \eta(\epsilon)}|\tilde F_n|)^{1-\frac{\theta}{b+1}}]^{(1-a)}
\\
&\lesssim \sum_{\epsilon\in 2^{-\N}}\epsilon^{a\theta-\delta(1+\theta \frac{b-1}{b+1})(1-a)}|E|^{\frac{b\theta}{b+1}}|F|^{1-\frac{\theta}{b+1}}\lesssim |E|^{\frac{b\theta}{b+1}}|F|^{1-\frac{\theta}{b+1}}
\end{align*}

\end{proof}


To use Lemma \ref{refinement epsilon decomp}, we will need to show that hypotheses (1) and (2) are satisfied. In Lemma \ref{refinement epsilon decomp 2}, we will prove that the following condition will suffice:
\begin{condition}\label{refinement condition Jac}
There exist $\tilde M, B>0$, $\xi\in \{-1,1\}$ independent of $E, F$ such that, for all $k,l\in \N$ satisfying $\max\{k,l,|k-l|\}>\tilde M$, with $sgn(k-l)=\xi$, we have $|F_l|^{b-1}\gtrsim 2^{|k-l|B}\alpha_{E_l}^b\beta_{F_{kl}}$ and $|E_l|^{b-1}\gtrsim 2^{|k-l|B}\alpha_{F_l}^b\beta_{E_{kl}}$, with implicit constants independent of $ l, k, E, F$.
\end{condition}
\begin{lemma}\label{refinement epsilon decomp 2}
Let all conditions of Lemma \ref{refinement epsilon decomp} be satisfied save for (1) and (2). If Condition \ref{refinement condition Jac} is satisfied, then (1) and (2) are also satisfied.
\end{lemma}

\begin{proof} (Previously done in \cite{CDSS}.) Since $\norm{\scriptT_{n}}_{L^{\frac{b}{b+1},1}\rightarrow L^{\frac{1}{b+1},\infty}}\leq \Lambda_\varphi\lesssim 1$ uniformly, then by $\eqref{refinement approx}$, 
$$
|E_n|^\frac{b}{b+1}|F|^\frac{b}{b+1}\gtrsim \scriptT_n(E_n,F)\approx \scriptT_n(E,F)\approx \epsilon |E|^\frac{b}{b+1}|F|^\frac{b}{b+1},
$$
so $|E_n|\gtrsim \epsilon^{\frac{b+1}{b}}|E|$ uniformly. By a similar argument, $|\tilde F_n|\gtrsim \epsilon^{\frac{b+1}{b}}|F|$.

\vspace{.5pc}

Suppose condition (1) fails. Then, for $M>\tilde M$ arbitrarily large, there exist $E,F, \epsilon\leq \Lambda_\varphi$ such that $\sum_{n\in \eta(\epsilon)\cap(\tilde M,\infty)}|E_n|> 10M\log(10\Lambda_\varphi\epsilon^{-1})|E|$. By the pigeonhole principal, there exists a finite, $M\log(10\Lambda_\varphi\epsilon^{-1})$-separated set $\eta'\subset \eta(\epsilon)\cap (\tilde M,\infty)$ such that $\sum_{n\in \eta'}|E_n|>10|E|$. Then by H\"older's inequality,
\begin{align*}
\sum_{n\in \eta'}|E_n|&=\int_E\sum_{n\in \eta'}\chi_{E_n}\leq |E|^\frac{1}{2}(\int_E (\sum_{n\in \eta'}\chi_{E_n})^2)^\frac{1}{2}
\\
&=|E|^\frac{1}{2}(\sum_{n\in \eta'}|E_n|+\sum_{n\neq m\in \eta'}2|E_n\cap E_m|)^\frac{1}{2}.
\end{align*}
Since $\sum_{n\in \eta'}|E_n|\geq\sum_{n\neq m\in \eta'}|E_n\cap E_m|$ would imply $\sum_{n\in\eta'}|E_n|\leq 3^\frac{1}{2}|E|^\frac{1}{2}(\sum_{n\in\eta'}|E_n|)^\frac{1}{2}$ and contradict $\sum_{n\in \eta'}|E_n|>10|E|$, then $\sum_{n\in  \eta'}|E_n|<\sum_{n\neq m\in \eta'}|E_n\cap E_m|$,
and therefore $\sum_{n\in \eta'}|E_n|\leq 3^{\frac 12}|E|^\frac{1}{2}(\sum_{n\neq m\in \eta'}|E_n\cap E_m|)^\frac{1}{2}$.
 
 \vspace{1pc}
 
Then, choosing $n_1, k, l\in \eta'$, $k\neq l$ so that $|E_{n_1}|=\min_{n\in \eta'} |E_n|$ and $|E_k\cap E_{l}|=\max_{n\neq m\in \eta'}|E_n\cap E_m|$, we have $(\#\eta')|E_{n_1}|\leq 3^\frac 12|E|^\frac{1}{2}(\#\eta')^{2\frac{1}{2}}|E_k\cap E_l|^\frac{1}{2}$, implying
\begin{equation}\label{refinement Ek El}
|E_k\cap E_l|\geq \tfrac{|E_{n_1}|^2}{3|E|} \gtrsim (\epsilon^\frac{b+1}{b})^2\tfrac{|E|^2}{|E|}=\epsilon^\frac{2b+2}{b}|E|.
\end{equation}
Since $|k-l|>10M\log(10\Lambda_\varphi \epsilon^{-1})>M$, then by Condition \ref{refinement condition Jac},
$$
|F_l|^{b-1}\gtrsim (10\Lambda_\varphi \epsilon^{-1})^{MB}\alpha_{E_l}^b\beta_{F_{kl}}.
$$
To simplify this, by the definitions of  $\scriptT(\cdot,\cdot)$ and $E_k$, \eqref{refinement approx}, and \eqref{refinement Ek El},
\begin{align*}
\mathcal{T}_k(E_k\cap E_l,F_k) &= \int_{E_k\cap E_l}\mathcal{T}_k^*\chi_{F_k}(w)dw\gtrsim \int_{E_k\cap E_l} \tfrac{\mathcal{T}_k(E,F_k)}{|E|} 
\\
&\approx  \mathcal{T}_k(E,F) \tfrac{|E_k\cap E_l|}{|E|} \gtrsim \epsilon^\frac{2b+2}{b} \mathcal{T}_k(E,F),
\end{align*}

so $\beta_{F_{kl}}=\frac{1}{4}\frac{\mathcal{T}_k(E_k\cap E_l,F_k)}{|F_k|}\gtrsim\epsilon^\frac{2b+2}{b}\frac{\mathcal{T}_k(E,F)}{|F_k|} \quad$ and $\alpha_{E_l}=\frac{1}{4}\frac{\mathcal{T}_l(E,F_l)}{|E|}\approx \frac{\mathcal{T}_l(E,F)}{|E|}$.

\vspace{1pc}

Putting these together,
$$
|F_l|^{b-1}\gtrsim (10\Lambda_\varphi\epsilon^{-1})^{MB-\frac{2b+2}{b}}\tfrac{\mathcal{T}_k(E,F)\mathcal{T}_l(E,F)^b}{|E|^b|F_k|} \gtrsim (10\Lambda_\varphi\epsilon^{-1})^{MB-\frac{2b+2}{b}-b-1}\tfrac{|E|^b|F|^b}{|E|^b|F_k|}.
$$

 Consequently, $|F_l|^{b-1}|F_m|\gtrsim (10\Lambda_\varphi\epsilon^{-1})^{MB-\frac{2b+2}{b}-b-1}|F|^b$, which for sufficiently large $M$ implies a contradiction, since $10\Lambda_\varphi \epsilon^{-1}\geq 10$ and $F_n\subset F$ for all $n$.

 And if (2) were false, by replacing each $F_n$ with $E_n$, $E_n$ with $\tilde F_n$, $F_{kl}$ with $E_{kl}$, and swapping $E$ and $F$,  the proof plays out identically.
\end{proof}
Now, we can finally move on to the main proposition.

\begin{prop}\label{refinementconclusion1}  The operator $\mathcal{T}_{R_j^e\cap \{|\omega|\in [1,2]\}}$ is of rwt $(\frac q3,q)$ for all $q\in [4,\infty)$ if the following hold:

\vspace{.25pc}

0) We can decompose $R_j^e\cap \{|\omega|\in [1,2]\}$ into a finite number of subsets $S\in \scriptS$ such that for each $S$, $S$ is either bounded or the following hold:

\vspace{.25pc}

1) $S$ can be broken into regions $\tau_n$, where $n$ belongs to some subset of $\mathbb{N}$, and where $|\tau_n|\lesssim 1$ for each $n$.

\vspace{.25pc}

2) On each $\tau_n$, $|\partial_{z_{i_1}} \varphi|\approx 2^{nD}$, and $|\partial_{z_{i_2}} \varphi|\approx 2^{nK}$, for some fixed $D, K\in \mathbb{R}\backslash \{0\}$, $K\geq D$, that may depend on $S$, where $\{i_1,i_2\}$ is some permutation of $\{1,2\}$. 
\end{prop}

\begin{proof} If $S$ is bounded, then Young's Inequality, Theorem \ref{Gressman} (since $|\omega|\approx 1$), and interpolation give us the result directly. For $S$ not bounded, we will use the notation of Notation \ref{refinement notation}, and set $b=3$ (for the purposes of Lemmas \ref{refinement epsilon decomp}-\ref{refinement epsilon decomp 2}). By our hypotheses, $|\tau_n|\lesssim 1$, and since $|\omega|\approx 1$ on $S$, then by Theorem \ref{Gressman}, $\norm{\mathcal{T}_n}_{L^{\frac 43}\rightarrow L^{4}}\lesssim 1$ uniformly in $n$. Thus, by Lemma \ref{refinement epsilon decomp} and Lemma \ref{refinement epsilon decomp 2}, it suffices to prove the following version of Condition \ref{refinement condition Jac}:

\begin{lemma}\label{refinement Jacobian 1}
Let $k,l\in \N$, and let $k$, $l$, and $|k-l|$ be sufficiently large, independent of $E$ and $F$. Furthermore, let $sgn(k-l)=sgn(D)$. Then hypotheses $(1)$ and $(2)$ of Proposition \ref{refinementconclusion1} imply that $|F_l|^2\gtrsim 2^{|k-l||D|}\alpha_{E_l}^3\beta_{F_{kl}}$ and $|E_l|^2\gtrsim 2^{|k-l||D|}\alpha_{F_l}^3\beta_{E_{kl}}$, with implicit constants independent of $ l, k, E, F$.
\end{lemma}

\begin{proof}

By symmetry, it will suffice to prove the $F_l$ inequality. Fix $u_0\in F_{kl}$. Define
$$
\Omega_1:=\{t\in \tau_k: u_0-(t,\varphi(t))=:w(t)\in E_k\cap E_l\}.$$
Then $|\Omega_1|=\mathcal{T}_k\chi_{E_k\cap E_l}(u_0)\geq \beta_{F_{kl}}$. For $t\in \Omega_1$, define
$$
\Omega_2(t):=\{s\in \tau_l: w(t)+(s,\varphi(s))\in F_l\}.$$
Then $|\Omega_2(t)|=T_l^*\chi_{F_l}(w(t))\geq \alpha_{E_l}$. Finally, we define $\Omega\subset \R^6$ and $\Psi:\Omega\rightarrow \R^6$:
$$\Omega :=\{(t,s_1,s_2)\in \R^6: t\in \Omega_1, s_i\in \Omega_2(t), i=1,2\};$$
$$ \Psi(t,s_1,s_2):=u_0-(t,\varphi(t))+(s_i,\varphi(s_i)): i=1,2.$$
Since $\Psi$ is a polynomial mapping $\Omega\subset\R^6$ into $\R^6$ with $\det D\Psi\not \equiv 0$, this map is $\scriptO(1)$-to-one off a set of measure zero. Since $\Psi(\Omega)\subset F_l\times F_l$, then $|F_l|^2\gtrsim \int_\Omega |\det D\Psi(t_1,s_1,s_2)|dt_1ds_1ds_2$. Defining $G(t,s_i):=\nabla\varphi(t)-\nabla\varphi(s_i)$, we can expand the Jacobian as follows:
\begin{align*}
|\det D\Psi(t;s_1,s_2)|
&= |\det(G(t,s_1), G(t,s_2))|
\\
&\approx |G(t,s_1)|\dist(G(t,s_2), \langle G(t,s_1)\rangle),
\end{align*}
which implies that $|F_l|^2 \gtrsim \int_\Omega |G(t,s_1)|dist(G(t,s_2), \langle G(t,s_1)\rangle)dtds_1ds_2$.

\vspace{.5pc}

By hypothesis (2) of Proposition \ref{refinementconclusion1}, after a possible reordering we can assume that that  $|\partial_{z_1}\varphi|\approx 2^{mK}$ and $|\partial_{z_2}\varphi|\approx 2^{mD}$ on $\tau_m$, for each $m$.  Additionally by our hypothesis, $sgn(k-l)=sgn(D)$, so $2^{kD}\geq 2^{lD}$ and $(k-l)D=|k-l||D|$.

Since $K\neq 0$ and $K\geq D$, and since $k,l,|k-l|$ are all sufficiently large, $|G(t,s_1)|\approx|\pi_1(G(t,s_1))| \approx \max(2^{kK},2^{lK})$. Here we recall that $t\in \tau_k$ and $s_i\in \tau_l$ for $i=1,2$.

\vspace{.5pc}

Fix $t, s_1$, and denote $\overline{G}_{1,t,s_1}:=G_1(t,s_1)$, which will be fixed, and $\overline{G}_{2,t}(s_2):=G_2(t,s_2)$. Then the map $s_2\mapsto \overline{G}_{2,t}(s_2)$ has a $1^{st}$ derivative comparable to $|\omega|\approx 1$, so that $|\overline{G}_{2,t}(\Omega_2(t))|\gtrsim \alpha_{E_l}$ and 
$$
\int_{\Omega_2(t)} \dist(G(t,s_2), \langle G(t,s_1)\rangle)ds_2\approx \int_{\overline{G}_{2,t}(\Omega_2(t))}\dist(\overline{G}_{2,t}, \langle \overline{G}_{1,t,s_1}\rangle)d\overline{G}_{2,t}.
$$
Since  $\pi_2(|\overline{G}_{1,t,s_1}-\overline{G}_{2,t}|)=\pi_2(|\nabla\varphi(s_2)-\nabla\varphi(s_1)|)\lesssim 2^{lD}$, then $\overline{G}_{2,t}(\Omega(t))\subset \R\times [\overline{G}_{1,t,s_1}-C2^{lD},\overline{G}_{1,t,s_1}+C2^{lD}]$ for some $C$ sufficiently large. However, $slope(\langle\overline{G}_{1,t,s_1}\rangle)\approx \frac{2^{kD}}{\max(2^{kK},2^{lK})}\lesssim 1$, so if we define $\gamma:=\alpha_{E_l}\frac{2^{|k-l||D|}}{\max(2^{kK},2^{lK})}$, then $|\scriptN_\gamma(\langle\overline{G}_{1,t,s_1}\rangle)\cap \R\times [\overline{G}_{1,t,s_1}-C2^{lD},\overline{G}_{1,t,s_1}+C2^{lD}]|\lesssim \gamma \frac{\max(2^{kK},2^{lK})}{2^{kD}}2^{lD}=\alpha_{E_l}$. Since $|\overline{G}_{2,t}(\Omega_2(t))|\gtrsim \alpha_{E_l}$, then for some sufficiently small $c$, $|\overline{G}_{2,t}(\Omega(t))\backslash \scriptN_{c\gamma}(\langle\overline{G}_{1,t,s_1}\rangle)|\geq \frac 12 |\overline{G}_{2,t}(\Omega(t))|$. Therefore,
$$
\int_{\overline{G}_{2,t}(\Omega_2(t))}\dist(\langle \overline{G}_{1,t,s_1}\rangle, \overline{G}_{2,t})d\overline{G}_{2,t}\gtrsim \gamma|\overline{G}_{2,t}(\Omega(t))|\gtrsim \tfrac{\alpha_{E_l}^2 2^{|k-l||D|}}{\max\{2^{kK},2^{lK}\}}.
$$
Thus,  $|F_l|^2\gtrsim \int_{\Omega_1}\int_{\Omega_2(t)} \max\{2^{kK},2^{lK}\}\frac{\alpha_{E_l}^2 2^{|k-l||D|}}{\max\{2^{kK},2^{lK}\}}ds_1dt\gtrsim 2^{|k-l||D|}\alpha_{E_l}^3\beta_{F_{kl}}$.

A near-identical argument gives $|E_l|^2\gtrsim 2^{|k-l||D|}\alpha_{F_l}^3\beta_{E_k}$, by swapping $E_n$ and $F_n$, and using $
\Psi(t,s_1,s_2)=w_0+(t,\varphi(t))-(s_i,\varphi(s_i)): i=1,2$, with $w_0\in E_{kl}$. This completes the proof of Lemma \ref{refinement Jacobian 1}.
\end{proof}
\noindent Hence, by Lemmas \ref{refinement epsilon decomp} and \ref{refinement epsilon decomp 2}, with $b=3$, the proof of Proposition \ref{refinementconclusion1} is complete.
\end{proof}


\section{Conclusion of case $T=d_\omega$}\label{S: Conclusion of Degenerate Case}
 
\begin{proposition} If $\nu_j$ (likewise $n_j$) equals $d_\omega$, then $\varphi$ satisfies the hypotheses of Proposition \ref{refinementconclusion1}.
\end{proposition}

\begin{proof} First, consider the case $j\geq 3$. If $n_j\geq d_\omega$, part (1) of Proposition \ref{dpsi} implies that $\min(r,s)=1$. By symmetry, we may then choose $s=1$ and perform the following change of variables over $R_j^e$: $x=z_1$ and $y=z_2-\lambda_j z_1^r$. Then 
\begin{align}\label{varchange}
\partial_{z_1}&=\partial_x-\lambda_j rx^{r-1}\partial_{y};\quad \quad\partial_{z_2}=\partial_{y};
\\
\partial_{z_2}^2&=\partial_{y}^2; \quad \quad\quad\partial_{z_1z_2}=\partial_{xy}-\lambda_jrx^{r-1}\partial_{y}^2; \nonumber
\\
\partial_{z_1}^2&=\partial_{x}^2+\lambda_j^2r^2x^{2r-2}\partial_y^2-2\lambda_j rx^{r-1}\partial_{xy}-\lambda r(r-1)x^{r-2}\partial_y. \nonumber
\end{align}

And finally the determinant Hessian $\omega=det D^2\varphi$ satisfies
\begin{equation}\label{Curve Curvature}
\omega=[\partial_x^2\varphi\partial_y^2\varphi-(\partial_{xy}\varphi)^2]-\lambda_jr(r-1)x^{r-2}\partial_y\varphi\partial_{y}^2\varphi.
\end{equation}
The relation $d_\omega=T$, with $y$ being a factor of $\omega$ of multiplicity $T$, requires that $\omega=Cx^Ty^T+\scriptO(y^{T+1})$, so in the region $R_j^e\cap \{|\omega|\approx 1\}$, $|x|\approx |y|^{-1}$. Similarly, when $j=1,2$, $d_\omega=T$ implies $\omega= Cz_1^Tz_2^T+\scriptO(z_2^{T+1})$ after possibly swapping $z_1$ and $z_2$, so with the choice $x=z_1$, $y=z_2$, the last sentence also holds for $j=1,2$.
\par 
By symmetry, it suffices to consider cases with $\lambda_j\geq 0$ and regions with $z_1,z_2\geq 0$. Additionally, we will focus on $y>0$, as the proof in the region $y<0$ is essentially identical. 
\par 
Next, we decompose $R_j^e\cap \{|\omega|\approx 1\}\cap\{z_1,z_2,z_2-\lambda_jz_1>0\}$ into $\tau_n=\{x\approx 2^n, y\approx 2^{-n}\}$, with $n=0,1,2,...$ By hypothesis (0) of Proposition \ref{refinementconclusion1}, it suffices to consider only $\tau_n$ where $n$ is sufficiently large, which will allow us to make the asymptotics in the next paragraph work. Each $\tau_n$ satisfies $|\tau_n|\approx 1$, satisfying the first condition of Proposition  \ref{refinementconclusion1}. 
\par 
The partial derivative $\partial_{z_1}\varphi$ satisfies $\partial_{z_1}\varphi=Cx^Jy^L+\scriptO(y^{L+1})\approx 2^{(J-L)n}$ for some $J,L\in\N_0$. Thus $\partial_{z_1}\varphi$ satisfies the second condition of Proposition \ref{refinementconclusion1} as long as $J\neq L$. Suppose $J=L$. Then $d_h(\partial_{z_1}\varphi)=\frac{Js+Jr}{s+r}=J$, and since $T=d_\omega$, then by Lemma \ref{domega dh}, $d_h(\partial_{z_1}\varphi)=d_h-\frac{s}{r+s}=\frac{d_\omega+2}{2}-\frac{s}{r+s}=\frac{T+2}{2}-\frac{s}{r+s}$. Combining these equations, $\frac{2s}{r+s}=T+2-2J\in \mathbb{N}$, so for some $m\in \mathbb{N}$, $2s=m(r+s)$, so $(2-m)s=mr$. This can only be satisfied if $m=1$ and $s=r$. A similar argument leads to an identical result for $\partial_{z_2}\varphi$. Hence, except for the homogeneous case where $s=r=1$, we have $J\neq L$, and the final condition of the Proposition \ref{refinementconclusion1} is satisfied.
\par 
Next, consider the homogeneous case. All factors are linear, so up to a linear transformation, $z_2^T$ is a factor of $\omega$.  We can decompose our options for $\varphi$ in the following way:
\begin{enumerate}
\item $\varphi=z_1^{M}z_2^{\tilde\nu}+\scriptO(z_2^{\tilde\nu+1})$, with $\tilde\nu\geq 2$;  
\item $\varphi=z_1^{M}z_2+\scriptO(z_2^{2})$; 
\item $\varphi=z_1^M+\scriptO(z_2^2)$; 
\item $\varphi=z_1^M+cz_1^{M-1}z_2+\scriptO(z_2^2)$.
\end{enumerate}

\vspace{.5pc}

By Lemma \ref{2nu-2}, option 2 would imply $T=0$, so this case is trivial.

\vspace{.5pc}

In option 1, $\omega= Cz_1^{2M-2}z_2^{2\tilde\nu-2}+\scriptO(z_2^{2\tilde\nu-1})$ by Lemma \ref{2nu-2}, and since $d_\omega=T$, it follows that $2M-2=2\nu-2=T$, implying that $M=\nu$. Then, $\nabla \varphi\approx (z_1^{\tilde\nu-1}z_2^{\tilde\nu},z_1^{\tilde\nu} z_2^{\tilde\nu-1})\approx (2^{-n},2^n)$ on $\tau_n$, so the lemma conditions are satisfied.

\vspace{.5pc}
In option 3, by Remark \ref{Homogeneous Cases Remark} following Proposition \ref{Homogeneous Cases}, $\varphi=z_1^M+cz_2^M$, some $c\neq 0$. Therefore, $\omega=Cz_1^{M-2}z_2^{M-2}$. Then, $\nabla \varphi \approx (z_1^{M-1},z_2^{M-1})=(2^{(M-1)n},2^{-(M-1)n})$, and since $\nabla \varphi(0,0)=0$ and $T>0$ each require $M\neq 1$, we are done.

\vspace{.5pc}

Finally, in option 4, $\nabla \varphi \approx (z_1^{M-1},z_1^{M-1})$, and since $\nabla \varphi(0,0)=0$ requires $M\neq 1$, we are done.

\vspace{.5pc}

Thus, in each case where $d_\omega=T$, the conditions of Proposition \ref{refinementconclusion1} hold. 
\end{proof}
Hence, by Proposition \ref{refinementconclusion1}, the proof of Lemma \ref{Vertex1Equalityintro} and thereby Proposition \ref{Vertex1Rectprop} is complete. \qed

\section{Twisted Cases}\label{S:Twisted Cases}
 
In this section, we will prove that in the twisted cases, the high order vanishing of $\omega$ has no effect on the $L^p$ bounds, by proving the following:
\begin{proposition}\label{Vector1twistprop}
In the twisted cases (i), (iia), and (iib), $\scriptT_{R_T}$ is of rwt $(p_{v_1},q_{v_1})$.
\end{proposition}
A few concrete examples of functions $\varphi$ falling into the twisted cases are: 
\begin{align*}
\varphi&=z_1^4+z_1^2z_2+\tfrac{1}{6}z_2^2; \quad \quad \varphi= z_1^5+z_1^3z_2+\tfrac{9}{40}z_1z_2^2  \quad &(Case(i));
\\
\varphi&=z_1^4+z_1^2z_2+z_2^2  \quad &(Case(iia));
\\
\varphi&=(z_2-z_1^2)(z_2-2z_1^2) \quad &(Case(iib)).
\end{align*}

First, we will discuss the idea of the proof. As before, we will decompose $R_T\subset[-1,1]^2$ into suitable dyadic rectangles $\tau_{j,k}$. In rectangular cases, the dominant factor $f_T$ of $\omega$ is also a factor of one or both components of $\nabla \varphi$, which causes $\nabla \varphi(\tau_{j,k})$ to be essentially convex. In the twisted cases, on the other hand, the dominant factor $f_T$ of $\omega$ arises from a ``fortuitous" cancellation, and not in an obvious way from the structure of $\varphi$ and $\nabla \varphi$. This leads to $\nabla\varphi$ ``twisting" near the associated curve $f_T=0$. Because of this, $\nabla \varphi(\tau_{j,k})\subset \R^2$ is highly non-convex, and similar to the neighborhood of a parabola.

In the method of refinements, the Jacobian determinant, which is used to find $L^{4/3}\rightarrow L^4$ bounds, relies primarily on the convex hull of $\nabla \varphi(\tau_{j,k})$, while the $L^\infty\rightarrow L^\infty$ bounds are connected to the measure of $\nabla \varphi(\tau_{j,k})$. Hence, $\nabla \varphi(\tau_{j,k})$ being highly non-convex causes the classes of $L^\infty\rightarrow L^\infty$ and $L^{4/3}\rightarrow L^4$ near-extremizers to be disjoint, and quantifying that tradeoff leads to a much better bounds, implying that $\scriptT_{R_T}$ is of rwt $(p_{v_1},q_{v_1})$, with $(p_{v_1},q_{v_1})$ lying on the scaling line.

We start our arguments with two lemmas that will be necessary to show that $\nabla\varphi(\tau_{j,k})$ is always highly non-convex in the twisted cases.

%
 
\begin{lemma}\label{refinement special case 1}
Set $y:=z_2- z_1^r$, where $r\geq 2$. If $y$ is a factor of $\omega$ with multiplicity $T>d_\omega$, and $y$ is not a factor of $\varphi$, then $y$ is not a factor of $\partial_{z_2}\varphi$.
\end{lemma}
\begin{proof}
 We express $\varphi$ as a polynomial in $x:=z_1$ and $y:=z_2-z_1^r$. By the hypotheses of Theorem \ref{Main Theorem Alt} and the lemma, $\varphi\neq z_1^J$ and $y$ does not divide $\varphi$. Thus, by the mixed homogeneity, after rescaling, 
$$
\varphi=x^J+c_lx^{J-lr}y^l+\scriptO(y^{l+1}),
$$
for some $J\neq 0$, $c_l\neq 0$. Suppose $l\geq 2$. From \eqref{Curve Curvature},
$$
\omega=J(J-1)c_ll(l-1)x^{2J-lr-2}y^{l-2}+\scriptO(y^{l-1}).
$$
Thus, $y$ is a factor of $\omega$ with multiplicity $T=l-2$. However, since we require $T>d_\omega$, and we know $d_\omega=2d_h-2=\frac{2J}{r+1}-2$, then $2J<l(r+1)$. From the $c_l$ term of $\varphi$, we know that $J\geq lr\geq l$, contradicting $2J<l(r+1)$. Thus, $l=1$, in which case $y$ is not a factor of $\partial_y\varphi=\partial_{z_2}\varphi$.
\end{proof}

\begin{lemma}\label{refinement special case 2} If $y:=z_2-z_1^r$, $r\geq 2$, is a factor of  $\partial_{z_1}\varphi$, and a factor of $\omega$ with multiplicity $T>d_\omega$, but not a factor of $\varphi$, then $\partial_{z_1}\varphi\equiv 0$. 
\end{lemma}
 
\begin{proof}
 
Assume $\partial_{z_1}\varphi\not\equiv 0$ and let $M\neq 0$ denote the multiplicity of $y$ in $\partial_{z_1}\varphi$. Defining $x:=z_1$, by Lemma \ref{refinement special case 1} we can write $\varphi$ as follows, for some $c_1\neq 0$, after rescaling: 
 $$\varphi=x^J+c_1x^{J-r}y+\scriptO(y^2).$$ 

\begin{claim} Let $M$ be the multiplicity of $y$ in $\partial_{z_1}\varphi$, and suppose that $M\neq 0$ and $\partial_{z_1}\varphi\not\equiv 0$. Then $y$ has multiplicity $M-1$ in $\omega$.
\end{claim}

\begin{claimproof} First, we rewrite $\omega$ in a way that preserves its relationship with $\partial_{z_1}\varphi$, while still taking advantage of the coordinates $x$ and $y$:
\begin{align*}
\omega&:=\partial_x(\partial_{z_1}\varphi)\partial_{yy}\varphi-\partial_y(\partial_{z_1}\varphi)\partial_{xy}\varphi.
\end{align*}

Since $M\neq 0$, $y$ is a factor of $\partial_y(\partial_{z_1}\varphi)$ of multiplicity $M-1$, and since $c_1\neq 0$ (and $J>r$), $y$ is not a factor of $\partial_{xy}\varphi$. Additionally, $y$ is a factor of $\partial_x(\partial_{z_1}\varphi)$ with multiplicity at least $M$. Thus, $y$ is factor of $\omega$ of multiplicity $M-1$. \end{claimproof}

Since $d(\partial_{z_1}\varphi)=\frac{Mr+K}{r+1}$ for some $K$, and by Lemma \ref{domega dh} (twice), the assumption $d_\omega<T$, and our above claim,
\begin{align*}
2\tfrac{Mr}{1+r}&\leq 2d(\partial_{z_1}\varphi)=2d_h-\tfrac 2{1+r}=d_\omega+2-\tfrac 2{1+r}
\\
&<T+2-\tfrac 2{1+r}=M+1-\tfrac 2{1+r}.
\end{align*}
This simplifies to $(M-1)(r-1)<0$, which is impossible by assumptions $M\geq 1$ and $r\geq 1$. Hence $\partial_{z_1}\varphi\equiv 0$. 
\end{proof}

Next, to shorten the remainder of the argument, we will unify the twisted cases. In Case(i), we can assume that $f_T=z_2$, so that $z_2$ has multiplicity $T$ in $\omega$ and multiplicity $0$ in $\varphi$. Thus, after rescaling, because of mixed homogeneity $\varphi$ can take the following form:
$$
(i) \quad \varphi =z_1^J+c_1z_1^{J-r}z_2(1+\scriptO(\tfrac{z_2}{z_1^r}))=x^J+c_1x^{J-r}y(1+\scriptO(\tfrac y{x^r})),
$$
for some $J\in \N$, where we used the change of coordinates $x:=z_1$, $y=z_2$ for (i). Similarly, in Cases(iia) and (iib), we can assume after rescaling that $f_T=z_2-z_1^r$, so that $z_2-z_1^r$ has multiplicity $T$ in $\omega$ and multiplicity $0$ or $1$ respectively in $\varphi$. Thus, after rescaling, because of mixed homogeneity, $\varphi$ can take the following form: 
\begin{align*} (iia) \quad \varphi &=z_1^J(1+\scriptO(\tfrac{z_2-z_1^r}{z_1^r}))=x^J(1+\scriptO(\tfrac y{x^r})); \\
(iib) \quad \varphi &=z_1^{J-r}(z_2- z_1^r)(1+\scriptO(\tfrac{z_2-z_1^r}{z_1^r}))=x^{J-r}y(1+\scriptO(\tfrac y{x^r})),
\end{align*}
for some $J$, where we used the change of coordinates $x:=z_1$, $y:=z_2- z_1^r$ for (iia-b). Also, we can assume that $r\geq 2$ by Lemma \ref{Homogeneous Cases}. Observe that in each case,
\begin{equation}\label{refinementtwisted dh}
d_h=\tfrac{J}{r+1}.
\end{equation}
The following lemma will unify the form of $\varphi$:
\begin{lemma}\label{refinement twisted nabla form} In each twisted case, for some $a_0, b_0\neq 0$, $\nabla \varphi$ can be written as
$$
\nabla \varphi=(\partial_{z_1}\varphi,\partial_{z_2}\varphi)=(x^{J-1}(a_0+\scriptO(\tfrac{y}{x^r})),x^{J-r}(b_0+\scriptO(\tfrac{y}{x^r})).
$$  
\end{lemma}

\begin{proof}
In case(i), $x=z_1$ and $y=z_2$, and
$$
\nabla\varphi=(Jz_1^{J-1}+\scriptO(z_2), c_1(J-r)z_1^{J-r}+\scriptO(z_2)),
$$
where by our definition of case(i), $c_1\neq 0$. Thus $a_0=J$ and $b_0=c_1(J-r)$ are both nonzero, since if $J=r$, $\varphi$ would violate the assumption $\nabla\varphi(0)=0$.

For cases (iia) and (iib), we will need to rewrite $\partial_{z_1}$ and $\partial_{z_2}$ in terms of $\partial_x$ and $\partial_y$. Doing so, we have 
$$
\partial_{z_1}=\partial_{x}- rx^{r-1}\partial_{y}; \quad \quad \partial_{z_2}=\partial_{y}.
$$

In case (iia), $\varphi=x^J+c_1x^{J-r}y+\scriptO(y^2)$, and so
$$
\nabla\varphi=(\partial_x\varphi- rx^{r-1}\partial_y\varphi,\partial_y\varphi)=(J- rc_1)x^{J-1}+\scriptO(y),c_1x^{J-r}+\scriptO(y)).
$$
By Lemma \ref{omega=0}, $\omega\not\equiv 0$, so by Lemmas \ref{refinement special case 2} and \ref{refinement special case 1}, respectively, $a_0, b_0\neq 0$.

In this case (iib), $\varphi=x^{J-r}y+\scriptO(y^2)$, and so 
$$
\nabla\varphi=(\partial_x\varphi- rx^{r-1}\partial_y\varphi,\partial_y\varphi)=(- rx^{J-1}+\scriptO(y),x^{J-r}+\scriptO(y)),
$$
resulting in $a_0=-r$ and $b_0=1$ being both nonzero.
\end{proof}

Additionally, we want to consider $\omega$ itself. There exists some $Q$ such that
$$
\omega\approx x^Qy^T+\scriptO(y^{T+1})=x^{Q+Tr}(\tfrac{y}{x^r})^T+\scriptO(y^{T+1}),
$$
where $d_\omega=\frac{Q+Tr}{r+1}=2d_h-2=\frac{2J-2r-2}{r+1}$ by \eqref{refinementtwisted dh}. Thus, $Q+Tr=2J-2r-2$, so 
\begin{equation} \label{eqn:AsymOmega}
\omega= x^{Q+Tr}(\tfrac{y}{x^r})^T+\scriptO(y^{T+1})=x^{2J-2r-2}(\tfrac{y}{x^r})^T+\scriptO(y^{T+1}).
\end{equation}

Next, we decompose $R_T=\{|\frac{y}{x^r}|<\epsilon,0<x<1\}$. Rescaling, it suffices to only consider $x<\tilde{c}$, for some small constant $\tilde{c}$ independent of $x$ and $y$. For simplicity, we give details when $y>0$, the case $y<0$ being similar.

We first define , in cases (iia-b), function $z(x,y):=(x,y+x^r)$, and in case (i), we define $z$ to be the identity. Thus, in each case, $z(x,y)=(z_1,z_2)$. Dyadically decomposing $R_T$, we define $\mathbf{m}:=(m_1,m_2)\in \N^2$, $x_\mathbf{m}:=2^{-m_1}$, $\frac{y_\mathbf{m}}{x_\mathbf{m}^r}:=2^{-m_2}$, and define curved rectangles $\tau_{\mathbf{m}}:=z([x_\mathbf{m},2x_\mathbf{m}]\times[y_\mathbf{m},2y_\mathbf{m}])$, along with extended rectangles $\tau_{\mathbf{m}}^e:=z([\frac 12 x_\mathbf{m},3x_\mathbf{m}]\times[y_\mathbf{m},2y_\mathbf{m}])$. Denote $\scriptT_\mathbf{m}:=\scriptT_{\tau_{\mathbf{m}}}$.

%

\vspace{.5pc}

To begin, we want to refine $E\rightsquigarrow E_\mathbf{m}$ and $F\rightsquigarrow F_\mathbf{m}$ so that $\mathcal{T}_{\mathbf{m}}(E,F) \approx \mathcal{T}_{\mathbf{m}}(E_\mathbf{m}, F_\mathbf{m})$. Define
\begin{align*}
E_\mathbf{m}&:=\{w\in E: \mathcal{T}^*_{\mathbf{m}}\chi_F(w)\geq \tfrac{1}{4}\tfrac{\mathcal{T}_{\mathbf{m}}(E,F)}{|E|}\}
\\
F_\mathbf{m}&:=\{u\in F: \mathcal{T}_{\mathbf{m}}\chi_{E_\mathbf{m}}(u)\geq \tfrac{1}{4}\tfrac{\mathcal{T}_{\mathbf{m}}(E_\mathbf{m},F)}{|F|}=:\alpha_\mathbf{m}\}
\\
E_\mathbf{m}'&:=\{w\in E_\mathbf{m}: \mathcal{T}_{\mathbf{m}}^*\chi_{E_\mathbf{m}}(w)\geq \tfrac{1}{4}\tfrac{\mathcal{T}_{\mathbf{m}}(E_\mathbf{m},F_\mathbf{m})}{|E_\mathbf{m}|}=:\beta_\mathbf{m}\}
\end{align*}
Then $\mathcal{T}_{\mathbf{m}}(E_\mathbf{m},F_\mathbf{m})\approx  \mathcal{T}_{\mathbf{m}}(E,F)$, by an argument similar to the argument leading to \eqref{refinement approx}.

 To proceed, we will create a map, and use the size of the Jacobian to acquire the desired bound. Fix $e_0\in E_\mathbf{m}'$, and define: 
$$\Omega_1:=\{t\in \tau_{\mathbf{m}} \text{ and } e_0+(t,\varphi(t))=:e_1(t)\in F\}.$$ 
Then $|\Omega_1|=\mathcal{T}_{\mathbf{m}}^*\chi_{F_\mathbf{m}}(e_0)\geq \beta_\mathbf{m}$. Next, define, for $t\in \Omega_1$,
$$\Omega_2(t):=\{s \in \tau_{\mathbf{m}} \text{ and } e_1(t)-(s,\varphi(s))\in E_\mathbf{m}\}.$$ Then $|\Omega_2(t)|=T_{\mathbf{m}}\chi_{E_\mathbf{m}}(e_1(t))\geq \alpha_\mathbf{m}$. Define the following:
$$\Omega :=\{(t;s^{(1)},s^{(2)})\in \R^6:t\in \Omega_1, s^{(i)}\in \Omega_2(t), i=1,2\};$$ 
$$
\Psi(t;s^{(1)},s^{(2)})= (e_0+(t,\varphi(t))-(s^{(i)},\varphi(s^{(i)}))): i=1,2.
$$
Since $\Psi$ is a polynomial mapping $\Omega\subset\R^6$ into $\R^6$, with $\det D\Psi\not\equiv 0$, it is $\scriptO(1)$-to-one off a set of measure zero. Then $\Psi(\Omega)\subset E_\mathbf{m}^2$ implies 
\\
$|E_\mathbf{m}|^2\gtrsim \int_\Omega |\det D\Psi(t;s^{(1)},s^{(2)})|dtds^{(1)}ds^{(2)}$. Expanding,
\begin{align*}
|\det D\Psi(t;s^{(1)},s^{(2)})|&=|\det(\nabla\varphi(s^{(1)})-\nabla\varphi(t), \nabla\varphi(s^{(2)})-\nabla\varphi(t))|\\  
&\approx \mu(\overline{\text{Conv}}\{\nabla\varphi(s^{(1)}),\nabla\varphi(s^{(2)}),\nabla\varphi(t)\}),
\end{align*}
where $\mu(\overline{\text{Conv}}(S))$, $S\subset \R^2$, is the area of the convex hull of $S$, implying that 
\begin{equation}\label{refinementtwisted Jacobian}
|E_\mathbf{m}|^2 \gtrsim \int_\Omega \mu(\overline{\text{Conv}}\{\nabla\varphi(s^{(1)}),\nabla\varphi(s^{(2)}),\nabla\varphi(t)\})dtds^{(1)}ds^{(2)}.
\end{equation}


To proceed, we will consider two regimes: $\alpha_\mathbf{m}\leq C\frac{y_\mathbf{m}}{x_\mathbf{m}^r}x_\mathbf{m}y_\mathbf{m}$ and $C\frac{y_\mathbf{m}}{x_\mathbf{m}^r}x_\mathbf{m}y_\mathbf{m}\leq \alpha_\mathbf{m}\lesssim x_\mathbf{m}y_\mathbf{m}$, for some sufficiently large constant $C$ to be decided later. For small $\alpha_\mathbf{m}$, we will directly calculate the $L^\frac 43\rightarrow L^4$ bound using Theorem \ref{Gressman} and interpolate. For the larger $\alpha_\mathbf{m}$, we will instead quantify how far the set $\nabla\varphi(\Omega_2(t))$ violates convexity and use \eqref{refinementtwisted Jacobian} as follows:

\begin{lemma}\label{refinementtwisted convhull} Let $t\in \Omega_1$ be fixed. If $\alpha_\mathbf{m}\geq C\tfrac{y_\mathbf{m}}{x_\mathbf{m}^r}x_\mathbf{m}y_\mathbf{m}$, then 
$$
\mu(\overline{\text{Conv}}\{\nabla\varphi(t),\nabla\varphi(s^{(1)}),\nabla\varphi(s^{(2)})\})\gtrsim x_\mathbf{m}^{-[J-2r+1]}(\tfrac {\alpha_\mathbf{m}}{x_\mathbf{m}y_\mathbf{m}}x_\mathbf{m}^{J-r})^3
$$
on some subset $\{t\}\times\Omega_2^{(1)}(t)\times\Omega_2^{(2)}(t)\subset \{t\}\times \Omega_2(t)\times \Omega_2(t)$, with $|\Omega_2^{(i)}(t)|\geq \frac 1{10}|\Omega_2(t)|$.
\end{lemma}

We will use Lemma \ref{3 point}, together with Lemmas \ref{refinement uncertainty scales}-\ref{thickness average}, to prove Lemma \ref{refinementtwisted convhull}.

\begin{lemma}\label{3 point}
Let $\gamma$ be a $C^2$ curve in $\R^2$, with curvature $\kappa\approx \Theta$, and such that the set of all unit tangent vectors of $\gamma$ belongs to the same half quadrant, or its mirror image. Let $p_1, p_2, p_3 \in \R^2$ lie in a $\delta^2\Theta$ neighborhood of $\gamma$, and satisfy $\norm{p_i-p_j}>a$ for all $i\neq j$. If $a>40\max\{\delta^2\Theta,\delta(\tfrac{\Theta}{\inf \kappa})^\frac 12\}$, then $\mu(\overline{\text{Conv}}\{p_1,p_2,p_3\})\gtrsim \Theta a^3$.
\end{lemma}

\begin{proof} 
There exist points $\tilde p_1, \tilde p_2, \tilde p_3$ on $\gamma$ closest to $p_1, p_2, p_3$, respectively, and since $a>40\delta^2\Theta$, the points $\tilde p_i$ have separation at least $\frac a2$. Since $\gamma$ cannot be a closed curve by our hypothesis on tangent lines, we can reorder indices such that $\tilde p_2$ lies between $\tilde p_1$ and $\tilde p_3$ on $\gamma$. 

We translate and rotate $\R^2$ so that $\tilde p_2=0$ and the tangent line of $\gamma$ at $0$ becomes horizontal. We then reflect $\R^2$ if needed so that $\gamma$ lies in the upper half plain and so that $\tilde p_1$ lies in the left plain, and $\tilde p_3$ in the right. We can then write $\gamma$ as the graph $(x,g(x))$ of some function $g$, where $g(0)=g'(0)=0$, $|g'|\leq 1$, and $g''\in (\frac 12 \inf \kappa, 4 \sup \kappa)$. Therefore, $g(x)\geq \frac 14 (\inf \kappa)x^2$.

We recall that $\pi_i$ projects points in $\R^2$ onto their $i$-th coordinate. The separation of the $\tilde p_i$ must be at least $\frac a2$, and since $|g'|\leq 1$, then $|\pi_1(\tilde p_1)|$, $|\pi_1(\tilde p_3)|\geq \frac a{2\sqrt 2}$. Therefore, $|\pi_2(\tilde p_1)|,|\pi_2(\tilde p_3)|\geq \frac 1{16} (\inf \kappa)a^2$.

Therefore, $\tilde p_1, \tilde p_2, \tilde p_3$ form a triangle of triangle with base and height greater than $\frac a{\sqrt 2}$ and $\frac 1{16}(\inf \kappa)a^2$, respectively. Since $\tilde p_i$ is within $\delta^2 \Theta$ of $p_i$ for each $i$, and since $a>40\max\{\delta^2\Theta,\delta(\tfrac{\Theta}{\inf \kappa})^\frac 12\}$, then $p_1$, $p_2$, $p_3$ form a triangle with base and height greater than $\frac a{2\sqrt 2}$ and $\frac 1{32}(\inf \kappa)a^2$, respectively. Hence, the triangle formed by $p_1, p_2, p_3$ has area $\gtrsim \Theta a^3$.
\end{proof}

Using Lemma \ref{3 point} to prove Lemma \ref{refinementtwisted convhull} will require us to identify a curve, bound its curvature, restrict the possible unit tangent vectors, establish a minimal separation of points, and show points remain in a neighborhood of the curve. A suitable curve $\gamma_y$, or $\gamma_y^*$, will be defined in Lemmas \ref{gamma curvature}-\ref{gamma function}. The curvature will be bounded in Lemma \ref{gamma curvature}. Lemma \ref{gamma second derivative} will restrict the possible unit tangent vectors. The minimal separation comes from Lemma \ref{refinement uncertainty scales} and our lower bound on $\alpha_\mathbf{m}$. And finally, the neighborhood size $\scriptH$ will be bounded in Lemmas \ref{gamma second derivative}-\ref{thickness maximum}.

Our goal will be to look at the structure of $\nabla\varphi(z(x,y))$, and then use Lemma \ref{3 point} on the points $\nabla\varphi(s^{(1)})$, $\nabla\varphi(s^{(2)})$, and $\nabla\varphi(t)$. To proceed, we will consider $(x,y)\in z^{-1}(\tau_{\mathbf{m}})$, and analyze $\nabla\varphi(\tau_{\mathbf{m}})$. Define new coordinates
$$
\tilde t:= z^{-1}(t); \qquad \tilde s^{(i)}:= z^{-1}(s^{(i)}), \quad i=1,2,
$$
where we recall that $z$ was defined so that $z^{-1}(\tau_{\mathbf{m}})$ is rectangular. Additionally, define
$$\Phi\equiv (\Phi_1,\Phi_2):= (\sgn{(a_0)}(\partial_{z_1}\varphi)\circ z,\sgn{(b_0)}(\partial_{z_2}\varphi)\circ z).$$
By Lemma \ref{refinement twisted nabla form}, $a_0, b_0\neq 0$ and
\begin{equation}\label{refinementtwisted Phi}
\Phi(x,y)=(x^{J-1}[|a_0|+\tilde{a_1}\tfrac{y}{x^r}+\scriptO((\tfrac{y}{x^r})^2)],x^{J-r}[|b_0|+\tilde{b_1}\tfrac{y}{x^r}+\scriptO((\tfrac{y}{x^r})^2)]),
\end{equation}
where $\tilde{a_1}$ and $\tilde{b_1}$ may be zero. Because $\Phi$ and $\nabla \varphi\circ z$ are equal up to coordinate sign changes, we can work purely with $\Phi$, since
\begin{equation}
\mu(\overline{\text{Conv}}\{\Phi(\tilde t),\Phi(\tilde s^{(1)}),\Phi(\tilde s^{(2)})\})=\mu(\overline{\text{Conv}}\{\nabla\varphi(t),\nabla\varphi(s^{(1)}),\nabla\varphi(s^{(2)})\}).
\end{equation}

First, we will show that $x$ and $\xi_2=\Phi_2(x,y)$ are strongly related.

\begin{lemma}\label{refinement uncertainty scales}
Let $(x,y),(x+\Delta x,y+\Delta y)\in z^{-1}(\tau_\mathbf{m}^e)$, and define $\delta_{\flat}:=x_\mathbf{m}\frac{y_\mathbf{m}}{x_\mathbf{m}^r}$ and $\delta_\Phi:=x_\mathbf{m}^{J-r}\frac{y_\mathbf{m}}{x_\mathbf{m}^r}$ as the $x,\Phi_2$ uncertainty scales, respectively. Then if either $|\Phi_2(x+\Delta x,y+\Delta y)-\Phi_2(x,y)|\geq C\delta_\Phi$, or $|\Delta x|\geq C\delta_{\flat}$, for a sufficiently large $C$, then 
$$
\text{sgn}{(\Delta x)}\big(\Phi_2(x+\Delta x,y+\Delta y)-\Phi_2(x,y)\big)\approx x_\mathbf{m}^{J-r-1}|\Delta x|.
$$
\end{lemma}

\begin{proof} By \eqref{refinementtwisted Phi}, recalling that $b_0\neq 0$, we have
\begin{align*}
\partial_x\Phi_2&=x^{J-r-1}[|b_0|(J-r)+\scriptO(\tfrac y{x^r})]\approx x_\mathbf{m}^{J-r-1}
\\
|\partial_y\Phi_2|&=|x^{J-2r}[\tilde{b_1}+\scriptO(\tfrac y{x^r})]|\lesssim x_\mathbf{m}^{J-2r}.
\end{align*}
Therefore, since $|\Delta y|\leq y_\mathbf{m}$ in $z^{-1}(\tau_\mathbf{m}^e)$, a change in $y$ will only change $\Phi_2$ by at most $\tilde Cx_\mathbf{m}^{J-2r}y_\mathbf{m}=\tilde C\delta_\Phi$. Next, consider $x_2>x_1$. Since $\partial_x\Phi_2\approx x_\mathbf{m}^{J-r-1}$, then $\Phi_2(x_2,y_1)-\Phi_2(x_1,y_1)\approx (x_2-x_1)x_\mathbf{m}^{J-r-1}$. Thus, by the triangle inequality, if $|\Phi_2(x_2,y_2)-\Phi_2(x_1,y_1)|>C\delta_\Phi$, or if $|x_2-x_1|>C\delta_\flat$, for some large $C>0$, then $\Phi_2(x_2,y_2)-\Phi_2(x_1,y_1)\approx (x_2-x_1)x_\mathbf{m}^{J-r-1}$. This concludes the proof of Lemma \ref{refinement uncertainty scales}.
\end{proof}
%
The next lemmas define and analyze the curves that will be used with Lemma \ref{3 point}.

\begin{lemma}\label{gamma curvature}
For $y\in [y_\mathbf{m},2y_\mathbf{m}]$, $\gamma_y:x\mapsto \Phi(x,y)$ has curvature $\kappa_y\approx x_\mathbf{m}^{-[J-2r+1]}:=\tilde{\Theta}$.
\end{lemma}
\begin{proof}[Proof of Lemma \ref{gamma curvature}] By direct computation using \eqref{refinementtwisted Phi}, 
\begin{align*}
\kappa_y(x)&=x^{-[J-2r+1]}\frac{|a_0b_0|(J-1)(J-r)(r-1)+\scriptO(\frac y{x^r})}{[x^{2r-2}a_0^2+b_0^2+\scriptO(\frac y{x^r})]^\frac{3}{2}}.
\end{align*}
Then, since $0<x\leq 1$ and $a_0,b_0\neq 0$,
$$
\kappa_y(x)\approx x^{-[J-2r+1]}\approx x_\mathbf{m}^{-[J-2r+1]}=:\tilde{\Theta},
$$
concluding the proof of Lemma \ref{gamma curvature}.
\end{proof}

\begin{lemma}
Let $y\in [y_\mathbf{m},2y_\mathbf{m}]$. Then there exists a function $\gamma_y^*:\Phi_2(z^{-1}(\tau_\mathbf{m}))\rightarrow \R$ such that
$$
\gamma_y([\tfrac 12x_\mathbf{m},3x_\mathbf{m}])\cap \{\R\times\Phi_2(z^{-1}(\tau_\mathbf{m}))\}=\{(\xi_1,\xi_2):\xi_1=\gamma_y^*(\xi_2),\xi_2\in \Phi_2(z^{-1}(\tau_\mathbf{m}))\}.
$$
\end{lemma}
This lemma is equivalent to the following lemma:
\begin{lemma}\label{gamma function}
Let $y\in [y_\mathbf{m},2y_\mathbf{m}]$. Then, for each $\xi_2\in \Phi_2(z^{-1}(\tau_\mathbf{m}))$, there exists a unique $\xi_1$ such that $(\xi_1,\xi_2)\in \gamma_y([\tfrac 12x_\mathbf{m},3x_\mathbf{m}])$.
\end{lemma}
\begin{proof}[Proof of Lemma \ref{gamma function}] Recall that $\pi_i$ is projection onto the i-th coordinate.
\vspace{.5pc}
\begin{enumerate}
\item \textit{Existence}: If $\xi_2\in \Phi_2(z^{-1}(\tau_\mathbf{m}))$, there exists $(\tilde{x},\tilde{y})\in z^{-1}(\tau_\mathbf{m})$ such that $\Phi_2(\tilde{x},\tilde{y})=\xi_2$. Now consider points in $z^{-1}(\tau_\mathbf{m}^e)$ with minimal and maximal $x-$values: $(\frac{x_\mathbf{m}}{2},y)$ and $(3x_\mathbf{m},y)$. Since $\frac{x_\mathbf{m}}{2}\leq \tilde{x}-\frac{x_\mathbf{m}}{2}\leq \tilde{x}-c\delta_\flat$ and $3x_\mathbf{m}\geq \tilde{x}+x_\mathbf{m}\geq \tilde{x}+c\delta_{\flat}$, then
$$
\Phi_2(\tfrac{x_\mathbf{m}}{2},y)<\xi_2< \Phi_2(3x_\mathbf{m},y) 
$$ 
by Lemma \ref{refinement uncertainty scales}. Then, by continuity of $\gamma_y(x)$ and the Intermediate Value Theorem, there exists $x\in [\frac 12 x_\mathbf{m},3x_\mathbf{m}]$ such that $\Phi_2(x,y)\equiv \pi_2(\gamma_y(x))=\xi_2$.
\item \textit{Uniqueness}: By \eqref{refinementtwisted Phi}, $\pi_1(\frac{\text{d}}{\text{d}x}\gamma_y(x))\equiv \partial_x\Phi_1(x,y)=x^{J-2}(|a_0|(J-1)+\scriptO(\frac{y}{x^r}))>0$ for each $(x,y)\in z^{-1}(\tau_\mathbf{m}^e)$, so for fixed $y\in [y_\mathbf{m},2y_\mathbf{m}]$, the map $x\mapsto \pi_1(\gamma_y(x))\equiv \Phi_1(x,y)$ is one-to-one. Hence, for fixed $y$, $\xi_1$ is mapped to by at most a single unique $\overline{x}$, and hence a unique $\overline{\xi_2}=\Phi_2(\overline{x},y)$.
\end{enumerate}
This concludes the proof of Lemma \ref{gamma function}.
\end{proof}
\begin{lemma}\label{gamma second derivative}
For every $y\in [y_\mathbf{m},2y_\mathbf{m}]$, $|\frac{\text{d}}{\text{d}\xi_2}\gamma_y^*(\xi_2)|<\frac 1{10}$ and $|\frac{\text{d}^2}{\text{d}\xi_2^2}\gamma_y^*(\xi_2)|\approx x_\mathbf{m}^{-[J-2r+1]}=:\tilde{\Theta}$ for all $\xi_2\in \Phi_2(\tau_\mathbf{m})$.
\end{lemma}
\begin{proof}
Since $|\partial_x\Phi_1|\approx x_\mathbf{m}^{J-2}$, and $|\partial_x\Phi_2|\approx x_\mathbf{m}^{J-r-1}$, then $|\frac{\text{d}}{\text{d}\xi_2}\gamma_y^*(\xi_2)|\approx x_\mathbf{m}^{r-1}<\frac 1{10}$, so the curvature of $\xi_2\mapsto  (\gamma_y^*(\xi_2),\xi_2)$ is comparable to $|\frac{\text{d}^2}{\text{d}\xi_2^2}\gamma_y^*(\xi_2)|$. Finally, by Lemma \ref{gamma curvature}, the curve $\xi_2\rightarrow (\gamma_y^*(\xi_2),\xi_2)$ has curvature $\approx x_\mathbf{m}^{-[J-2r+1]}$, concluding the proof of Lemma \ref{gamma second derivative}.
\end{proof}
Next, define 
 $$\Phi(z^{-1}(\tau_\mathbf{m}))^e:=\Phi(z^{-1}(\tau_\mathbf{m}^e))\cap\{\R\times\Phi_2(z^{-1}(\tau_\mathbf{m}))\},$$
 and observe that $\Phi(z^{-1}(\tau_\mathbf{m}))^e$ satisfies
\begin{equation}
\Phi(z^{-1}(\tau_\mathbf{m}))\subset \Phi(z^{-1}(\tau_\mathbf{m}))^e\subset \Phi(z^{-1}(\tau_\mathbf{m}^e)).
\end{equation}
The expanded image $\Phi(z^{-1}(\tau_\mathbf{m}))^e$ essentially rectangularizes our image $\Phi(z^{-1}(\tau_\mathbf{m}))$, making it more practical for proving results.
\begin{lemma}\label{gamma boundary}
The set $\Phi(z^{-1}(\tau_\mathbf{m}))^e$ satisfies
$$\Phi(z^{-1}(\tau_\mathbf{m}))^e=\cup_{\xi_2\in \Phi_2(z^{-1}(\tau_\mathbf{m}))}[\min(\gamma_{y_\mathbf{m}}^*(\xi_2),\gamma_{2y_\mathbf{m}}^*(\xi_2)),\max(\gamma_{y_\mathbf{m}}^*(\xi_2),\gamma_{2y_\mathbf{m}}^*(\xi_2))]\times \{\xi_2\}.$$
\end{lemma}
\begin{proof}
Since $|Jac\Phi|\approx |\omega|>0$ on $\tau_\mathbf{m}^e$, $\Phi$ is an open map by the inverse function theorem, so $\Phi^{-1}(\text{Bdry}\Phi(z^{-1}(\tau_\mathbf{m}^e)))\subset \text{Bdry}z^{-1}(\tau_\mathbf{m}^e)$. By Lemma \ref{refinement uncertainty scales}, for all $y\in [y_\mathbf{m}, 2y_\mathbf{m}]$, $\Phi_2(\frac 12 x_\mathbf{m}, y)$, $\Phi_2(3x_\mathbf{m},y)\not\in \Phi_2(z^{-1}(\tau_\mathbf{m}))$, so the boundary of $\Phi(z^{-1}(\tau_{\mathbf{m}}))^e$ is a union of two horizontal line segments and subsets of the curves $(\gamma^*_{y_\mathbf{m}}(\cdot),\cdot)$ and $(\gamma^*_{2y_\mathbf{m}}(\cdot),\cdot)$ over $\Phi_2(z^{-1}(\tau_\mathbf{m})))$.
\end{proof}
\begin{lemma}\label{gamma diameter}
For $\xi_2\in \Phi_2(z^{-1}(\tau_\mathbf{m}))$, $\mu$ the Lebesgue measure on $\R$, $\mu(\{\xi_1:(\xi_1,\xi_2)\in\Phi(z^{-1}(\tau_\mathbf{m}))^e\})=\text{diam}\{\xi_1:(\xi_1,\xi_2)\in\Phi(z^{-1}(\tau_\mathbf{m}))^e\}= |\gamma_{2y_\mathbf{m}}^*(\xi_2)-\gamma_{y_\mathbf{m}}^*(\xi_2)|$.
\end{lemma}
\begin{proof}
Follows directly from Lemma \ref{gamma boundary}.
\end{proof}
To proceed, we will use the following definitions:
\begin{align}
\scriptW_{[\xi_2^{(1)},\xi_2^{(2)}]}&:=\tfrac{1}{|\xi_2^{(2)}-\xi_2^{(1)}|}\int_{\xi_2\in [\xi_2^{(1)},\xi_2^{(2)}]}\text{diam}\{\xi_1:(\xi_1,\xi_2)\in\Phi(z^{-1}(\tau_\mathbf{m}))^e\}\text{d}\xi_2;
\\
\scriptH &:=\max_{\xi_2\in\Phi_2(z^{-1}(\tau_\mathbf{m}))}\text{diam}\{\xi_1:(\xi_1,\xi_2)\in\Phi(z^{-1}(\tau_\mathbf{m}))\}.
\end{align}

\begin{lemma}\label{thickness average}
If $\xi_2^{(1)}, \xi_2^{(2)}\in \Phi_2(z^{-1}(\tau_\mathbf{m}))$, and $\xi_2^{(2)}-\xi_2^{(1)}>C\delta_\Phi$, then 
$$\scriptW_{[\xi_2^{(1)},\xi_2^{(2)}]}\approx x_\mathbf{m}^{J-1}(\tfrac{y_\mathbf{m}}{x_\mathbf{m}^r})^{T+1}.$$
\end{lemma}
\begin{proof}
By Lemma \ref{refinement uncertainty scales}, $\Phi([x_1,x_2]\times [y_\mathbf{m},2y_\mathbf{m}]) \subset\Phi(z^{-1}(\tau_\mathbf{m}))^e\cap (\R\times [\xi_2^{(1)},\xi_2^{(2)}])\subset \Phi([x_1',x_2']\times [y_\mathbf{m},2y_\mathbf{m}])$, some $x_1',x_2',x_1,x_2\in [\frac 12 x_\mathbf{m},3x_\mathbf{m}]$, where $x_2-x_1\approx x_2'-x_1'\approx (\frac 1{x_\mathbf{m}})^{J-r-1}(\xi_2^{(2)}-\xi_2^{(1)})$. Then
\begin{align*}
\scriptW_{[\xi_2^{(1)},\xi_2^{(2)}]}&\approx \tfrac{|\Phi([x_1,x_2]\times [y_\mathbf{m},2y_\mathbf{m}])|}{\xi_2^{(2)}-\xi_2^{(1)}}\approx \tfrac{|Jac\Phi|y_\mathbf{m}(x_2-x_1)}{\xi_2^{(2)}-\xi_2^{(1)}}
\\
&\approx x_\mathbf{m}^{2J-2r-2}(\tfrac{y_\mathbf{m}}{x_\mathbf{m}^r})^Ty_\mathbf{m}\tfrac 1{x_\mathbf{m}^{J-r-1}}\approx x_\mathbf{m}^{J-1}(\tfrac{y_\mathbf{m}}{x_\mathbf{m}^r})^{T+1}
\end{align*}
since $|Jac \Phi|\approx |\omega|\approx x_\mathbf{m}^{2J-2r-2}(\tfrac{y_\mathbf{m}}{x_\mathbf{m}^r})^T$ by \eqref{eqn:AsymOmega}.
\end{proof}
\begin{lemma}\label{thickness maximum}
The following inequality holds: $\scriptH\lesssim \tilde{\Theta}\delta_\Phi^2$.
\end{lemma}

\begin{proof}
We can consider two cases: $\scriptH\gg \scriptW_{\Phi_2(z^{-1}(\tau_\mathbf{m}))}$ or $\scriptH\lesssim \scriptW_{\Phi_2(z^{-1}(\tau_\mathbf{m}))}$.
\vspace{.25pc}
\\
If $\scriptH\lesssim \scriptW_{\Phi_2(z^{-1}(\tau_\mathbf{m}))}$, then
$$
\scriptH\lesssim x_\mathbf{m}^{J-1}(\tfrac{y_\mathbf{m}}{x_\mathbf{m}^r})^{T+1}\leq x_\mathbf{m}^{J-1}(\tfrac{y_\mathbf{m}}{x_\mathbf{m}^r})^{2}=x_\mathbf{m}^{-[J-2r+1]}(x_\mathbf{m}^{J-r}\tfrac{y_\mathbf{m}}{x_\mathbf{m}^r})^2 = \tilde{\Theta}\delta_\Phi^2
$$
by Lemma \ref{thickness average}, $T\geq 1$, simple reordering, and the definitions of $\tilde \Theta$ and $\delta_\Phi$. 
\vspace{.25pc}
\\
If $\scriptH\gg \scriptW_{\Phi_2(z^{-1}(\tau_\mathbf{m}))}$, then by Lemma \ref{gamma diameter}, there exist $c_1$, $\xi_2'$ such that $c_1\in [\xi_2',\xi_2'+3C\delta_\Phi]\subset \Phi_2(z^{-1}(\tau_m))$ and 
\begin{equation*}
|\gamma_{2y_\mathbf{m}}^*(c_1)-\gamma_{y_\mathbf{m}}^*(c_1)|\gtrsim \scriptH.
\end{equation*}
Additionally, by Lemma \ref{thickness average}, 
\begin{equation*}
\text{avg}_{[\xi_2',\xi_2'+C\delta_\Phi]}|\gamma_{2y_\mathbf{m}}^*-\gamma_{y_\mathbf{m}}^*|\approx \scriptW_{\Phi_2(z^{-1}(\tau_\mathbf{m}))}\approx \text{avg}_{[\xi_2'+2C\delta_\Phi,\xi_2'+3C\delta_\Phi]}|\gamma_{2y_\mathbf{m}}^*-\gamma_{y_\mathbf{m}}^*|,
\end{equation*}
and so, by the Intermediate Value Theorem, there exist $c_2\in [\xi_2',\xi_2'+C\delta_\Phi]$ and $c_3\in [\xi_2'+2C\delta_\Phi,\xi_2'+3C\delta_\Phi]$ such that
\begin{equation*}
|\gamma_{2y_\mathbf{m}}^*(c_2)-\gamma_{y_\mathbf{m}}^*(c_2)|, |\gamma_{2y_\mathbf{m}}^*(c_3)-\gamma_{y_\mathbf{m}}^*(c_3)|\lesssim \scriptW_{\Phi_2(z^{-1}(\tau_\mathbf{m}))}.
\end{equation*}
Then, since $|c_2-c_3|\approx \delta_\Phi$,  $|c_1-c_3|\lesssim \delta_\Phi$, and $\scriptH\gg\scriptW_{\Phi_2(z^{-1}(\tau_\mathbf{m}))}$, by the Mean Value Theorem there exist $c_4, c_5\in [\xi_2',\xi_2'+3C\delta_\Phi]$ such that
\begin{equation*}
|\gamma_{2y_\mathbf{m}}^{*'}(c_4)-\gamma_{y_\mathbf{m}}^{*'}(c_4)|\lesssim \tfrac{\scriptW_{\Phi_2(z^{-1}(\tau_\mathbf{m}))}}{\delta_\Phi} \text{ and } |\gamma_{2y_\mathbf{m}}^{*'}(c_5)-\gamma_{y_\mathbf{m}}^{*'}(c_5)|\gtrsim \tfrac{\scriptH}{\delta_\Phi}.
\end{equation*}
Since $\scriptH\gg \scriptW_{\Phi_2(z^{-1}(\tau_\mathbf{m}))}$ and $|c_4-c_5|\lesssim \delta_\Phi$, by the Mean Value Theorem there exists $c_6\in [\xi_2',\xi_2'+3C\delta_\Phi]$ such that
\begin{equation*}
|\gamma_{2y_\mathbf{m}}^{*''}(c_6)-\gamma_{y_\mathbf{m}}^{*''}(c_6)|\gtrsim \tfrac{\scriptH}{\delta_\Phi^2}.
\end{equation*}
Since Lemma \ref{gamma second derivative} implies that $|\gamma_{2y_\mathbf{m}}^{*''}(c_6)-\gamma_{y_\mathbf{m}}^{*''}(c_6)|\lesssim \tilde{\Theta}$, then $\scriptH\lesssim \tilde{\Theta}\delta_\Phi^2$.
\end{proof}
Now, we will proceed with the proof of Lemma \ref{refinementtwisted convhull}:
\begin{proof}[Lemma \ref{refinementtwisted convhull} proof]
For fixed $t$, there exist $\Omega_2^{(1)}(t), \Omega_2^{(2)}(t)\subset \Omega_2(t)$ such that $|\Omega_2^{(i)}(t)|\geq \frac 1{10}|\Omega_2(t)|$, and such that $t_1, s^{(1)}_1$, and $s^{(2)}_1$ (and equivalently $\tilde t_1$, $\tilde s^{(1)}_1$, and $\tilde s^{(2)}_1$) are separated by at least $\frac 14\frac{\alpha_\mathbf{m}}{y_\mathbf{m}}\geq C\tfrac{y_\mathbf{m}}{x_\mathbf{m}^r}x_\mathbf{m}=C\delta_\flat$.  Then, by Lemma \ref{refinement uncertainty scales}, $\Phi(\tilde t), \Phi(\tilde s^{(1)})$, and  $\Phi(\tilde s^{(2)})$ have mutual $\xi_2-separation$ of at least $\tilde C\delta_\Phi$, where we can make $\tilde C$ arbitrarily large by making the $C$ on the lower bound of $\alpha_{\mathbf{m}}$ sufficiently large. Furthermore, by Lemma \ref{thickness maximum}, the points $\Phi(\tilde t), \Phi(\tilde s^{(i)})$ lie in an $\scriptH\approx\tilde{\Theta}\delta_\Phi^2$ neighborhood of a curve, and by Lemmas \ref{gamma curvature} and \ref{gamma second derivative}, respectively, that curve has curvature $\approx\tilde{\Theta}=x_\mathbf{m}^{-[J-2r+1]}$, and derivative bounded by $\frac 1{10}$, which bounds the possible unit tangent vectors.  

Thus, by choosing the constant on the lower bound of $\alpha_\mathbf{m}$ to be large enough, we can separate $p_1=\Phi(\tilde t), p_2=\Phi(\tilde s^{(1)})$, and  $p_3=\Phi(\tilde s^{(2)})$ by $a=\tilde C\delta_{\Phi}=\tilde Cx_{\mathbf{m}}^{J-r}\frac{y_\mathbf{m}}{x_{\mathbf{m}}^r}$, for a sufficiently large $\tilde C$. By choosing $\Theta$ and $\delta$ so that $\Theta=\tilde \Theta$ and $\Theta\delta^2=\scriptH\approx \tilde\Theta\delta_\Phi^2$, the neighborhood $\Theta \delta^2$ will equal $\scriptH$ so that $\delta\approx \delta_\Phi$. Since $x_{\mathbf{m}}$, $\frac{y_\mathbf{m}}{x_{\mathbf{m}}^r}<1$, then $a=C\delta_{\Phi}$ will satisfy $a>40\max\{\delta^2\Theta,\delta(\tfrac{\Theta}{\inf \kappa})^\frac 12\}$ if $C$ is sufficiently large (note that $\inf \kappa\approx \Theta$), completing the hypotheses of Lemma \ref{3 point}, and implying that on $\{t\}\times\Omega'\times\Omega''$,
\begin{equation*}
\mu(\overline{\text{Conv}}\{\nabla\varphi(t),\nabla\varphi(s^{(1)}),\nabla\varphi(s^{(2)})\})\gtrsim x_\mathbf{m}^{-[J-2r+1]}(\tfrac {\alpha_\mathbf{m}}{x_\mathbf{m}y_\mathbf{m}}x_\mathbf{m}^{J-r})^3. \qedhere
\end{equation*} 
\end{proof}
Next, recall our decomposition of $R_T$ into $\tau_\mathbf{m}$. This induces a decomposition $\scriptT=\sum_{m_2}\sum_{m_1}\scriptT_{\mathbf{m}}$. Define $\scriptT_{m_2}:=\sum_{m_1}\scriptT_{(m_1,m_2)}$. 
\begin{prop}\label{refinementtwistedm2prop} There exists $\sigma(T,J,r)>0$ such that
$$
\scriptT_{m_2}(E,F)\lesssim 2^{-\sigma m_2}|E|^{\frac 3{d_\omega+4}}|F|^{1-\frac 1{d_\omega+4}}.
$$
\end{prop}
\begin{proof} By Young's Inequality, since $|\tau_\mathbf{m}|=x_\mathbf{m}y_\mathbf{m}$, then $\alpha_\mathbf{m}\leq x_\mathbf{m}y_\mathbf{m}$. First, consider $\alpha_\mathbf{m}
\geq C\tfrac{y_\mathbf{m}}{x_\mathbf{m}^r}x_\mathbf{m}y_\mathbf{m}$. By \eqref{refinementtwisted Jacobian} and Lemma \ref{refinementtwisted convhull}, 
\begin{align*}
|E_\mathbf{m}|^2&\gtrsim \int_{\Omega} x_\mathbf{m}^{-[J-2r+1]}(\tfrac {\alpha_\mathbf{m}}{x_\mathbf{m}y_\mathbf{m}}x_\mathbf{m}^{J-r})^3 dtds_1ds_2 
\\
&\gtrsim x_\mathbf{m}^{2[J-r-1]}(\tfrac{y_\mathbf{m}}{x_\mathbf{m}^r})^{-1}(\tfrac{\alpha_\mathbf{m}}{x_\mathbf{m}y_\mathbf{m}})^2{\alpha_\mathbf{m}}\min_{t}(|\Omega_2(t)|)^2|\Omega_1| 
\\
&=x_\mathbf{m}^{2[J-r-1]}(\tfrac{y_\mathbf{m}}{x_\mathbf{m}^r})^{-1}(\tfrac{\alpha_\mathbf{m}}{x_\mathbf{m}y_\mathbf{m}})^2\alpha_\mathbf{m}^3\beta_\mathbf{m}.
\end{align*}
We recall that $\beta_\mathbf{m}\approx\frac{1}{4}\frac{ \mathcal{T}_{\mathbf{m}}(E,F)}{|E_\mathbf{m}|}$ and $\alpha_\mathbf{m}\approx \frac{\mathcal{T}_{\mathbf{m}}(E,F)}{|F|}$, so that
$$|E_\mathbf{m}|^2\gtrsim x_\mathbf{m}^{2[J-r-1]}(\tfrac{y_\mathbf{m}}{x_\mathbf{m}^r})^{-1}(\tfrac{\alpha_\mathbf{m}}{x_\mathbf{m}y_\mathbf{m}})^2\alpha_\mathbf{m}^3\beta_\mathbf{m}\gtrsim x_\mathbf{m}^{2[J-r-1]}(\tfrac{y_\mathbf{m}}{x_\mathbf{m}^r})^{-1}(\tfrac{\alpha_\mathbf{m}}{x_\mathbf{m}y_\mathbf{m}})^2\tfrac{\mathcal{T}_{\mathbf{m}}(E,F)^4}{|E_\mathbf{m}||F|^3}.$$
 
 Then, since $|E|\geq|E_\mathbf{m}|$,
\begin{equation}\label{refinementtwisted 3414}
\mathcal{T}_{\mathbf{m}}(E,F)\lesssim (\tfrac{\alpha}{x_\mathbf{m}y_\mathbf{m}})^{-\frac 12}x_\mathbf{m}^{-\frac 12[J-r-1]}(\tfrac{y_\mathbf{m}}{x_\mathbf{m}^r})^\frac 14|E|^\frac{3}{4}|F|^\frac{3}{4}.
\end{equation}
 
 Also, by the definition of $\alpha_{\mathbf{m}}$, we can write the following $L^\infty\rightarrow L^\infty$ bound:
\begin{equation}\label{refinementtwisted 00}
\scriptT_{\mathbf{m}}(E,F)\lesssim\alpha_{\mathbf{m}}|F|=x_\mathbf{m}y_\mathbf{m}(\tfrac{\alpha_{\mathbf{m}}}{x_\mathbf{m}y_\mathbf{m}})|F|=x_\mathbf{m}^{r+1}(\tfrac{y_\mathbf{m}}{x_\mathbf{m}^r})(\tfrac{\alpha_{\mathbf{m}}}{x_\mathbf{m}y_\mathbf{m}})|F|.
\end{equation}

Since $\frac{y_{\mathbf{m}}}{x_{\mathbf{m}}^{r}}:=2^{-m_2}\lesssim \tfrac{\alpha_{\mathbf{m}}}{x_{\mathbf{m}}y_{\mathbf{m}}}\lesssim 1$, with $x_\mathbf{m}:=2^{-m_1}$, define
$$
\eta(m_3):=\{\mathbf{m}:\tfrac{\alpha_\mathbf{m}}{x_\mathbf{m}y_\mathbf{m}}\approx 2^{-m_3}\},
$$
which induces a function $m_3=m_3(m_1,m_2)$. Define $\scriptT_{m_2,2}:=\sum_{m_1: m_3(m_1,m_2)<m_2-C}\scriptT_{\mathbf{m}}$ and $\scriptT_{m_2,1}:=\sum_{m_1: m_3(m_1,m_2)\geq m_2-C}\scriptT_{\mathbf{m}}$, for a suitable $C\geq 0$. Then, combining \eqref{refinementtwisted 3414} and \eqref{refinementtwisted 00}, 
\begin{align*}\label{refinementtwistedm22}
\scriptT_{m_2,2}(E,F)&\lesssim \sum_{m_3=0}^{m_2-C}\sum_{m_1=0}^\infty\min\{2^{\frac{-m_3}{2}}2^{\frac{m_1}{2}[J-r-1]}2^{\frac{-m_2}{4}}|E|^{\frac{3}{4}}|F|^{\frac{3}{4}},2^{-(r+1)m_1}2^{-m_2}2^{-m_3}|F|\}\nonumber
\\
&\lesssim \sum_{m_3=0}^{m_2-C}2^{-m_2[\frac{2J-r-1}{2(J+r+1)}]}2^{-m_3[\frac{J-2r-2}{J+r+1}]}|E|^{\frac 32\frac{r+1}{J+r+1}}|F|^{1-\frac 12\frac{r+1}{J+r+1}} \nonumber
\\
&\lesssim 2^{-m_2[\frac{\min(2J-r-1,4J-5r-5)}{2(J+r+1)}]}|E|^\frac{3}{d_\omega+4}|F|^{1-\frac{1}{d_\omega+4}},
\end{align*}
where we used $d_h=\frac J{r+1}$ from \eqref{refinementtwisted dh} and $d_\omega=2d_h-2$ from Lemma \ref{domega dh} in the last line. Next, since $\omega=Cx^Qy^T+\scriptO(y^{T+1})$ by \eqref{eqn:AsymOmega}, then $d_\omega=\frac{Q+Tr}{r+1}>\frac{Tr}{r+1}\geq\frac{2}{3}T\geq \frac 23$, and the relationship $d_\omega=2d_h-2$ implies that $\frac{J}{r+1}=d_h>\frac 43$. Thus, $\frac{\min(2J-r-1,4J-5r-5)}{2(J+r+1)}>0$, so $\scriptT_{m_2,2}$ satisfies Proposition \ref{refinementtwistedm2prop}.

\vspace{.5pc}
Finally, consider the case $\frac{\alpha_{\mathbf{m}}}{x_\mathbf{m}y_\mathbf{m}}\lesssim \frac{y_\mathbf{m}}{x_\mathbf{m}^r}$. By the definition of $\alpha_\mathbf{m}$,
\begin{equation}\label{refinementtwisted 00 small}
\scriptT_{\mathbf{m}}(E,F)\lesssim x_\mathbf{m}^{r+1}(\tfrac{y_\mathbf{m}}{x_\mathbf{m}^r})^2|F|,
\end{equation}
and by equation \eqref{eqn:AsymOmega} and Theorem \ref{Gressman},
\begin{equation}\label{refinementtwisted 3414 small}
\scriptT_{\mathbf{m}}(E,F)\lesssim x_\mathbf{m}^{-\frac 12(J-r-1)}(\tfrac{y_\mathbf{m}}{x_\mathbf{m}^r})^{-\frac T4}|E|^\frac 34|F|^\frac 34.
\end{equation}
Then, combining \eqref{refinementtwisted 00 small} and \eqref{refinementtwisted 3414 small},
\begin{align*}
\scriptT_{m_2,1}(E,F)&\lesssim \sum_{m_1=0}^\infty \min\{2^{-(r+1)m_1}2^{-2m_2}|F|,2^{\frac{m_1}{2}[J-r-1]}2^\frac{m_2T}{4}|E|^\frac{3}{4}|F|^\frac{3}{4}\}
\\
&\lesssim 2^{-m_2\frac{4J-4r-4-T(r+1)}{2[J+r+1]}}|E|^{\frac 32\frac{r+1}{J+r+1}}|F|^{1-\frac 12\frac{r+1}{J+r+1}} \\
&= 2^{-m_2\frac{2Q+T(r-1)}{2[J+r+1]}}|E|^\frac{3}{d_\omega+4}|F|^{1-\frac 1{d_\omega+4}}.
\end{align*}
where we used the relation $2(J-r-1)=Q+Tr$ shown in \eqref{eqn:AsymOmega}, as well as $d_h=\frac J{r+1}$ from \eqref{refinementtwisted dh} and $d_\omega=2d_h-2$ from Lemma \ref{domega dh} in the last line. Thus, $\scriptT_{m_2,1}$ satisfies Proposition \ref{refinementtwistedm2prop}, so the proof of Proposition \ref{refinementtwistedm2prop} is complete.
\end{proof}
 Then, since $m_2\geq 0$ on $R_T$, we can sum over all $m_2$ to get
$$
\scriptT_{R_T}(E,F)\leq \sum_{m_2=0}^\infty\scriptT_{m_2}(E,F)\lesssim |E|^\frac 3{d_\omega+4}|F|^{1-\frac 1{d_\omega+4}}.
$$
Thus, $\scriptT_{R_T}$ is of rwt $(\frac{d_\omega+4}{3},d_\omega+4)$, concluding the proof of Proposition \ref{Vector1twistprop}. \qed

 \section{Second Relevant Vertex: Case breakdown}\label{S:Second Relevant Vertex Intro}
In Lemma \ref{relevant vertices}, we found that the polygon arising in Theorem \ref{Main Theorem Alt} has at most two relevant vertices. With Theorem \ref{Main Theorem Alt} now proven for $\scriptT_{[-1,1]^2\backslash \scriptR_T}$, and proven at the first relevant vertex for $\scriptT_{R_T}$, we now turn our attention to proving Theorem \ref{Main Theorem Alt} for $\scriptT_{R_T}$ at the second relevant vertex. We recall that the second vertex, $(\frac 1{p_{v_2}},\frac 1{q_{v_2}})$, exists only in cases (N) and (A), and belongs in the set $\overline{Conv}\{(0,0),(\frac 34,\frac 14),(\frac 23,\frac 13)\}$, either in the interior or on the boundary. Based on the location of $(\frac 1{p_{v_2}},\frac 1{q_{v_2}})$, we further decompose cases (N) and (A).
 
 \begin{definition}\label{Vertex2CasesDef}
Define $\scriptD:=\overline{Conv}\{(0,0),(\frac 34,\frac 14),(\frac 23,\frac 13)\}$. Depending on the location of $(\frac 1{p_{v_2}},\frac 1{q_{v_2}})$, we break Cases $(N)$ and $(A)$ into the following cases:

$(N_{Int}),(A_{Int})$: \hspace{.5pc}$\quad (\frac 1{p_{v_2}},\frac 1{q_{v_2}})\in Int(\scriptD);$

$(N_{q=p'}),(A_{q=p'})$: $\quad q_{v_2}=p_{v_2}', (\frac 1{p_{v_2}},\frac 1{q_{v_2}})\neq (\frac 23,\frac 13);$

$(N_{(\frac 23,\frac 13)})$: \hspace{.7pc}$\qquad \quad (\frac 1{p_{v_2}},\frac 1{q_{v_2}})= (\frac 23,\frac 13);$

$(N_{q=2p}^{scal})$: \hspace{.7pc} $\qquad \quad q_{v_2}=2p_{v_2}\neq 3$, $\frac 1{q_{v_2}}=\frac 1{p_{v_2}}-\frac 1{d_h+1};$

$(N_{q=2p}^{\neq scal})$: \hspace{.7pc} $\qquad \quad  q_{v_2}=2p_{v_2}\neq 3$, $\frac 1{q_{v_2}}\neq \frac 1{p_{v_2}}-\frac 1{d_h+1}$.

\end{definition}

We will show in the following lemma that the above subcases completely cover cases (N) and (A), and we will furthermore identify qualities of the functions $\varphi$ that belong to each subcase.
 
\begin{lemma}\label{Vertex2Cases}
Every $\varphi$ belonging to Cases $(N)$ or $(A)$ belongs to one of the subcases of Definition \ref{Vertex2CasesDef}. Additionally, for each subcase, the following hold, where we refer to the line $\frac 1q=\frac 1p-\frac 1{d_h+1}$ as the scaling line:
\\
$(A_{Int})$: $(\frac 1{p_{v_2}},\frac 1{q_{v_2}})$ lies on the scaling line, when $\frac 1{p_{v_2}}=\frac{(2T+5)-(d_h+1)}{(d_h+1)(T+2)}$.
\\
$(N_{Int})$: $(\frac 1{p_{v_2}},\frac 1{q_{v_2}})$ lies on the scaling line, when $\frac 1{p_{v_2}}=\frac{N+1-d_h}{d_h+1}$. Additionally, in 

this case $N<d_h+1$.
\\
$(N_{q=2p}^{scal}): (\frac 1{p_{v_2}},\frac 1{q_{v_2}})=(\frac 2N,\frac 1N)$, with $N=d_h+1$, and $N>3$.
\\
$(N_{q=2p}^{\neq scal}): (\frac 1{p_{v_2}},\frac 1{q_{v_2}})=(\frac 2N,\frac 1N)$, with $N>d_h+1$, and $N>3$
\\
$(N_{(\frac 23,\frac 13)}): (\frac 1{p_{v_2}},\frac 1{q_{v_2}})=(\frac 23,\frac 13)$, and up to a rescaling and a reordering of $z_1$ and $z_2$, 

$\varphi=(z_2-z_1^2)^3$.
\\
$(N_{q=p'}): (\frac 1{p_{v_2}},\frac 1{q_{v_2}})$ lies where the scaling line intersects the line $q=p'$, and up to 

a rescaling and a reordering of $z_1$ and $z_2$, $\varphi=(z_2-z_1^2)^2$.
\\
$(A_{q=p'}): (\frac 1{p_{v_2}},\frac 1{q_{v_2}})$ lies where the scaling line intersects the line $q=p'$, and up to 

a rescaling and a reordering of $z_1$ and $z_2$, $\varphi=z_1^2\pm z_2^S$, for some $S\geq 3$.
\end{lemma}  
 %
%
%
%
%
%
%
%
\begin{proof} In case (A), by Corollary \ref{multiplicities connection}, $T=A-2$. Thus, lines $\frac{1}{q}=\frac{A+1}{2A+1}\frac{1}{p}-\frac{1}{2A+1}=\frac{T+3}{2T+5}\frac{1}{p}-\frac{1}{2T+5}$ and $\frac{1}{q}=\frac{1}{p}-\frac{1}{d_h+1}$ intersect when 
$$
(\tfrac 1{p_{v_2}},\tfrac 1{q_{v_2}})=(\tfrac{(2T+5)-(d_h+1)}{(d_h+1)(T+2)},\tfrac{(T+3)-(d_h+1)}{(d_h+1)(T+2)}),
$$
which proves the statement for subcase $(A_{Int})$. Since $d_h>0$ and $T>d_{\omega}=2d_h-2$, $2<\tfrac{q_{v_2}}{p_{v_2}}<3$ is always satisfied. Additionally, the inequality $\frac{1}{p_{v_2}}+\frac{1}{q_{v_2}}<1$ is equivalent to 
\begin{equation}\label{location...1}
d_h>\tfrac{2(T+2)}{T+4},
\end{equation}
which always holds when $d_h\geq 2$. To find when it fails, since $\varphi$ is not homogeneous in Case (A) by Proposition \ref{Homogeneous Cases}, then by the definition of $A$, up to a rescaling and an interchanging of $z_1$ and $z_2$,
$$
\varphi=z_1^J+c_1z_2^{ls}z_1^{J-lr}+o(z_2^{ls}), \textrm{ some } J,l\geq 1, c_1\neq 0,
$$
where $A=ls\geq 2$, and by Corollary \ref{multiplicities connection}, $T=ls-2$. Then, \eqref{location...1} is equivalent to $d_h>\frac{2ls}{ls+2}$. Also, from the structure of $\varphi$, $d_h=\frac{Js}{s+r}$ and $J\geq \min(lr,2)$ ($J$ cannot be $1$ since $\nabla \varphi(0,0)=0$). When $lr=1$, we have $l=r=1$, and 
$$
d_h= \tfrac{Js}{s+r}\geq \tfrac{2s}{s+1}>\tfrac{2s}{s+2}=\tfrac{2ls}{ls+2}.
$$
When $lr> 2$, then $\frac r{s+r}>\frac 2{ls+2}$, and
$$
d_h= \tfrac{Js}{s+r}\geq \tfrac{lsr}{s+r}>\tfrac{2ls}{ls+2}.
$$
When $lr=2$ and $J>2$, then $\frac{r}{s+r}=\frac{2}{ls+2}$, so
$$
d_h= \tfrac{Js}{s+r}>\tfrac{lsr}{s+r}=\tfrac{2ls}{ls+2}.
$$
Thus, $q_{v_2}$ is always greater than $p_{v_2}'$ when either $lr\neq 2$ or when $lr=2$ and $J\neq 2$. However, if $J=lr=2$,  then $d_h=\frac{lsr}{s+r}=\frac{2ls}{ls+2}$, implying that $q_{v_2}=p_{v_2}'$. And since $J=lr=2$, after rescaling, $\varphi$ can be written as
$$
\varphi=z_1^2\pm z_2^S
$$
for some $S\geq 3$, as was claimed for subcase $(A_{q=p'})$. This completes the proof for the subcases of Case(A).

For case (N), when $N< d_h+1$, lines $\frac{1}{q}=\frac{N+1}{N+2}\frac{1}{p}-\frac{1}{N+2}$ and $\frac{1}{q}=\frac{1}{p}-\frac{1}{d_h+1}$ intersect when 
$$
(\tfrac{1}{p_{v_2}},\tfrac{1}{q_{v_2}})=(\tfrac{(N+2)-(d_h+1)}{d_h+1},\tfrac{(N+1)-(d_h+1)}{d_h+1}),
$$
as stated for subcase $(N_{Int})$. Then 
$$
\tfrac{q_{v_2}}{p_{v_2}}=\tfrac{(N+2)-(d_h+1)}{(N+1)-(d_h+1)}
$$
satisfies $\tfrac{q_{v_2}}{p_{v_2}}<3$ since $2N-3=T>d_\omega=2d_h-2$, and $\tfrac{q_{v_2}}{p_{v_2}}>2$ is satisfied since $N<d_h+1$.

Furthermore, since $\frac{1}{p_{v_2}}+\frac{1}{q_{v_2}}=\tfrac{(2N+3)-(2d_h+2)}{d_h+1}$, then the inequality $\frac{1}{p_{v_2}}+\frac{1}{q_{v_2}}<1$ is equivalent to 
$$
d_h>\tfrac{2}{3}N.
$$
First, if $N\geq 3$, then $N<d_h+1$ implies that $d_h>\frac 23 N$. Next, $d_h$ satisfies $d_h=\frac{rN+k}{r+1}$ for some $k\geq 0$ in $\N\cup \{0\}$. When $r>2$ or $k>0$, then 
$$
d_h=\tfrac{rN+k}{r+1}>\tfrac{2}{3}N.
$$
Therefore, when $N<d_h+1$, if $N\neq 2$, or $k\neq 0$, or $r\neq 2$, then $q_{v_2}>p_{v_2}'$. 

However, when $r=2$, $k=0$, and $N=2$, then $N<d_h+1$ and $d_h=\frac 23 N$, implying that  $q_{v_2}=p_{v_2}'$ and that up to rescaling and a swapping of $z_1$ and $z_2$, 
$$
\varphi=(z_2-z_1^2)^2,
$$
as was claimed for subcase $(N_{q=p'})$.

Finally, for case (N), with $N\geq d_h+1$, the lines $\frac{1}{q}=\frac{N+1}{N+2}\frac{1}{p}-\frac{1}{N+2}$ and $\frac{1}{q}=\frac{1}{p}-\frac{1}{N}$ intersect when $(\frac 1{p_{v_2}},\frac 1{q_{v_2}})=(\frac 2N,\frac 1N)$, as stated for subcases $(N_{q=2p}^{scal})$ and $(N_{q=2p}^{\neq scal})$. Thus, when $N\geq d_h+1$, $\frac{q_{v_2}}{p_{v_2}}=2$ is always satisfied.

 If $N> 3$, then $\frac{1}{p_{v_2}}+\frac 1{q_{v_2}}<1$. Now, for some $k\geq 0$ in $\N\cup \{0\}$, depending on $\varphi$, $d_h=\frac{rN+k}{r+1}$. Thus, when either $N>d_h+1$, or when $N=d_h+1$ and $l\neq 0$ or $r\neq 2$, $N$ must be larger than $3$ implying that $\frac{1}{p_{v_2}}+\frac 1{q_{v_2}}<1$.
 
Alternatively, when $N=d_h+1$, $r=2$, $k=0$, and $N=3$, then $(\frac 1{p_{v_2}},\frac 1{q_{v_2}})=(\frac 23,\frac 13)$, and up to a rescaling and a swapping of $z_1$ and $z_2$,
$$
\varphi=(z_2-z_1^2)^3,
$$
as was claimed for subcase $(N_{(\frac 23,\frac 13)})$. This completes the proof for the subcases of Case(N).
\end{proof}

 
 \section{The rwt $(\frac 32, 3)$ bound and the second relevant vertex.}\label{S:Second Relevant Vertex Conclusion}

In this section, we will go through each subcase from Definition \ref{Vertex2CasesDef}, proving that $\scriptT_{R_T}$ is of rwt $(p_{v_2},q_{v_2})$ and completing the proof of Theorem \ref{Main Theorem Alt}.
 \subsection{Case $(A_{q=p'})$}\label{SS:Second Relevant Vertex 1}
By Lemma \ref{Vertex2Cases}, after rescaling, we may assume
$$
\varphi=t_1^2\pm t_2^S,
$$
where $S\geq 3$. By symmetry, it suffices to consider $\scriptT$ over the set $\{t_1,t_2\geq 0\}$. We decompose $(\R_+)^2$ into strips $S_j=\{t_2\approx 2^{-j}\}$. This induces a decomposition $\scriptT=\sum \scriptT_j$. We now bound $\norm{T_j}_{L^\frac 32\rightarrow L^3}$ using Minkowski's Inequality, Theorem \ref{Gressman} on $\R^2$, and Young's Inequality as follows:
\begin{align}\label{As....1}
\norm{\scriptT_jf}_{L^3(\R^3)}&=\norm{f(x- (t,\varphi(t)))}_{L^3_xL^1_t(\{t\in S_j\})} \nonumber\\ 
&\lesssim \norm{f(x-(t,\varphi(t)))|\partial_{t_1}^2\varphi|^\frac 13}_{L^3_{x_2}L^1_{t_2}L^3_{x_1x_3}L^1_{t_1}(\{t\in S_j\})}\nonumber\\
&\lesssim \norm{f(y_1,x_2-t_2,y_3)}_{L^3_{x_2}L^1_{t_2}L^\frac 32_{y_1y_3}(\{t_2\sim 2^{-j}\})}\nonumber\\
&=\norm{(\norm{f(y_1,\cdot,y_3)}_{L^\frac{3}{2}_{y_1y_3}}*\chi_{\{\cdot\approx 2^{-j}\}})}_{L^3}\nonumber\\
&\leq \norm{\chi_{\{\cdot\approx 2^{-j}\}}}_{L^\frac{3}{2}}\norm{f}_{L^\frac{3}{2}(\R^3)}
=2^{-\frac{2}{3}j}\norm{f}_{L^\frac{3}{2}(\R^3)}.
\end{align}



\vspace{.5pc}

Since $|\omega|:=|det D^2\varphi|\approx t_2^{S-2}\approx 2^{-(S-2)j}$ on $S_j$, Theorem \ref{Gressman} implies
\begin{equation}\label{As....2}
 \norm{\scriptT_jf}_{L^4(\R^3)}\lesssim 2^{\frac{(S-2)j}{4}} \norm{f}_{L^\frac{4}{3}(\R^3)}.
\end{equation}




Combining \eqref{As....1} and \eqref{As....2},
$$
\scriptT_{(\R_+)^2}(E,F)\lesssim\sum_{j=-\infty}^\infty \min(2^{-\frac{2}{3}j}|E|^\frac{2}{3}|F|^\frac{2}{3},2^{\frac{(S-2)j}{4}}|E|^\frac{3}{4}|F|^\frac{3}{4})\approx |E|^\frac{2S+2}{3S+2}|F|^\frac{2S+2}{3S+2}.
$$

Thus, $\scriptT_{\R^2}$ is of rwt $(p_S,q_S)$ for $(\frac 1{p_S},\frac 1{q_S})=(\frac{2S+2}{3S+2},\frac S{3S+2})$. As $\frac 1{p_S}+\frac 1{q_S}=1$ and $(p_S,q_S)$ lies on the scaling line of $\scriptT$, then $(p_S,q_S)=(p_{v_2},q_{v_2})$ by Lemma \ref{Vertex2Cases}. \qed
 \vspace{.75pc}

\subsection{Case $(N_{q=p'})$} By Lemma \ref{Vertex2Cases}, after rescaling, we may assume
$$
\varphi=(t_2-t_1^2)^2.
$$
We decompose $\R^2$ into strips $S_j=\{|t_2-t_1^2|\approx 2^{-j}\}$, which induces a decomposition $\scriptT=\sum \scriptT_j$. By the change of coordinates $u_1:=t_1$, $u_2:=t_2-t_1^2$, (setting $\psi(u_1,u_2):=u_2+u_1^2$), Minkowski's Inequality, Theorem \ref{Gressman} (on $\R^2$), the change of variables $v=u_2^2$, and Young's Inequality,
\begin{align}\label{Ns....1}
\norm{\scriptT_jf}_{L^3(\R^3)}&:=\norm{f(x-(t,\varphi(t))}_{L^3_xL^1_t(\{t\in S_j\})} \nonumber\\
&\approx \norm{f(x-(u_1,\psi(u),u_2^2))|\partial_{u_1}^2\psi|^\frac 13}_{L^3_xL^1_u(\{|u_2|\sim 2^{-j}\})}\nonumber\\ 
&\leq \norm{f(x-(u_1,\psi(u),u_2^2))|\partial_{u_1}^2\psi|^\frac 13}_{L^3_{x_3}L^1_{u_2}L^3_{x_1x_2}L^1_{u_1}(\{|u_2|\sim 2^{-j}\})}\nonumber\\
&\lesssim \norm{f(y_1,y_2,x_3-u_2^2)}_{L^3_{x_3}L^1_{u_2}L^\frac 32_{y_1y_2}(\{|u_2|\sim 2^{-j}\})} \nonumber\\
&\approx 2^j\norm{f(y_1,y_2,x_3-v)}_{L^3_{x_3}L^1_{v}L^\frac{3}{2}_{y_1y_2}(\{v\sim 2^{-2j}\})}\nonumber\\
&=2^j\norm{\norm{f(y_1,y_2,\cdot)}_{L^\frac{3}{2}_{y_1y_2}}*\chi_{\{|\cdot|\approx 2^{-2j}\}}}_{L^3}\nonumber\\
&\leq 2^j\norm{\chi_{\{|\cdot|\approx 2^{-2j}\}}}_{L^\frac{3}{2}}\norm{f}_{L^\frac{3}{2}(\R^3)}
=2^j2^{-\frac{4j}{3}}\norm{f}_{L^\frac{3}{2}(\R^3)}
=2^{-\frac{j}{3}}\norm{f}_{L^\frac{3}{2}(\R^3)}.
\end{align}

\vspace{.5pc}

Since $|\omega|:=|det(D^2\varphi)|\approx |t_2- t_1^2|\approx 2^{-j}$ on $S_j$, Theorem \ref{Gressman} implies
\begin{equation}\label{Ns....2}
 \norm{\scriptT_jf}_{L^4(\R^3)}\lesssim 2^{\frac{j}{4}} \norm{f}_{L^\frac{4}{3}(\R^3)}.
\end{equation}
Combining \ref{Ns....1} and \ref{Ns....2}, 
$$
\scriptT_{\R^2}(E,F)\lesssim\sum_{j=-\infty}^\infty\min(2^{-\frac{1}{3}j}|E|^\frac{2}{3}|F|^\frac{2}{3},2^{\frac{j}{4}}|E|^\frac{3}{4}|F|^\frac{3}{4})\approx |E|^\frac 57|F|^\frac 57.
$$
Thus, $\scriptT_{\R^2}$ is of rwt $(\frac 75,\frac 72)$. As $\frac 57+\frac 27=1$, and as $(\frac 75,\frac 72)$ lies on the scaling line of $\scriptT$, then $(\frac 75,\frac 72)=(p_{v_2},q_{v_2})$ by Theorem \ref{Vertex2Cases}.

\vspace{0.1pc}
 
 \subsection{Case $(N_{(\frac 23,\frac 13)})$} By Lemma \ref{Vertex2Cases}, after rescaling, we may assume
$$
\varphi=(t_2-t_1^2)^3.
$$
By the change of variable $u_1:=t_1$, $u_2:=t_2-t_1^2$ (setting $\psi(u_1,u_2):=u_2+u_1^2$), Minkowski's inequality, Theorem \ref{Gressman} (on $\R^2$), the change of variables $v=u_2^3$, and Young's Inequality,
\begin{align*}
\norm{\scriptT f}_{L^3(\R^3)}&=\norm{f(x-(t,\varphi(t)))}_{L^3_xL^1_t} \\
&\approx \norm{f(x-(u_1,\psi(u),u_2^3))|\partial_{u_1}^2\psi|^\frac 13}_{L^3_xL^1_u}\\ 
&\leq \norm{f(x-(u_1,\psi(u),u_2^3))|\partial_{u_1}^2\psi|^\frac 13}_{L^3_{x_3}L^1_{u_2}L^3_{x_1x_2}L^1_{u_1}}\\
&\lesssim \norm{f(y_1,y_2,x_3-u_2^3)}_{L^3_{x_3}L^1_{u_2}L^\frac 32_{y_1y_2}}\\
&\approx \norm{v^{-\frac{2}{3}}\norm{f(y_1,y_2,x_3-v)}_{L^\frac{3}{2}_{y_1y_2}}}_{L^3_{x_3}L^1_v}\\
&=\norm{\norm{f(y_1,y_2,\cdot)}_{L^\frac{3}{2}_{y_1y_2}}*(\cdot)^{-\frac{2}{3}}}_{L^3}\leq \norm{(\cdot)^{-\frac{2}{3}}}_{L^{\frac{3}{2},w}}\norm{f}_{L^\frac{3}{2}(\R^3)}
=\norm{f}_{L^\frac{3}{2}(\R^3)}.
\end{align*}

Thus, $\scriptT_{\R^2}$ has a strong type bound at $(p,q)=(\frac 32,3)$, which equals $(p_{v_2},q_{v_2})$ by Lemma \ref{Vertex2Cases}. \qed

\subsection{Case $(A_{Int})$} In Case(A), we can assume that $f_T=t_2$, so that $t_2$ has multiplicity $T$ in $\omega$ and multiplicity $0$ in $\varphi$. Thus, after rescaling, because of mixed homogeneity, $\varphi$ will take the following form: 
$$
\varphi = t_1^{J}\pm t_2^{ls}t_1^{J-lr}(1+\scriptO(\tfrac{t_2^s}{t_1^r})) \qtq{where} ls\geq 2, r\geq 1, J\geq 1,
$$
for some $l,J$. We recall from the definition of case(A) that $A=ls\geq 2$. Then, by Lemma \ref{A-2}, $\omega$ will take the form
$$
\omega=Ct_1^{2J-lr-2}t_2^{ls-2}(1+\scriptO(\tfrac{t_2^s}{t_1^r})),\quad  C\neq 0.
$$
Thus, $T=ls-2=A-2$, and we define $Q:=2J-lr-2$, so that $|\omega|\approx |t_1|^Q|t_2|^T$ on $R_T^e=\{|t_2|^s\leq \tilde\epsilon |t_1|^r\}$. By symmetry, it will suffice to consider only the region where $t_1,t_2\geq 0$.
We define $S:=R^e_T\cap \{|\omega|\approx 1\}\cap\{t_1,t_2\geq 0\}$, which decomposes into regions $\tau_j:=\{t_1\approx 2^{\frac jQ},t_2\approx 2^{-\frac jT}\}$ and induces the decomposition $\scriptT_{S}=\sum_jT_j$. Then, by Minkowski's Inequality, Theorem \ref{Gressman} (on $\R^2$), and Young's Inequality,
\begin{align}\label{Aint....1}
\norm{\scriptT_jf}_{L_3(\R^3)}&=\norm{f(x-(t,\varphi(t))}_{L^3_xL^1_t(\{t\in \tau_j\})} \nonumber\\
&\approx 2^{-\frac 13 \frac jQ(J-2)}\norm{f(x-(t,\varphi(t))|\partial_{t_1}^2\varphi|^\frac 13}_{L^3_xL^1_t(\{t\in \tau_j\})}\nonumber\\ 
&\leq 2^{-\frac 13 \frac jQ(J-2)}\norm{f(x-(t,\varphi(t))|\partial_{t_1}^2\varphi|^\frac 13}_{L^3_{x_2}L^1_{t_2}L^3_{x_1x_3}L^1_{t_1}(\{t\in \tau_j\})}\nonumber\\
&\lesssim 2^{-\frac 13 \frac jQ(J-2)}\norm{f(y_1,x_2-t_2,y_3)}_{L^3_{x_2}L^1_{t_2}L^\frac 32_{y_1y_3}(\{t_2\sim 2^{-j/T}\})}\nonumber\\
&=2^{-\frac 13 \frac jQ(J-2)}\norm{\norm{f(y_1,\cdot, y_3)}_{L^\frac{3}{2}_{y_1y_3}}*\chi_{\{\cdot\approx 2^{-\frac jT}\}}}_{L^3}\nonumber\\
&\leq 2^{-\frac{1}{3}\frac jQ(J-2)}\norm{\chi_{\{\cdot\approx 2^{-\frac jT}\}}}_{L^\frac{3}{2}}\norm{f}_{L^\frac{3}{2}(\R^3)}
=2^{-\frac{1}{3}\frac jQ(J-2)}2^{-\frac{2}{3}\frac jT}\norm{f}_{L^\frac{3}{2}(\R^3)}\nonumber\\
&=2^{-\frac{j}{3}[\frac{(J-2)}{Q}+\frac{2}{T}]}\norm{f}_{L^\frac{3}{2}(\R^3)}.
\end{align}

Additionally, since $|\tau_j|\approx 2^{(\frac 1Q-\frac 1T)j}$, we have
\begin{equation}\label{Aint....2}||\scriptT_j||_{\infty\rightarrow \infty}\leq 2^{(\frac 1Q-\frac 1T)j}.\end{equation}
Combining \eqref{Aint....1} and \eqref{Aint....2},
$$
\scriptT_j(E,F)\lesssim \min(2^{(\frac 1Q-\frac 1T)j}|F|,2^{-\frac{j}{3}[\frac{(J-2)}{Q}+\frac{2}{T}]}|E|^\frac{2}{3}|F|^\frac{2}{3}).
$$
Noting that $T-Q> 0$ since $T>d_\omega=\frac{Qs+Tr}{r+s}$, we interpolate as follows:
\begin{align*}
T_{S}(E,F)&\lesssim \sum_{j=-\infty}^{\infty} \min(2^{(\frac 1Q-\frac 1T)j}|F|,2^{-\frac{j}{3}[\frac{(J-2)}{Q}+\frac{2}{T}]}|E|^\frac{2}{3}|F|^\frac{2}{3})
\\
 &= |E|^\frac{2(T-Q)}{(J+1)T-Q}|F|^{1-\frac{T-Q}{(J+1)T-Q}}.
\end{align*}
Thus, $\scriptT_{R_T^e\cap\{|\omega|\approx 1\}}$ is of rwt $(\frac{q_I}{2},q_I)$ for $(\frac{2}{q_I},\frac{1}{q_I})=(\frac{2(T-Q)}{(J+1)T-Q},\frac{T-Q}{(J+1)T-Q})$. Using equalities $Qs+Tr=d_\omega (r+s)$, $d_h=\frac{Js}{r+s}$, and $d_\omega=2d_h-2$, and some messy arithmetic, we can rewrite $\frac 1{q_I}$ as
\begin{equation}\label{Aint....qI}
\tfrac{1}{q_I}=\tfrac{2(T-d_\omega)}{4(T-d_\omega)+(T+2)d_\omega}.
\end{equation}
Since $\scriptT_{R_T^e\cap\{|\omega|\approx 1\}}$ is of stong type $(\frac 43,4)$ by Theorem \ref{Gressman}, and of  rwt $(p_I, q_I)$, where $q_I=2p_I$, and since $R_T^e$ is \textbf{$\kappa_\varphi$}-scale invariant, then by Case 3 of Proposition \ref{scaling prop}, $\scriptT_{R_T^e}$ is of rwt $(p_S,q_S)$, when $(p_S,q_S)$ lies on the scaling line of $\scriptT$ and 
$$
\tfrac 1{p_S}=\tfrac{3-\frac 8{q_I}+\frac 12\frac 1{q_I}(d_\omega+4)}{(1-\frac 2{q_I})(d_\omega+4)}.
$$
Since $(p_S,q_S)$ lies on the scaling line, then by Lemma \ref{Vertex2Cases} it suffices to verify that $p_S=p_{v_2}$. Using \eqref{Aint....qI} and the identity $d_\omega=2d_h-2$, after some messy arithmetic, we find that
$$
\tfrac 1{p_S}=\tfrac{2T+4-d_h}{(T+2)(d_h+1)}.
$$
Thus, by Lemma \ref{Vertex2Cases}, $p_{v_2}=p_S$, and so $\scriptT_{R_T^e}$ is of rwt $(p_{v_2},q_{v_2})$. \qed
 
\subsection{Cases $(N_{Int})$, $(N_{q=2p}^{\neq scal})$, $(N_{q=2p}^{scal})$}   In Cases(N), we can assume after rescaling that $f_T=z_2-z_1^r$, so that $z_2-z_1^r$ has multiplicity $T$ in $\omega$ and multiplicity $N$ in $\varphi$. Thus, after rescaling, because of mixed homogeneity, $\varphi$ can take the following form:  
$$
\varphi=z_1^J(z_2- z_1^r)^N(1+\scriptO(\tfrac{z_2-z_1^r}{z_1^r})), \quad \text{ with }
d_h=\tfrac{J+rN}{r+1},$$
for some $J$. Then, by Lemma \ref{2N-3}, 
\begin{equation}\label{omegacaseN}
\omega=Cz_1^{2J+r-2}(z_2- z_1^r)^{2N-3}(1+\scriptO(\tfrac{z_2-z_1^r}{z_1^r})), \quad C\neq 0.
\end{equation}
Thus, $T=2N-3$.  We define $Q:=2J+r-2$, so that $|\omega|\approx |z_1|^Q|z_2- z_1^r|^T$ on $R_T^e=\{|z_2- z_1^r|<\tilde\epsilon |z_1|^r\}$. To simplify the argument, we will demonstrate the proof when $z_1, z_2-z_1^r\geq 0$. All other regions follow similarly. We decompose $R_T^e\cap \{z_1,z_2-z_1^r\geq 0\}$ into sets $\tau_{j,k}:=\{z_1\approx 2^{-j}:=\x_j, z_2- z_1^r\approx 2^{-k}:=\y_k\}$, which induces a decomposition $\scriptT_{R_T^e\cap\{z_1,z_2-z_1^r\geq 0\}}=\sum_{j,k}\scriptT_{j,k}$. For ease, we perform the change of variables $x=z_1, y=z_2- z_1^r$. To find the rwt $(\frac 32,3)$ bound in these cases, our previous methods will not work, so we instead use a variant of the method of refinements. 

\begin{lemma}\label{refinementreduced main} The operator $\scriptT_{j,k}$ satisfies the following rwt $(\frac 32,3)$ bound:
\begin{align*}
\mathcal{T}_{j,k}(E,F)&\lesssim 2^{\frac j3(J+r-2)}2^{\frac k3(N-3)}|E|^\frac{2}{3}|F|^\frac{2}{3} =\big(\x_j^{J+r-2}\y_k^{N-3}\big)^{-\frac 13}|E|^\frac 23|F|^\frac 23 \\
&= 2^{\frac j3[(r+1)(1+d_h)-(rN+3)]+\frac k3(N-3)}|E|^\frac 23|F|^\frac 23.
\end{align*}
\end{lemma}

\begin{proof} First, we will refine $E\rightsquigarrow E_{j,k}$ and $F\rightsquigarrow F_{j,k}$ as follows. Define
\begin{align*}
F_{j,k}&:=\{u\in F: \mathcal{T}_{j,k}\chi_E(u)\geq \tfrac{1}{4}\tfrac{\mathcal{T}_{j,k}(E,F)}{|F|}\}
\\
E_{j,k}&:=\{w\in E: \mathcal{T}_{j,k}^*\chi_{F_{j,k}}(w)\geq \tfrac{1}{4}\tfrac{\mathcal{T}_{j,k}(E,F_{j,k})}{|E|}=:\alpha_{j,k}\}
\\
F_{j,k}'&:=\{u\in F_{j,k}: \mathcal{T}_{j,k}\chi_{E_{j,k}}(u)\geq \tfrac{1}{4}\tfrac{\mathcal{T}_{j,k}(E_{j,k},F_{j,k})}{|F_{j,k}|}=:\beta_{j,k}\}.
\end{align*}
Note that $\alpha_{j,k}$ and $\beta_{j,k}$ also depend on $E$ and $F$. Then $\mathcal{T}_{j,k}(E_{j,k},F_{j,k})\approx  \mathcal{T}_{j,k}(E,F)$ by the short argument leading to \eqref{refinement approx}. To proceed, we will construct a map, and use the size of the Jacobian to prove the lemma. First, we define $\psi(x):= x^r$ and $z(t):= (t_1,t_2+\psi(t_1))$. Then, the surface equation has the following equivalent forms: $(z_1,z_2,\varphi(z_1,z_2))=(x,y+\psi(x),\varphi(x,y+\psi(x))=(z(x,y),\varphi(z(x,y))$. 

Fix $u_0\in F_{j,k}'$ and define 
$$
\Omega_1:=\{t\in \R^2: z(t)\in\tau_{j,k} \text{ and } u_0-(t_1,t_2+\psi(t_1),\varphi(z(t)))=:w(t)\in E_{j,k}\}. 
$$
Then $|\Omega_1|=\mathcal{T}_{j,k}\chi_{E_{j,k}}(u_0)\geq \beta_{j,k}$. To achieve a lower-dimensional result with the method of refinements, we fix one variable. Rewriting $|\Omega_1|$,
$$
\frac{|\Omega_1|}{2^{-k}}=\dashint_{t_2\approx 2^{-k}}\int_{t_1\approx 2^{-j}}\chi_{\Omega_1}(t_1,t_2)dt_1dt_2.
$$
Therefore, by H\"older's inequality, there exists a fixed $t_2\approx 2^{-k}$ such that
$$
|\Omega_{1,t_2}|:= |\{t_1:t\in \Omega_1\}|=\int_{t_1\approx 2^{-j}}\chi_{\Omega_1}(t_1,t_2)dt_1 \geq \tfrac{\beta_{j,k}}{2^{-k}}=\tfrac{\beta_{j,k}}{\y_k}.
$$
For $t_1\in \Omega_{1,t_2}$, define
$$
\Omega_{2,t_2}(t_1):=\{s\in \R^2: s\in \tau_{j,k} \text{ and } w(t)+(s_1,s_2+\psi(s_1),\varphi(z(s)))\in F_{j,k}\}. 
$$
Then $|\Omega_{2,t_2}(t_1)|=T_{j,k}^*\chi_{F_{j,k}}(w(t))\geq \alpha_{j,k}$.

Define $\Omega_{t_2} :=\{(t_1,s_1,s_2)\in \R^3: t_1\in \Omega_{1,t_2}, s=(s_1,s_2)\in \Omega_{2,t_2}(t_1)\}$, and
\begin{align*}
\Psi_{t_2}(t_1,s_1,s_2)&:=x_0-(t_1,t_2+\psi(t_1),\varphi(z(t)))+(s_1,s_2+\psi(s_1),\varphi(z(s))).
\end{align*}
Since $\Psi_{t_2}$ is a polynomial mapping $\Omega\subset\R^3$ into $\R^3$, it is $\scriptO(1)$-to-one off a set of measure zero. Since $\Psi_{t_2}(\Omega_{t_2})\subset F_{j,k}$, we have $|F_{j,k}|\gtrsim \int_{\Omega_{t_2}} |\det D\Psi_{t_2}(t_1,s_1,s_2)|dt_1ds_1ds_2$. Expanding,
\begin{align*}
|\det D\Psi_{t_2}(t_1,s_1,s_2)|&=|[\psi'(t_1)-\psi'(s_1)]\partial_{s_2}\varphi(z(s))-[\partial_{t_1}\varphi(z(t))-\partial_{s_1}\varphi(z(s))] |
\\
 &=|\psi''(\tilde{x})\partial_{s_2}\varphi(z(s))(t_1-s_1)-(\partial_{v_1},\partial_{v_2})\partial_{v_1}\varphi(z(v))\cdot (t-s) |
\\
&=|[\psi''(\tilde{x})\partial_{s_2}\varphi(z(s))-\partial^2_{v_1}\varphi(z(v))](t_1-s_1)-\partial_{v_1v_2}\varphi(z(v))(t_2-s_2)|,
\end{align*}
for some $\tilde{x}$ between $t_1$ and $s_1$, and some $v$ lying on the line between $t$ and $s$, by the Mean Value Theorem.

Comparing each term, and recalling definition $\x_j:=2^{-j}$, $\y_k:=2^{-k}$:
\begin{align*}
|\psi''(\tilde{x})\partial_{s_2}\varphi(z(s))|&\approx \x_j^{r-2}\x_j^J\y_k^{N-1}=\x_j^{J+r-2}\y_k^{N-1};
\\
|\partial_{v_1v_1}\varphi(z(v))|&\approx \x_j^{J-2}\y_k^N;
\\
|\partial_{v_1v_2}\varphi(z(v))|&\approx \x_j^{J-1}\y_k^{N-1}.
\end{align*}
Since $\y_k<\tilde\epsilon \x_j^r$ on $R^e_T$, we have $|\psi''(\tilde{x})\partial_{s_2}\varphi(z(s))|\gg |\partial_{v_1v_1}\varphi(z(v))|$, so we can disregard the $\partial_{v_1v_1}\varphi(z(v))$ term. Comparing the remaining terms, as long as 
\begin{equation}\label{refinementreduced condition}
|s_1-t_1|\x_j^{r-1}\gg |s_2-t_2|
\end{equation}
the $\psi''(\tilde{x})\partial_{s_2}\varphi(z(s))(t_1-s_1)$ term dominates, and we get
\begin{equation}\label{refinementreduced condition a}
|\det D\Psi_{t_2}(t_1,s_1,s_2)|\approx |s_1-t_1|\x_j^{J+r-2}\y_k^{N-1}.
\end{equation}
Let $t_1$ be fixed. Since $|\Omega_{2,t_2}(t_1)|\geq \alpha_{j,k}$ and since $s\in \Omega_{2,t_2}(t_1)$ implies that $s_2\approx \y_k$, then on over half of $\Omega_{2,t_2}(t_1)$, we have $|s_1-t_1|\gtrsim \frac{\alpha_{j,k}}{\y_k}$. Also, since $s_2, t_2\approx 2^{-k}=\y_k$, we get $|s_2-t_2|\lesssim \y_k$.

Thus, based on \eqref{refinementreduced condition}, we have two cases: $\frac{\alpha_{j,k}}{\y_k}\x_j^{r-1}\gg \y_k$ and $\frac{\alpha_{j,k}}{\y_k}\x_j^{r-1}\lesssim \y_k$. 

\vspace{.75pc}

In the case $\frac{\alpha_{j,k}}{\y_k}\x_j^{r-1}\gg\y_k$, \eqref{refinementreduced condition} holds on over half of $\Omega_{2,t_2}(t_1)$, and
$$
|\det D\Psi_{t_2}(t_1,s_1,s_2)|\approx |s_1-t_1|\x_j^{J+r-2}\y_k^{N-1}\gtrsim \tfrac{\alpha_{j,k}}{\y_k}\x_j^{J+r-2}\y_k^{N-1}
$$
on over half of $\Omega_{2,t_2}(t_1)$, implying that
\begin{align*}
|F_{j,k}|&\gtrsim \int_{\Omega_{t_2}} \tfrac{\alpha_{j,k}}{\y_k}\x_j^{J+r-2}\y_k^{N-1}dt_1ds_1ds_2\gtrsim \alpha_{j,k} \x_j^{J+r-2}\y_k^{N-2}\min_{t_1}(|\Omega_{2,t_2}(t_1)|)|\Omega_{1,t_2}| \\
&\gtrsim \alpha_{j,k} \x_j^{J+r-2}\y_k^{N-2}\alpha_{j,k}\tfrac{\beta_{j,k}}{\y_k}=\x_j^{J+r-2}\y_k^{N-3}\alpha_{j,k}^2\beta_{j,k}.
\end{align*}
By definition, $\beta_{j,k}\approx \frac{ \mathcal{T}_{j,k}(E,F)}{|F_{j,k}|}$ and $\alpha_{j,k}\approx \frac{\mathcal{T}_{j,k}(E,F)}{|E|}$, and so
$$|F_{j,k}|\gtrsim \x_j^{J+r-2}\y_k^{N-3}\alpha_{j,k}^2\beta_{j,k}\gtrsim \x_j^{J+r-2}\y_k^{N-3}\tfrac{\mathcal{T}_{j,k}(E,F)^3}{|E|^2|F_{j,k}|}.$$
 Then, using $|F_{j,k}|\leq |F|$, $d_h=\frac{J+rN}{r+1}$, $\x_j:=2^{-j}$, and $\y_k:=2^{-k}$, we conclude
\begin{align*}
\mathcal{T}_{j,k}(E,F)&\lesssim (\x_j^{J+r-2}\y_k^{N-3})^{-\frac 13}|E|^\frac{2}{3}|F|^\frac{2}{3}
\\
&=2^{\frac j3[(r+1)(1+d_h)-(rN+3)]+\frac k3(N-3)}|E|^\frac 23|F|^\frac 23 \\
&=2^{-j}2^{-k}2^{-\frac j3[(r+1)(d_h+1)-rN]}2^{\frac {kN}3}|E|^\frac 23|F|^\frac 23,
\end{align*}
concluding the first case where $\frac{\alpha_{j,k}}{\y_k}\x_j^{r-1}\gg \y_k$.

\vspace{.75pc}

For the second case, we combine  $\frac{\alpha_{j,k}}{\y_k}\x_j^{r-1}\lesssim \y_k$ and $\scriptT_{j,k}(E,F)\lesssim \alpha_{j,k} |E|$ to get the $L^1\rightarrow L^1$ bound 
\begin{equation}\label{refinementreduced 11}
\scriptT_{j,k}(E,F)\lesssim \y_k^2 \x_j^{-(r-1)}|E|.
\end{equation}
Since $\tau_{j,k}$ has measure $\x_j\y_k$, Young's Inequality gives an $L^\infty\rightarrow L^\infty$ bound of
\begin{equation}\label{refinementreduced 00}
\scriptT_{j,k}(E,F)\lesssim \x_j\y_k|F|.
\end{equation}
Interpolating \eqref{refinementreduced 11} and \eqref{refinementreduced 00}, we get the $L^2\rightarrow L^2$ bound
\begin{equation}\label{refinementreduced 22}
\scriptT_{j,k}(E,F)\lesssim \y_k^{\frac{3}{2}}\x_j^{-\frac{1}{2}(r-2)}|E|^\frac 12|F|^\frac 12.
\end{equation}
Next, by $\eqref{omegacaseN}$, we know that on $\tau_{j,k}\subset R^e_T$,
$$
|\omega|\approx \x_j^{2J+r-2} \y_k^{2N-3},
$$
so by Theorem \ref{Gressman}, we have the $L^\frac 43\rightarrow L^4$ bound
\begin{equation}\label{refinementreduced 4}
\scriptT_{j,k}(E,F)\lesssim (\x_j^{2J+r-2} \y_k^{2N-3})^{-\frac 14}|E|^\frac 34|F|^\frac 34.
\end{equation} 
Finally, interpolating \eqref{refinementreduced 22} and \eqref{refinementreduced 4},
\begin{align*}
\scriptT_{j,k}(E,F)&\lesssim \big(\y_k^{\frac{3}{2}}\x_j^{-\frac{1}{2}(r-2)}|E|^\frac 12|F|^\frac 12\big)^\frac 13\big((\x_j^{2J+r-2} \y_k^{2N-3})^{-\frac 14}|E|^\frac 34|F|^\frac 34\big)^\frac 23 \\
&=\big(\x_j^{J+r-2}\y_k^{N-3}\big)^{-\frac 13}|E|^\frac 23|F|^\frac 23,
\end{align*}
as in the first case. Thus, the proof of Lemma \ref{refinementreduced main} is complete.
\end{proof} 
 
\subsubsection{Case $(N_{Int})$} We start by defining $\tilde S:=R_T^e\cap \{|\omega|\approx 1\}\cap\{z_1,z_2-z_1^r\geq 0\}$, which we can decompose into regions $\tau_{\frac {-j}Q,\frac jT}:=\{z_1=:x\sim 2^{\frac{j}{Q}},z_2-z_1^r=:y\sim 2^{-\frac jT}\}$, recalling that $T=2N-3$ and $Q=2J+r-2$ from \eqref{omegacaseN}. This induces the decomposition $\scriptT_{\tilde S}=\sum \scriptT_j$.
From Lemma \ref{refinementreduced main}, replacing $j$ with $\frac {-j}Q$ and $k$ with $\frac jT$,
\begin{align}\label{Nint....1}
\scriptT_{j}(E,F)&\lesssim 2^{-\frac{j}{3Q}[(r+1)(1+d_h)-(rN+3)]+\frac{j}{3T}(N-3)}|E|^\frac 23|F|^\frac 23
\nonumber\\
&=2^{-\frac j3[\frac BQ-\frac {N-3}T]}|E|^\frac 23|F|^\frac 23,
\end{align}
 where $B=(r+1)[1+d_h]-(rN+3)$. Since $|\tau_{\frac{-j}{Q},\frac{j}{T}}|\approx 2^{(\frac 1Q-\frac 1T)j}$, we also have
\begin{equation}\label{Nint....2}
\scriptT_{j}(E,F)\lesssim 2^{(\frac 1Q-\frac 1T)j}|F|.
\end{equation} 
Combining \eqref{Nint....1} and \eqref{Nint....2}, 
$$
\scriptT_j(E,F)\lesssim \min(2^{(\frac 1Q-\frac 1T)j}|F|,2^{-\frac{j}{3}(\frac BQ-\frac{N-3}T)}|E|^\frac{2}{3}|F|^\frac{2}{3}).
$$
Noting that $T-Q>0$ since $T>d_\omega=\frac{Q+Tr}{r+1}$, we interpolate as follows:
\begin{align*}
T_{\tilde S}(E,F)&\lesssim \sum_j\min(2^{(\frac 1Q-\frac 1T)j}|F|,2^{-\frac{j}{3}(\frac BQ-\frac{N-3}T)}|E|^\frac{2}{3}|F|^\frac{2}{3}) \\
&\approx |E|^\frac{2(T-Q)}{(3+B)T-NQ}|F|^{1-\frac{T-Q}{(3+B)T-NQ}}.
\end{align*}
Thus, $T_{R_T^e\cap\{|\omega|\approx 1\}}$ is of rwt $(\frac {q_I}{2},q_I)$ for $(\frac{2}{q_I},\frac{1}{q_I})=(\frac{2(T-Q)}{(3+B)T-NQ},\frac{T-Q}{(3+B)T-NQ})$. 

Since $T_{R_T^e\cap\{|\omega|\approx 1\}}$ is of strong type $(\frac 43,4)$ by Theorem \ref{Gressman}, and of rwt $(\frac {q_I}{2},q_I)$, and since $R_T^e$ is \textbf{$\kappa_{\varphi}$}-scale invariant, then by Case 3 of Proposition \ref{scaling prop}, $\scriptT_{R_T^e}$ is of rwt $(p_S,q_S)$, where $(p_S,q_S)$ lies on the scaling line of $\scriptT$ and
$$
\big(\tfrac{2}{p_S},\tfrac 2{q_S}\big)=\big(\tfrac{3-\frac 8{q_I}+\frac 1{q_I}(d_h+1)}{(1-\frac 2{q_I})(d_h+1)},\tfrac{1-\frac 4{q_I}+\frac 1{q_I}(d_h+1)}{(1-\frac 2{q_I}\big)(d_h+1)}\big).
$$
Since $(p_S,q_S)$ lies on the scaling line, then by Lemma \ref{Vertex2Cases} it suffices to verify that $p_S=p_{v_2}$.  Using the identities $Q=2J+r-2$, $B=(r+1)[1+d_h]-(rN+3)$, $T=2N-3$, $d_h=\frac{J+rN}{r+1}$, and our equation for $\frac 1{q_I}$, after some messy arithmetic we find that
$$
\tfrac 1{p_S}=\tfrac{N+1-d_h}{d_h+1}.
$$
Thus, by Lemma \ref{Vertex2Cases}, $p_{v_2}=p_S$, and so $\scriptT_{R_T^e}$ is of rwt $(p_{v_2}, q_{v_2})$. \qed

\subsubsection{Case $(N_{q=2p}^{\neq scal})$} In this case, $(\frac 1{p_{v_2}},\frac 1{q_{v_2}})=(\frac 2N,\frac 1N)$ lies off the scaling line, so we will be using $R_T$ instead of $R_T^e$. We decompose $R_T\cap\{z_1,z_2-z_1^r\geq 0\}$ into regions $\tau_{j,k}=\{x=z_1\sim 2^{-j}$, $y=z_2-z_1^r\sim 2^{-k}\}\cap[-1,1]^2$, where $-\infty\leq k\leq \infty$ and $0\leq j\leq \infty$. This induces the decomposition $\scriptT_{R_T\cap\{z_1,z_2-z_1^r\geq 0\}}=\sum_{j,k}\scriptT_{j,k}$. Combining the result of Lemma \ref{refinementreduced main} with the implications of $|\tau_{j,k}|\lesssim 2^{-j}2^{-k}$,
$$
\scriptT_{j,k}(E,F)\lesssim \min\{2^{-j}2^{-k}|F|,2^{-j}2^{-k}2^{\frac j3[(r+1)(d_h+1)-rN]}2^\frac {kN}3|E|^\frac 23|F|^\frac 23\}.
$$ 
Defining $\scriptT_j:=\sum_k \scriptT_{j,k}$, and interpolating over $k$,
\begin{align*}
\scriptT_{j}(E,F)&\lesssim \sum_k \min\{2^{-j}2^{-k}|F|,2^{-j}2^{-k}2^{\frac j3[(r+1)(d_h+1)-rN]}2^\frac {kN}3|E|^\frac 23|F|^\frac 23\}
\\
&\approx 2^{-\frac jN(r+1)[1-\frac{d_h+1}{N}]}|E|^\frac 2N|F|^{1-\frac 1N}.
\end{align*}
Since $N>d_h+1$,
$$
\scriptT_{R_T\cap\{z_1,z_2-z_1^r\geq 0\}}(E,F)\lesssim \sum_{j=0}^{\infty}2^{-\frac jN(r+1)[1-\frac {d+1}N]}|E|^\frac 2N|F|^{1-\frac 1N}\lesssim |E|^\frac 2N|F|^{1-\frac 1N},
$$
so $\scriptT_{R_T}$ is of rwt $(\frac N2, N)$, which is $(p_{v_2},q_{v_2})$ by Lemma $\ref{Vertex2Cases}$. \qed
 
\subsubsection{Case $(N_{q=2p}^{scal})$} By Lemma \ref{Vertex2Cases}, it suffices to prove the following proposition:
\begin{proposition}
In Case $(N_{q=2p}^{scal})$, $\scriptT_{R_T^e}(E,F)\lesssim |E|^\frac 2N |F|^{1-\frac 1N}$.
\end{proposition}

\begin{proof} Our result will quickly follow from the following lemma:

\begin{lemma}\label{refinement N2} In Case $(N_{q=2p}^{scal})$, $\norm{\mathcal{T}_{R_T^e\cap \{|xy|\approx 1\}}}_{L^{\frac q2,1}\rightarrow L^{q,\infty}}<\infty$ for all $q\in [3,\infty)$.

\end{lemma}

\begin{proof}

Recall the change of coordinates $x=z_1, y=z_2-z_1^r$. It suffices to consider the region with $x,y\geq 0$. We decompose $R_T^e\cap \{|xy|\approx 1\}\cap\{x,y\geq 0\}$ into regions $\tau_{-n,n}:=\{z_1\approx 2^n,z_2-z_1^r\approx 2^{-n}\}$ for $n\in \N$, and denote $\scriptT_n:=\scriptT_{\tau_{-n,n}}$. By Lemma \ref{refinementreduced main}, using $N=d_h+1$,
\begin{align*}
\mathcal{T}_n(E,F)&\lesssim 2^{-\frac n3[(r+1)N-(rN+3)]+\frac n3(N-3)}|E|^\frac 23|F|^\frac 23 =|E|^\frac 23|F|^\frac 23.
\end{align*}


We will use the notation of Notation \ref{refinement notation}, with $b=2$, to refine $E$ and $F$. Since $|\tau_{-n,n}|\lesssim 1$, and since $\scriptT_n(E,F)\lesssim |E|^\frac 23|F|^\frac 23$, by Lemma \ref{refinement epsilon decomp} and Lemma \ref{refinement epsilon decomp 2}, it suffices to prove the following version of Condition \ref{refinement condition Jac}:

\begin{lemma}\label{refinementreducedstrip}
Let $k,l\in \N$, and let $k$, $l$, and $k-l$ be sufficiently large, independent of $E$ and $F$. Then $|F_l|\gtrsim 2^{|k-l|r}\alpha_{E_l}^2\beta_{F_{kl}}$ and $|E_l|\gtrsim 2^{|k-l|r}\alpha_{F_l}^2\beta_{E_{kl}}$, with implicit constants independent of $ l, k, E, F$.
\end{lemma}

\begin{proof}
By symmetry, it will suffice to prove the $F_l$ inequality. Fix $u_0\in F_{kl}$. Define
\begin{align*}
\Omega_1&:=\{t\in\R^2: z(t)\in\tau_{-k,k} \text{ and } u_0-(t_1,t_2+ t_1^r,\varphi(z(t)))=:w(t)\in E_k\cap E_l\}, 
\end{align*}
recalling that $z(t):=(t_1,t_2+t_1^r)$. Then $|\Omega_1|=\mathcal{T}_k\chi_{E_k\cap E_l}(u_0)\geq \beta_{F_{kl}}$. To achieve a lower-dimensional result, we fix one variable. Rewriting $|\Omega_1|$,
$$
\frac{|\Omega_1|}{2^{-k}}=\dashint_{t_2\approx 2^{-k}}\int_{t_1\approx 2^{k}}\chi_{\Omega_1}(t_1,t_2)dt_1dt_2.
$$
Therefore, by H\"older's inequality, there exists a fixed $t_2\approx 2^{-k}$ such that
$$
|\Omega_{1,t_2}|:= |\{t_1:(t_1,t_2)\in \Omega_1\}|=\int_{t_1\approx 2^k}\chi_{\Omega_1}(t_1,t_2)dt_1 \geq \frac{\beta_{F_k}}{2^{-k}}.
$$
For $t_1\in \Omega_{1,t_2}$, define
\begin{align*}
\Omega_{2,t_2}(t_1)&:=\{s\in\R^2: z(s)\in\tau_{-l,l} \text{ and } w(t)+(s_1,s_2+s_1^r,\varphi(z(s)))\in F_l\}.
\end{align*}
Then $|\Omega_{2,t_2}(t_1)|=T_l^*\chi_{F_l}(w(t))\geq \alpha_{E_l}$. Finally, define the following: 
$$
\Omega_{t_2} :=\{(t_1,s_1,s_2)\in \R^3: t_1\in \Omega_{1,t_2}, s=(s_1,s_2)\in \Omega_2(t_1)\};$$
$$
\Psi_{t_2}(t_1,s_2,s_2):=x_0-(t_1,t_2+ t_1^r,\varphi(t))+(s_1,s_2+ s_1^r,\varphi(s_i)).
$$
Since $\Psi_{t_2}$ is a polynomial mapping $\Omega_{t_2}\subset\R^3\rightarrow \R^3$, it is $\scriptO(1)$-to-one off a set of measure zero. Since $\Psi_{t_2}(\Omega_{t_2})\subset F_l$, then $|F_l|\gtrsim \int_{\Omega_{t_2}} |\det D\Psi_{t_2}(t_1,s_1,s_2)|dt_1ds_1ds_2.$ Then
$$
|\det D\Psi_{t_2}(t_1,s_1,s_2)|
= r[t_1^{r-1}-s_1^{r-1}]\partial_{s_2}\varphi(z(s))-[\partial_{t_1}(z(t))-\partial_{s_1}(z(s))].
$$
\begin{claim} For $t\in z^{-1}(\tau_{-k,k})$, and $s\in z^{-1}(\tau_{-l,l})$, with $k,l,k-l$ sufficiently large, independent of $E$, $F$, 
$|t_1^{r-1}\partial_{s_2}\varphi(z(s))|\gg |s_1^{r-1}\partial_{s_2}\varphi(z(s))|\gg |\partial_{s_1}\varphi(z(s))|\gtrsim|\partial_{t_1}\varphi(z(t))|$.
\end{claim}

\begin{claimproof} First, for $k,l,k-l$ sufficiently large, $t_1\approx 2^k\gg 2^l\approx s_1\gg 1$. 

\vspace{.5pc}

\underline{$|t_1^{r-1}\partial_{s_2}\varphi(z(s))|\gg |s_1^{r-1}\partial_{s_2}\varphi(z(s))|$}: Since $t_1\gg s_1$, this is clear.

\vspace{.75pc}

Next, recall that $\varphi(z(x,y))= x^Jy^N+o(y^N)$. If $J=0$, then $\partial_{s_1}\varphi(z(s))=\partial_{t_1}\varphi(z(t))=0$, and $|s_1^{r-1}\partial_{s_2}\varphi(z(s))|>0$, so the rest of the claim holds. Thus, it suffices to consider $J\geq 1$.
\vspace{.5pc}

\underline{$|s_1^{r-1}\partial_{s_2}\varphi(z(s))|\gg |\partial_{s_1}\varphi(z(s))|$}: First, $s_1\gg 1$. Also, $\varphi(z(s))=s_1^Js_2^N+\scriptO(s_2^{N+1})$, with $J\geq 1$, so $|\partial_{s_2}\varphi(z(s))|\approx s_1^J s_2^{N-1}\gg s_1^{J-1} s_2^N\approx |\partial_{s_1}\varphi(z(s))|$, since $s_1 \gg1$ and $s_2\ll 1$. 

\vspace{.75pc}

\underline{$|\partial_{s_1}\varphi(z(s))|\gtrsim|\partial_{t_1}\varphi(z(t))|$}:  Since $\varphi(z(x,y))\approx x^Jy^N$, with $N=d_h+1$, we have the relation $N-1=d_h=\frac{J+Nr}{r+1}$, implying $J=N-r-1$. Since $s_1s_2\approx t_1t_2\approx 1$ and $s_2\gg t_2$, then $|\partial_{s_1}\varphi(z(s))|\approx s_1^{J-1}s_2^N= s_1^{N-r-2}s_2^N\approx s_2^{r+2} \gg t_2^{r+2}\approx t_1^{J-1}t_2^N\approx |\partial_{t_1}\varphi(z(t))|$.
\end{claimproof}

\vspace{.75pc}

Therefore,
\begin{align*}
|\det D\Psi_{t_2}(t_1,s_1,s_2)|&\approx |t_1^{r-1}\partial_{s_2}\varphi(z(s))|\approx t_1^{r-1}s_1^Js_2^{N-1}
\\
&\approx 2^{k(r-1)}2^{l(N-r-1)}2^{-l(N-1)}= 2^{|k-l|r}2^{-k}.
\end{align*}
 Then, since $1=|\tau_{-l,l}|\geq |\Omega_2(t_1)|\geq \alpha_{E_l}$,
\begin{align*}
|F_l|&\gtrsim \int_{\Omega_{t_2}} 2^{|k-l|r}2^{-k}dt_1ds_1ds_2\gtrsim 2^{|k-l|r}2^{-k}\min_{t_1}(|\Omega_{2,t_2}(t_1)|)|\Omega_{1,t_2}|.
\\
&\gtrsim 2^{|k-l|r}2^{-k}\alpha_{E_l}\frac{\beta_{F_{kl}}}{2^{-k}}\geq 2^{|k-l|r}\alpha_{E_l}^2\beta_{F_k}.\end{align*}

A near-identical argument gives $|E_l|\gtrsim 2^{|k-l|r}\alpha_{F_l}^2\beta_{E_k}$, by swapping $E_n$ and $F_n$, and using $
\Psi_{t_2}(t_1,s_2,s_2)=w_0+(t_1,t_2+t_1^r,\varphi(z(t)))-(s_1,s_2+ s_1^r,\varphi(z(s)))$, for $w_0\in E_{kl}$. This completes the proof of Lemma \ref{refinementreducedstrip}.
\end{proof}

Hence, by Lemmas \ref{refinement epsilon decomp} and \ref{refinement epsilon decomp 2}, with $b=2$, the proof of Lemma \ref{refinement N2} is complete.
\end{proof}

Finally, since $|xy|$ is $\frac{\bm{\kappa_\varphi}}{D}$-mixed homogeneous for some $D>0$, and since $R_T^e$ is $\bm{\kappa_\varphi}$-scale-invariant, then by Case 2 of Proposition \ref{scaling prop}, we conclude (using $d_h+1=N$)
$$
\scriptT_{R_T^e}(E,F)\lesssim |E|^\frac 2N |F|^{1-\frac 1N},
$$ 
so $\scriptT_{R_T^e}$ is of rwt $(\frac N2, N)$, which is $(p_{v_2},q_{v_2})$ by Lemma \ref{Vertex2Cases}.
\end{proof}

\end{document}